\documentclass[11pt]{amsart}

\usepackage{url}
\usepackage{amsmath, amscd, amsthm, amssymb}
\usepackage[margin=1in]{geometry}

\def\C{{\mathbf C}}

\def\Z{{\mathbf Z}}
\def\Q{{\mathbf Q}}

\newtheorem{theorem}{Theorem}[subsection]

\newtheorem{lemma}[theorem]{Lemma}

\newtheorem{proposition}[theorem]{Proposition}

\newtheorem{corollary}[theorem]{Corollary}

\newtheorem{claim}[theorem]{Claim}

\theoremstyle{definition}
\newtheorem{definition}[theorem]{Definition}

\newtheorem{example}[theorem]{Example}

\theoremstyle{remark}
\newtheorem{remark}[theorem]{Remark}

\newcommand{\mm}[4]{\left(\begin{smallmatrix} #1 & #2\\ #3 & #4\end{smallmatrix}\right)}

\DeclareMathOperator{\tr}{tr}

\DeclareMathOperator{\SO}{SO}

\DeclareMathOperator{\GSp}{GSp}

\DeclareMathOperator{\GU}{GU}

\DeclareMathOperator{\SL}{SL}
\DeclareMathOperator{\GL}{GL}

\DeclareMathOperator{\diag}{diag}

\begin{document}
\title{Lifting laws and arithmetic invariant theory}
\author{Aaron Pollack}
\address{Department of Mathematics\\ Institute for Advanced Study\\ Princeton, NJ USA}
\email{aaronjp@math.ias.edu}
\thanks{This work partially supported by the NSF through grant DMS-1401858 and the Schmidt Fund at the IAS}

\begin{abstract} In this paper we discuss \emph{lifting laws} which, roughly, are ways of ``lifting" elements of the open orbit of one prehomogeneous vector space to elements of the minimal nonzero orbit of another prehomogeneous vector space.  We prove a handful of these lifting laws, and show how they can be used to help solve certain problems in arithmetic invariant theory.  Of the results contained in this article are twisted versions of certain parametrization theorems of Bhargava.\end{abstract}

\maketitle


\setcounter{tocdepth}{1}
\tableofcontents

\section{Introduction}\label{intro} Suppose $F$ is a field, $G$ a reductive group over $F$, and $V$ an $F$-linear finite dimensional representation of $G$.  The study of the orbits of $G(F)$ on $V(F)$ falls under the heading of \emph{invariant theory}.  When $F = \C$ is the field of complex numbers, the study of such orbits is geometric invariant theory.  When $F = \Q$ is the field of rational numbers, following \cite{bg, bgw}, the consideration of these orbits is called arithmetic invariant theory.  One can also replace the field $F$ by other rings, such as the integers $\Z$, and again consider the orbits of $G(\Z)$ on $V(\Z)$ (if $G$ and $V$ have a structure over the integers.)  This too falls under the heading of arithmetic invariant theory. That the study of the orbits $G(\Z)$ on $V(\Z)$ is a rich and useful subject was made evident in the seminal papers \cite{bhargavaI,bhargavaII, bhargavaIII,bhargavaIV} of Bhargava.

\subsection{Twisted orbit parametrizations}  The main results of this paper are twisted versions of the results of \cite{bhargavaI} and \cite{bhargavaII}.  Recall that in \cite{bhargavaI}, Bhargava found five explicit orbit parametrizations $G(\Z)\backslash V(\Z)$.  The central example of \cite{bhargavaI} is $G(\Z) = \SL_2(\Z) \times \SL_2(\Z) \times \SL_2(\Z)$ acting on $V(\Z) = \Z^2 \otimes \Z^2 \otimes \Z^2$, the triple tensor product of the defining two-dimensional representation of $\SL_2$.  Roughly, it is proved in \emph{loc. cit.} that the orbits parametrize pairs $(S,(I_1,I_2,I_3))$ where $S$ is a quadratic ring, i.e., a commutative ring that is free of rank two as a $\Z$-module, and $I_j$ are fractional $S$-ideal classes whose product $I_1I_2 I_3 = 1$ in the class group of $S$.  

The four other examples considered in \cite{bhargavaI} are as follows:
\begin{enumerate}
\item $G(\Z) = \SL_2(\Z)$, $V(\Z) = Sym^3(\Z^2)$;
\item $G(\Z) = \SL_2(\Z) \times \SL_2(\Z)$, $V(\Z) = \Z^2 \otimes Sym^2(\Z^2)$;
\item $G(\Z) = \SL_2(\Z) \times \SL_4(\Z)$, $V(\Z) = \Z^2 \otimes \wedge^2(\Z^4)$;
\item $G(\Z) = \SL_6(\Z)$, $V(\Z) = \wedge^3(\Z^6)$. \end{enumerate}
Bhargava parametrized the orbits $G(\Z)$ on $V(\Z)$ in the above cases in terms of tuples $(S, M,\ldots)$ where $S$ is again a quadratic ring, $M$ is a certain type of finite $S$-module, the $\ldots$ represents that there is possibly some auxiliary data, and all this data is taken modulo certain equivalences.

The above examples $(G,V)$ of \cite{bhargavaI} are all integral models of \emph{prehomogeneous vector spaces} (PVSs).  Recall that a pair $(G,V)$ of a linear algebraic group $G$ over a field $F$ and a finite dimensional $F$-rational representation $V$ of $G$ is a PVS if there is a (necessarily unique) $G$-invariant Zariski open $V^{open} \subseteq V$ for which $V^{open}(\overline{F})$ consists of one $G(\overline{F})$ orbit, where $\overline{F}$ is an algebraic closure of $F$.  The results of \cite{bhargavaI,bhargavaII, bhargavaIII,bhargavaIV} were preceded by work of Wright-Yukie \cite{wrightYukie} and Kable-Yukie \cite{kableYukie}.  In the paper \cite{wrightYukie}, the authors parametrize orbits $G(F) \backslash V(F)^{open}$ where $(G,V)$ are certain PVSs and $F$ is a field, in terms of finite \'{e}tale extensions of $F$.  The groups $G$ in \emph{loc. cit.} are all split.  In \cite{kableYukie}, the authors consider similar orbit problems, where now the group $G$ is no longer split.  The proofs in \cite{wrightYukie, kableYukie} are Galois-cohomological in nature.

Following \cite{bhargavaI,bhargavaII, bhargavaIII,bhargavaIV}, there has been much activity in the realm of orbit parametrizations and their applications.  For instance, the reader may see \cite{ganSavin, bg,bgw,taniguchi,bhargavaHo,wood,thorne,bhk} for many orbit parametrizations, although this list is not meant to be complete.  While there has been work giving parametrizations $G(\Z)\backslash V(\Z)$ over the integers and more general base rings (e.g., \cite{woodQuartic}) when the group $G$ is split, and parametrizations $G(F)\backslash V(F)$ over fields when $G$ is not split, there has been little work (e.g., some parts of \cite{ganSavin}) on orbit parametrizations $G(\Z)\backslash V(\Z)$ when $G$ is not split.  The main results of this article are of this latter form.

We first give a twisted version of the results of Bhargava from \cite{bhargavaI}.  To setup the result, suppose $A$ is one of the following\footnote{These rings $A$ all have a degree $3$ norm map.  For instance, in the final case, this map is the reduced norm of the central simple algebra.} types of ring:
\begin{enumerate}
\item $\Z$;
\item $\Z \times \Z$;
\item an order in an \'{e}tale cubic extension of $\Q$;
\item $\Z \times Q$, where $Q$ is an order in a quaternion algebra over $\Q$;
\item an order in a central simple algebra of degree $3$ over $\Q$. \end{enumerate}
If $A$ is as above, then there is a cubic polynomial action of $\GL_2(A)$ on $W_{A} := \Z \oplus A \oplus A \oplus \Z$, and this is an integral model for a PVS.  For instance, when $A=\Z$, this is the action of $\GL_2(\Z)$ on its symmetric cube representation, $\mathrm{Sym}^3(\Z^2)$.  When $A = \Z \times \Z \times \Z$, this is the action of $\GL_2(\Z)^3$ on $\Z^2 \otimes \Z^2 \otimes \Z^2$.  These two examples were considered in \cite{bhargavaI}.  In general, one can think of the action of $\GL_2(A)$ on $W_{A} = \Z \oplus A \oplus A \oplus \Z$ as a twisted symmetric cube representation.  See section \ref{sec:new} for more details.  

The group $\GL_2(A)$ has a degree $6$ norm map; define $\SL_2(A)$ to be its kernel.  The first result parametrizes the orbits $W_{A}^{open}\slash \SL_2(A)$ in terms of quadratic rings $S$ and certain $S\otimes A$-modules.  To be more precise, suppose $I \subseteq (S \otimes A)_{\Q}$, where for any $\Z$-module $M$, $M_{\Q} := M \otimes_{\Z} \Q$.  We say that $I$ is an $S \otimes A$ fractional ideal if
\begin{itemize}
\item $SI$ is contained in $I$;
\item $IA$ is contained in $I$;
\item $I = b_1 A + b_2 A$, $b_1, b_2 \in (S \otimes A)_{\Q}$, is free of rank two as an $A$-module. \end{itemize}
We are careful with the right and left actions of $S$ and $A$ above, since $A$ need not be commutative.  We say that $I$ is \emph{oriented} if it comes equipped with an $\SL_2(A)$ orbit of bases $(b_1,b_2)\SL_2(A)$.

\begin{theorem}\label{introThm1} There is an explicit bijection, to be given below, between the orbits $W_{A}^{open}/\SL_2(A)$ and triples of data $(S,I,\beta)/\sim$ where
\begin{itemize}
\item $S$ is an oriented quadratic ring with nonzero discriminant;
\item $I$ is an oriented fractional $S \otimes A$-ideal;
\item $\beta \in (S_{\Q})^\times$;
\item the pair $(I,\beta)$ is \textbf{balanced}.\end{itemize} 
This data is taken modulo the equivalence $(S,I,\beta) \sim (S',I',\beta')$ if there exists an orientation preserving $\varphi: S \stackrel{\sim}{\rightarrow} S'$ and $x \in (S' \otimes A_{\Q})^\times$ such that $I' = x\varphi(I)$ and $\beta' = n_{A\otimes S'_{\Q}/S'_{\Q}}(x)\varphi(\beta)$.  Here $\varphi(I)$ means that both the fractional ideal $I$ and its orientation are pushed forward via $\varphi$.\end{theorem}
In the theorem, that $S$ is oriented means that a generator of the rank one $\Z$-module $S/\Z$ has been chosen.  Furthermore, $n_{A}: A \rightarrow \Z$ is a certain cubic norm on $A$, and $n_{A\otimes S_{\Q}/S_{\Q}}$ is its extension to $A\otimes S_{\Q} \rightarrow S_{\Q}$.  The condition that the pair $(I,\beta)$ is balanced means, very roughly, that the norm $n_{A\otimes S_{\Q}/S_{\Q}}(I)$ is the principal fractional $S$-ideal generated by $\beta$.  See section \ref{sec:statement} for the precise definition.

In the cases when $A$ is split, i.e., when $A = \Z, \Z \times \Z, \Z \times \Z \times \Z, \Z \times M_2(\Z),$ and $M_3(\Z)$, Theorem \ref{introThm1} gives the five orbit parametrizations of Bhargava from \cite{bhargavaI}.  The nonsplit cases are new.  Furthermore, although we have stated Theorem \ref{introThm1} over the base ring $\Z$, below we prove it for a more general class of base rings.  We also mention that Theorem \ref{introThm1} is proved uniformly in the ring $A$, even though there appear to be five cases.  When the ring $A$ is an order in an \'{e}tale cubic extension of $\Q$, Gan-Savin \cite{ganSavin} also considered the above orbit problem, although only classified the orbits over fields.

Theorem \ref{introThm1} is suggested by some of the beautiful arguments of Bhargava-Ho \cite{bhargavaHo}.  Namely, if one considers the orbit problems in Theorem \ref{introThm1}, with the integers $\Z$ replaced by a field, then certain arguments in \cite[section 6]{bhargavaHo}, when specialized to degenerate cases, should yield a parametrization of these orbits.  (The arguments of \cite{bhargavaHo} are quite different from our own, however.) It is an interesting question if the arguments suggested by \cite{bhargavaHo} can be made to work over the integers.  Instead of attempting to use the more geometric techniques of \emph{loc. cit.}, the proof of Theorem \ref{introThm1} blends twisted versions of some of the constructions of \cite{bhargavaI}, together with what we call ``lifting laws", which are discussed below.  To a first approximation, the lifting laws take the place of the solving of equations in \cite{bhargavaI}, although they also have many useful technical consequences.

The second main result is a twisted version of the parametrizations from \cite{bhargavaII}.  To setup this result, suppose $C$ is one of the following\footnote{These are integral models of the associative composition algebras.} rings:
\begin{enumerate}
\item $\Z$;
\item an order in an \'{e}tale quadratic extension of $\Q$;
\item an order in a quaternion algebra over $\Q$. \end{enumerate}
In each case, there is a involution $\tau: C \rightarrow C$, which is the identity in the case $C=\Z$ and nontrivial in the other cases.  If $m \in M_3(C)$, set $m^* = \,^tm^\tau$, the conjugate transpose of $m$.  Define $H_3(C) = \{h \in M_3(C): h^* = h\}$, the Hermitian $3 \times 3$ matrices over $C$. The group $\GL_3(C)$ acts on $H_3(C)$ via $h \cdot m = m^*hm$, and then $\GL_2(\Z) \otimes \GL_3(C)$ acts on $\Z^2 \otimes H_3(C)$.  Here the $\GL_2(\Z)$ action on $\Z^2$ is the usual one, twisted by the determinant. These actions are integral models of PVSs.

Similarly to the above, $\GL_3(C)$ has a degree $6$ norm map, coming from the reduced norm on $M_3(C)$.  Denote by $\SL_3(C)$ the kernel of this degree $6$ norm on $\GL_3(C)$.  There is likewise a notion of an oriented $T \otimes C$-ideal, where $T$ is a cubic ring, i.e., $T$ is commutative ring that is free of rank $3$ as a $\Z$-module.  We have the following result.

\begin{theorem}\label{introThm2} There is an explicit bijection, to be given below, between the orbits $(\Z^2 \otimes H_3(C))^{open}\slash \GL_2(\Z) \times \SL_3(C)$ and triples of data $(T,I,\beta)/\sim$ where
\begin{itemize}
\item $T$ is a cubic ring with nonzero discriminant;
\item $I$ is an oriented fractional $T \otimes C$-ideal;
\item $\beta \in (T_{\Q})^\times$;
\item the pair $(I,\beta)$ is \textbf{balanced}. \end{itemize}
These data are taken modulo the equivalence $\sim$, with $(T,I,\beta) \sim (T',I',\beta')$ if there exists $\varphi: T \stackrel{\sim}{\rightarrow} T'$ and $x \in (T' \otimes C_{\Q})^\times$ such that $I' = x\varphi(I)$, $\beta' = xx^\tau \varphi(\beta)$. Here $\varphi(I)$ means that both the fractional ideal and its orientation are pushed forward via $\varphi$.\end{theorem}

Our remarks regarding Theorem \ref{introThm2} are nearly identical to those following Theorem \ref{introThm1}.  First, the balanced condition very roughly means that $n_{T_{\Q} \otimes C/T_{\Q}}(I)$ is the principal fractional $T$-ideal generated by $\beta$.  See section \ref{sec:statement} for the precise condition.  Second, the split cases of Theorem \ref{introThm2}, namely when $C = \Z, \Z \times \Z$ and $M_2(\Z)$ are the orbit parametrizations of Bhargava from \cite{bhargavaII}.  When $C = R$ or $R \times R$ is split, but the ground ring $\Z$ has been replaced by a more general ring $R$, these parametrizations are due to Wood \cite{wood}.  In fact, when $C = R$ or $R \times R$, Wood considers the orbits of $\GL_n(C)$ on the PVS $R^2 \otimes H_n(C)$ for general $n$.  Next, we again mention that in the text below, we prove Theorem \ref{introThm2} both uniformly in $C$ and over more general base rings than the integers.  The nonsplit cases of Theorem \ref{introThm2} are new.

As with Theorem \ref{introThm1}, Theorem \ref{introThm2} is suggested by some of the arguments of \cite[section 5]{bhargavaHo}, which when specialized to degenerate cases should yield a parametrization, over fields, of the orbits considered in Theorem \ref{introThm2}.  Again, the arguments suggested by \emph{loc. cit.} are quite different from our own.  For the proof of Theorem \ref{introThm2}, we blend twisted versions of constructions in \cite{bhargavaII}, some ideas from \cite{wood}, and lifting laws.

As mentioned, one of the key ingredients in the proof of the results above is what we call a \emph{lifting law}, and in fact this paper is as much about lifting laws as it is about the parametrizations Theorem \ref{introThm1} and Theorem \ref{introThm2} above.  Roughly speaking, a lifting law is a way of ``lifting" an element in an orbit of one PVS to an element in the minimal nonzero orbit of another.  Since the precise definition of a lifting law and its relevance might seem strange to a non-expert, we defer a discussion of lifting laws to section \ref{sec:LLintro} below.  In section \ref{sec:LLintro}, we will define precisely what we mean by a lifting law, and also give the motivating examples for this definition.

\subsection{Other results} While the main results in this paper are the twisted orbit parametrizations Theorem \ref{introThm1} and Theorem \ref{introThm2}, we also give a few other results.  

First, we use lifting laws to solve the orbit parametrization problems considered in Theorems \ref{introThm1} and \ref{introThm2} when the base ring $\Z$ is replaced by a field.  This is Theorem \ref{B1thmField} and Theorem \ref{B2F} below. These parametrizations when the base ring is a field are due to Wright-Yukie \cite{wrightYukie}, Kable-Yukie \cite{kableYukie}, and Taniguchi \cite{taniguchi}, whose methods are all via Galois cohomology. Of course, the more general form of Theorems \ref{introThm1} and \ref{introThm2} proved below covers this case when the base ring is a field.  However, the proofs we give in Theorem \ref{B1thmField} and Theorem \ref{B2F} are not specializations of the arguments used to prove Theorems \ref{introThm1} and \ref{introThm2}, and we feel that the arguments we use are sufficiently interesting to merit their inclusion. In any case, several of the technical results used to prove Theorem \ref{B1thmField} and Theorem \ref{B2F} are needed in the proof of Theorems \ref{introThm1} and \ref{introThm2}.

In this context, we also mention two other results, the work of Kato-Yukie \cite{katoYukie} and the recent National University of Singapore master's thesis of Tan \cite{tan}.  Note that, in Theorem \ref{introThm2} and then again in Theorem \ref{B2F}, the composition algebra $C$ is associative.  In the work \cite{katoYukie}, the authors overcome the difficulty of non-associativity and succeed in parametrizing the orbits on the PVS $H_3(C)^2$ when $C$ is an octonion algebra. Tan, in \cite{tan}, solves a variant of the orbit problem considered in \cite{kableYukie}.  Tan uses formulas of \cite{bhargavaII} to associate a twisted composition algebra to a non-degenerate element $(A,B) \in H_3(K)^2$, where $K$ is a quadratic \'{e}tale extension of the ground field $F$.  The works of Kato-Yukie \cite{katoYukie} and Tan \cite{tan} associate different (and more interesting) data to elements in $H_3(C)^2$ than we associate to such elements in Theorem \ref{B2F}.

While in sections \ref{bI} and \ref{bII} we discuss both lifting laws and their application to arithmetic invariant theory, in sections \ref{nccubes} and \ref{lowerRank} we only consider lifting laws.  In section \ref{nccubes}, we give a lifting law that is a variant of a result of Gan and Savin from \cite{ganSavin}.  See the beginning of section \ref{nccubes} for more details. The lifting laws in sections \ref{bI}, \ref{bII}, and \ref{nccubes} start with an element of the open orbit of some prehomogeneous vector space, and lift it to an element of the minimal orbit of another prehomogeneous vector space.  In section \ref{lowerRank}, we consider a few examples of lifting laws where one starts with elements in one of the non-open orbits of a prehomogeneous vector space.

This paper uses various algebraic devices and group representations, such as cubic norm structures, composition algebras, and the Freudenthal construction, with which the reader might not already be familiar. In section \ref{preliminaries} we discuss all of this background information.  After section \ref{preliminaries}, the remaining sections are almost entirely self-contained, and only rely on section \ref{preliminaries}.  The exception to this statement is that section \ref{nccubes} also uses results from section \ref{bI}.  Thus after section \ref{preliminaries}, the other sections may (almost) be read in any order.

We now briefly re-summarize the layout of the paper.  In section \ref{sec:LLintro}, we define and give an introduction to lifting laws, by way of more familiar examples.  In section \ref{sec:statement}, we give precise statements of our main results on twisted orbit parametrizations, which imply Theorems \ref{introThm1} and \ref{introThm2}.  In section \ref{preliminaries}, we discuss all the algebraic preliminaries needed for the rest of the paper.  In section \ref{bI} (resp. section \ref{bII}) we solve a class of orbit problems that includes those considered in Theorem \ref{introThm1} (resp. Theorem \ref{introThm2}).  In section \ref{nccubes} we give a variant of a result of Gan-Savin from \cite{ganSavin}, while in section \ref{lowerRank} we give some ``lower rank" lifting laws.

\subsection{A note on the proofs} The lifting laws are, in essence, complicated polynomial identities on certain prehomogeneous vector spaces.  To prove them, we often make crucial use of prehomogeneity, in the following way.  Suppose the vector space $V$ is prehomogeneous for the action of the group $G$, and suppose $P_1, P_2$ are polynomial functions $V \rightarrow W$ from $V$ to another $G$-representation $W$.  To prove $P_1(v) = P_2(v)$ for all $v \in V$, it suffices to prove this equality over an algebraic closure $k$ of $F$, and then by Zariski density to prove the identity on the open orbit for $G(k)$ on $V(k)$.  To prove $P_1 = P_2$ on the open orbit, it then suffices to verify $P_1(v_0) = P_2(v_0)$ for some well-chosen $v_0$, and to prove that $P_1$ and $P_2$ are suitably equivariant for the action of $G$.  (From the equivariance one obtains $P_1(v) = P_1(g v_0) = gP_1(v_0) = gP_2(v_0) = P_2(v)$.)  This sort of argument will be used frequently below without detailed explanation.

\subsection{Acknowledgments} We hope it is clear to the reader that this paper owes its existence to the extraordinary papers \cite{bhargavaI,bhargavaII} of Bhargava.  We also single out the beautiful works of Gross-Savin \cite{grossSavin} and Bhargava-Ho \cite{bhargavaHo}, from which the author first learned of various exotic algebraic structures and their relevance for automorphic forms and arithmetic invariant theory.  It is a pleasure to thank Wee Teck Gan, Benedict Gross, Wei Ho, and Gordan Savin for their very helpful comments on and conversations about an earlier version of this manuscript. We also thank the anonymous referee, whose comments on this paper have (hopefully) significantly improved the exposition, and whose remarks regarding section \ref{nccubes} led to a substantially simpler treatment of the lifting law in this section.

\subsection{Notation} Throughout the entire paper, \emph{we only work in characteristic $0$}.  We will always write $F$ for our ground field, a field of characteristic $0$, and $R$ will denote a subring of $F$.  For an $R$-module $M$, $M_F$ denotes $M\otimes_{R}F$.  The letter $C$ will always denote a composition algebra over $F$ or $R$, $J$ will denote a cubic norm structure over $F$ or $R$, and $A$ will denote a cubic norm structure that is also compatibly an associative $F$ or $R$-algebra.  We write $E$ for a finite \'{e}tale extension of $F$, which will always either be a quadratic or cubic extension.  We often write $S$ or $T$ for a particular subring of $E$, that is a finite $R$-module, although sometimes the letters $S$ and $T$ will have different meanings.  The letter $k$ will usually denote an arbitrary field of characteristic $0$, often a large extension field of $F$ such as an algebraic closure.

We write $V_3$ for the defining representation of $\GL_3$ on row or column vectors (depending on the context.)  We write $V_2$ or $W_2$ for the defining representation of $\GL_2$ on row or column vectors, again depending on context.  (In general, we use the notation $W_{*}$ for vector spaces that come equipped with a symplectic form, and $V_{*}$ for spaces that do not.)

\section{Lifting laws: Introduction}\label{sec:LLintro} While this paper is about the results Theorems \ref{introThm1} and \ref{introThm2} above, it is also about lifting laws, which is the main technical ingredient that we use to prove these results.  The purpose of this section is to give an introduction to lifting laws.  We begin by motivating the definition of a lifting law by considering the parametrization of the non-degenerate orbits of $\SL_2(\Z)^3$ on $\Z^2 \otimes \Z^2 \otimes \Z^2$ in \cite{bhargavaI}.  Then, we give the precise definition of a lifting law and mention its connection to the theta correspondence in automorphic forms.  Finally, we give two other examples of how lifting laws connect to arithmetic invariant theory: the parametrization of cubic rings by binary cubic forms and the parametrization of quaternion orders by ternary quadratic forms.  We remark that subsections \ref{subsec:LLb1} and \ref{subsec:LLbcfs} are more or less special cases of the lifting laws discussed below in sections \ref{bI} and \ref{nccubes}, respectively. 

\subsection{Bhargava's cubes and ideal classes in quadratic rings}\label{subsec:LLb1} Consider the parametrization in \cite{bhargavaI} of the $G(\Z) = \SL_2(\Z) \times \SL_2(\Z) \times \SL_2(\Z)$ orbits on the non-degenerate elements of $V(\Z) = \Z^2 \otimes \Z^2 \otimes \Z^2$.  Recall that the parametrization given in \cite{bhargavaI} is in terms of quadratic rings $S = \Z[\tau]$ and a balanced triple $(I_1,I_2,I_3)$ of fractional $S$-ideals.  That the triple $(I_1, I_2, I_3)$ is balanced\footnote{More generally, if $\beta \in (S_{\Q})^\times$, and $(I_1,I_2,I_3)$ are a triple of fractional $S$-ideals, one says that the data $((I_1,I_2,I_3),\beta)$ is balanced if $I_1 I_2 I_3 \subseteq \beta S$ and $N(I_1)N(I_2)N(I_3) = N(\beta)$.} means that the product of the ideals $I_1 I_2 I_3$ is contained in $S$, and the product of the norms of the ideals $N(I_1)N(I_2)N(I_3) = 1$.

Denote by $e_1, f_1$ a basis for the first copy of $\Z^2$ in $V(\Z)$, and similarly $e_2,f_2$, and $e_3,f_3$ bases for the second and third copy of $\Z^2$ in $V(\Z)$.  Suppose one is given a balanced triple $(I_1,I_2,I_3)$ of fractional $S$-ideals.  Choose a $\Z$-basis $\alpha_1, \beta_1$ for $I_1$, $\alpha_2, \beta_2$ for $I_2$, and $\alpha_3, \beta_3$ for $I_3$. Then associated to this data, one considers the pure tensor
\begin{equation}\label{eqn:pureTensor1}(\alpha_1 e_1 + \beta_1 f_1) \otimes (\alpha_2 e_2 + \beta_2 f_2) \otimes (\alpha_3 e_3 + \beta_3 f_3)\end{equation}
in $V(\Z) \otimes S_{\Q}$.  Since the triple $(I_1, I_2, I_3)$ is balanced, all the products $\alpha_1 \alpha_2 \alpha_3$, $\alpha_1 \alpha_2 \beta_3$, $\alpha_1 \beta_2 \alpha_3$ etc are in $S$.  Hence, there exists $v$ and $v_0$ in $V(\Z)$ so that
\begin{equation}\label{pureTensor}(\alpha_1 e_1 + \beta_1 f_1) \otimes (\alpha_2 e_2 + \beta_2 f_2) \otimes (\alpha_3 e_3 + \beta_3 f_3) = \tau v + v_0 \in V(\Z) \otimes_{\Z} S = \tau V(\Z) \oplus V(\Z).\end{equation}
The parametrization of the $G(\Z)$ orbits on $V(\Z)$ in \cite{bhargavaI} proceeds through this construction: one associates to the balanced triple $(I_1, I_2, I_3)$ the $G(\Z)$-orbit $G(\Z) v$.

To have the possibility of parametrizing the orbits of $G(\Z)$ on $V(\Z)$ via triples of balanced ideals, one must be able to invert this construction.  Thus given $v \in V(\Z)$, one asks the following question:  Does there exist a quadratic ring $S=\Z[\tau]$, and $v_0 \in V(\Z)$, so that $\tilde{v} = \tau v + v_0$ is explicitly given as a pure tensor in 
\[V'(\Q) := V(\Q) \otimes_{\Q} S_{\Q} \approx V(\Q) \oplus V(\Q)?\]
Answering this question is the many equations that must be solved on page 235 of \emph{loc. cit.}, and is a crucial step in the argument.

Set $G' = \SL_2(S_\Q) \times \SL_2(S_\Q) \times \SL_2(S_\Q)$, which then acts on $V'$.  Suppose $S$ and $v_0$ are such that $\tilde{v}  = \tau v + v_0$ is a pure tensor in $V'$.  Because $\tilde{v}$ is a pure tensor in $V'$, the stabilizer in $G'$ of the line $S_{\Q}\tilde{v}$ is a parabolic subgroup of $G'$.  Thus, another way of phrasing this question is the following: Given $v \in V(\Q)$, does there exist a quadratic extension $S_{\Q}$ of $\Q$ and a $v_0 \in V_0(\Q) = V(\Q)$ so that the stabilizer of the line $\Q \tilde{v}$ in $V'(\Q)$ is (essentially) a parabolic subgroup of $G'$? More informally, can the element $v \in V(\Q)$ be ``lifted'' to a pure tensor $\tilde{v}$ in $V'(\Q)$?  This discussion motivates the definition of a lifting law, which we now give.

\subsection{Definition of lifting law} Suppose that $F$ is a field, $(G, V)$ is a PVS over $F$, and $v \in V(F)$ is not $0$.  
\begin{definition}\label{defn:LL} A lifting law for $G, V$ and $v$ is the data of another finite-dimensional $F$-linear $G$-representation $V_0$, an element $v_0 \in V_0$, and an action of a (bigger) linear algebraic $F$-group $G'$ on $V' = V \oplus V_0$ so that
\begin{enumerate}
\item the $G$-linear action on $V' = V \oplus V_0$ factors through $G'$,
\item the stabilizer $S$ in $G'$ of the line $F\tilde{v}$, where $\tilde{v} = (v,v_0)$, satisfies $P' \supseteq S \supseteq [P',P']$, where $P'$ is a parabolic subgroup of $G'$. \end{enumerate}
\end{definition}
The second condition roughly means that $\tilde{v} = (v,v_0)$ is in the minimal nonzero orbit of $G'$ on $V'$.  Note that in the definition, the data $(G',V',v_0)$, and in particular the group $G'$, depends on the element $v \in V(F)$.

The sort of lifting problem considered in Definition \ref{defn:LL} is familiar from the theta correspondence, especially in the context of exceptional groups.  In this setting, the above lifting problem arises when computing the restriction of a theta function on one group $G_1'$ to a smaller group $G_1$.  The reader may see, for instance, work of Gross-Savin \cite{grossSavin}, Gan \cite{ganATM,ganSWC}, Lucianovic \cite{lucianovic}, or Weissman \cite{weissman} for some examples of these computations.

As a simple example of a lifting law, suppose $q$ is a non-degenerate quadratic form on the vector space $V$, $G = \SO(V,q)$ is the special orthogonal group preserving this quadratic form, and $v \in V$ satisfies $q(v) = \alpha \neq 0$.  Set $V_0 = F$, the trivial representation of $G$, with quadratic form $q_0(x) = \alpha x^2$.  Put on $V' = V \oplus V_0$ the quadratic form $q' = q \perp (-q_0)$, i.e., $q'((x,x_0))= q(x) - \alpha x_0^2$ for $x \in V$ and $x_0 \in V_0$.  With this quadratic form, $V'$ becomes a representation of the group $G'=\SO(V',q')$, and it is clear that the $G = \SO(V,q)$-action on $V'$ preserves $q'$, and thus factors through the bigger orthogonal group $G'$.  Furthermore, the lifted vector $\tilde{v} = (v,1)$ satisfies $q'(\tilde{v}) = 0$, and thus the stabilizer of the line $F\tilde{v}$ is a parabolic subgroup of $G'$.  Thus the association of $V',q'$ and $\tilde{v}$ to the vector $v \in V$ is a lifting law.

Observe that in the situation of the previous paragraph, the orbits of $G = \SO(V,q)$ on $V\setminus\{0\}$ are parametrized by the value of the quadratic form $q$, i.e., $v_1 = g v_2$ for some $g \in \SO(V,q)$ if and only if $q(v_1) = q(v_2)$.  Furthermore, it is the value of $q$ on $v$ that is the key input into the definition of the quadratic form $q'$ on $V'$ that enables one to lift $v$ to an isotropic vector $\tilde{v}$ in $V'$.  Thus this lifting law for orthogonal groups has something to do with the invariant theory of $\SO(V,q)$ acting on its defining representation $V$.

In fact, lifting laws are connected to many problems in arithmetic invariant theory, and are no doubt recognizable to experts in this context as well.  To make clear the ideas, we now give two more ``classical" examples of lifting laws.

\subsection{Binary cubic forms and cubic rings}\label{subsec:LLbcfs} Recall that classically, binary cubic forms parametrize cubic rings.  This parametrization goes back at least to \cite{deloneFaddeev}.  In this subsection, we briefly give the example of the lifting law that is related to this parametrization.  The example in this subsection is essentially from \cite[Lemma 3.3]{ggj} or \cite[Proposition 6.9]{ganSW}, but see also \cite{hms} and \cite{grossSavin}.

Consider the action of $G=\GL_2(F)$ on the space $V = \mathrm{Sym}^3(F^2)\otimes \det^{-1}$ of binary cubic forms $q(x,y) = ax^3 + bx^2y + cxy^2 + dy^2$.  Here the action is given by $(q \cdot g)(x,y) = \det(g)^{-1} q(g (x,y)^t)$.  As is well-known, the orbits of this action parametrize cubic $F$-algebras $L$, and the parametrization works over the integers and not just fields.  See for instance \cite{ggs} or \cite{bhargavaII}.  The $R$-algebra $T$ associated to the binary cubic form $q$ is specified by saying it has basis $1,\omega,\theta$, with multiplication table
\begin{equation}\label{multTable1} \omega \theta = -ad;\; \omega^2 = -ac + a \theta - b \omega;\; \theta^2 = -bd + c \theta - d \omega. \end{equation}

Conversely, if $T$ is a cubic ring over the integral domain $R$, and $1, \omega,\theta$ is a basis of $T$ as a free $R$-module of rank three, as in \cite{ggs,grossLucianovic} one says that $(1,\omega,\theta)$ is a \emph{good basis} if $\omega \theta \in R$.  One can then show that there exists $a,b,c,d \in R$ so that $\omega, \theta$ obey the above multiplication table.

From the point of view of cubic rings, the multiplication table (\ref{multTable1}) is perhaps a little mysterious.  However, from the point of view of lifting laws, this multiplication can be understood conceptually.  Suppose $T$ is a cubic ring over $R$ and $L = T_F$ is \'{e}tale over $F$.  As mentioned in the introduction, there is a cubic polynomial action of $\GL_2(L)$ on $V' := W_{L} = F \oplus L \oplus L \oplus F$.  Set $G' = \{g \in \GL_2(L): \det(g) \in F^\times\}$, and let $G'$ act on $V'$ by $vg = \det(g)^{-1} v \cdot g$, where $v \cdot g$ is the cubic polynomial action of $\GL_2(L)$ on $W_{L}$.

The following claim is essentially \cite[Lemma 3.3]{ggj} or \cite[Proposition 6.9]{ganSW}.
\begin{claim} Suppose that there exists $a,b,c,d$ so that $\omega,\theta$ obey the multiplication table (\ref{multTable1}).  Then the stabilizer of the line $F(a,-\omega,\theta,d) \subseteq V'$ in $G'$ is a parabolic subgroup of $G'$.  Conversely, suppose there exists $a,d \in F$ so that the stabilizer of the line $F(a,-\omega,\theta,d) \subseteq V'$ in $G'$ is a parabolic subgroup of $G'$.  Set $b = \tr_{L/F}(-\omega)$ and $c = \tr_{L/F}(\theta)$.  Then $\omega,\theta$ obey the multiplication table (\ref{multTable1}). \end{claim}

The claim suggests that good bases for cubic rings may be understood conceptually in terms of lifting laws.  To make this precise, set $V_0 = M_2(F)$ the $2 \times 2$ matrices over $F$, with $G = \GL_2(F)$ action given by right translation: $m \cdot g = mg$ if $m \in M_2(F)$ and $g \in \GL_2(F)$.  Then, $G \subseteq G'$, and as a $G$-module, $V' \simeq V \oplus V_0$.  Furthermore, one can choose this isomorphism of $G$-modules so that
\[\tilde{v}:=(a,-\omega,\theta,d) \mapsto (v,v_0) = (ax^3 + bx^2y+ cxy^2 + dy^3, 1_2) \in V \oplus V_0.\]
Thus, finding a good basis for the cubic ring $T$ is the same as finding $\tilde{v} = (v,v_0) \in V \oplus V_0 \simeq V'$ with the line $F\tilde{v}$ in $V'$ stabilized by a parabolic subgroup of $G'$.  Thus the association of cubic rings with binary cubic forms is closely connected to a lifting law.

The above lifting law is used in an automorphic setting in, for instance, \cite{ganSW} and \cite{ggj}.  Implicit in \cite{ganSavin} is a generalization of the above lifting law, and in section \ref{nccubes} below we give a further extension of this lifting law.

\subsection{Ternary quadratic forms and quaternion orders} As is well-known, ternary quadratic forms $q(x,y,z) = ax^2 + by^2 + cz^2 + dyz + ezx + fxy$, modulo the action of $\GL_3$, parametrize quaternion algebras up to isomorphism.  When the coefficients of the ternary quadratic form $a,b,c,d,e,f$ are in some ring $R$ as opposed to a field, the orbits of $\GL_3(R)$ on them parametrize quaternion orders over $R$; see \cite{grossLucianovic} and \cite{voight}.

As with the case of binary cubic forms and cubic rings, the parametrization of quaternion orders over $R$ by ternary quadratic forms uses the notion of a \emph{good basis}.  More precisely, suppose $q(x,y,z)$ is a ternary quadratic form, as above.  As always, assume $R$ is an integral domain with fraction field $F$ of characteristic $0$.  Let $\GL_3$ act on the right of ternary quadratic forms via $(q \cdot g)(x,y,z) = \det(g)q(\,^tg^{-1} (x,y,z)^t)$ for $g \in \GL_3(R)$.  One associates to $q(x,y,z)$ a quaternion order $Q$ over $R$ with basis $1, v_1, v_2, v_3$, as follows.  First, denote by $e_1, e_2, e_3$ the standard basis of $V_3 = R^3$ (row vectors).  We have a quadratic form $q$ on $V_3$ via $q(xe_1 + ye_2 + ze_3) = q(x,y,z)$, for $x, y, z$ in $R$.  Set $Q = \mathrm{Clif}^{+}(V_3,q)$, the even part of the Clifford algebra associated to $(V_3,q)$, and $v_1 = e_2 e_3, v_2 = e_3 e_1, v_3 = e_1 e_2$.  Then $Q$ is a quaternion order over $R$, and the basis $v_1, v_2, v_3$ satisfies the multiplication table 
\begin{flalign}\label{multTable2} v_2 v_3 &= a v_1^*;&\nonumber \\ v_3 v_1 &= b v_2^*;& \nonumber \\ v_1 v_2 &= c v_3^* &\nonumber \\ v_1v_1^* &= bc; v_2 v_2^* = ca; v_3v_3^* = ab. \end{flalign}
Here $*$ is the involution on $Q$.

Conversely, if $Q$ is a quaternion order over $R$ (assumed to be free of rank $4$ over $R$), and $1,v_1, v_2,v_3$ is a basis of $Q$ as a free $R$-module, one says that $(v_1,v_2,v_3)$ is a \emph{good basis} (see \cite{grossLucianovic}) if there exist $a,b,c \in R$ so that the $v_i$ obey the multiplication table above.  Then every quaternion ring $Q$ over $R$ admits a good basis \cite{grossLucianovic}, and their parametrization through ternary quadratic forms in \emph{loc. cit.} proceeds through the use of good bases.

Good bases for quaternion orders are closely connected to a lifting law.  Namely, set $G = \GL_3$, $V = \mathrm{Sym}^2(V_3)^\vee \otimes \det = H_3(F)$, so that elements $v \in V(F)$ correspond to ternary quadratic forms.  Here we have identified $\mathrm{Sym}^2(V_3)^\vee$ with $H_3(F)$, the Hermitian $3 \times 3$ matrices over the constant associative composition algebra $F$.  The implied action of $\GL_3(F)$ on $H_3(F)$ is $h \cdot g = \det(g) g^{-1}h \,^tg^{-1}$.  The pair $(G,V)$ is a PVS.  If $v \in V^{open}$ corresponds to the quaternion algebra $Q_F$, set $V' = H_3(Q_F)$ and define $G' = \{(\delta, m) \in \GL_1(F) \times \GL_3(Q_F): \delta^{2} = N_6(m)\}$.  Here $N_6(m)$ is the degree six reduced norm on $M_3(Q_F)$, and $G'$ acts on $H_3(Q_F)$ via $h \mapsto \delta m^{-1} h \,^*m^{-1}$. Clearly $V \hookrightarrow V'$, and as a representation of $\GL_3(F)$, $V' = V \oplus V_0$, where $V_0 = M_3(F)$ with $\GL_3(F)$ action given by right translation, $m \mapsto mg$.

Now, suppose $q(x,y,z) = ax^2 + by^2 + cz^2 + dyz + ezx + fxy$ is as above.  As an element of $H_3(F)$, $q(x,y,z)$ corresponds to 
\[v = \left(\begin{array}{ccc} a& \frac{f}{2}& \frac{e}{2} \\ \frac{f}{2} & b &\frac{d}{2} \\ \frac{e}{2} & \frac{d}{2} & c\end{array}\right).\]
Suppose $(v_1, v_2, v_3)$ is the good basis of the quaternion algebra $Q$ corresponding to $q(x,y,z)$, so that the $v_i$ satisfy the multiplication table (\ref{multTable2}).  Define 
\[\tilde{v} = \left(\begin{array}{ccc} a & v_3 & v_2^* \\ v_3^* & b& v_1 \\ v_2 & v_1^* & c \end{array}\right),\]
so that $\tilde{v} \in H_3(Q_F)$.  Then the stabilizer of the line $F\tilde{v}$ in $G'$ is a parabolic subgroup of $G'$.  In terms of the decomposition $V' = V \oplus V_0 \simeq H_3(F) \oplus M_3(F)$, one can choose this isomorphism so that $\tilde{v} = (v,v_0)$, with $v_0 = 1_3 \in M_3(F)$.  Thus good bases in the setting of quaternion orders naturally correspond to lifting laws.  

We remark that an application of this correspondence to automorphic forms is given in \cite{lucianovic}. For the reader familiar with Tits' so-called second construction of cubic norm structures, the above decomposition $H_3(Q) = H_3(F) \oplus M_3(F)$ is a special case of this.  In fact, quite generally, Tits' construction can be regarded as a lifting law.

\section{Statement of twisted orbit parametrizations}\label{sec:statement} In the introduction, we gave rough statements of the main results Theorem \ref{introThm1} and Theorem \ref{introThm2}. The purpose of this section is to give precise statements of the more general versions of these results that we prove below.  The proofs will be given in subsections \ref{subsec:integralB1} and \ref{subsec:integralB2}, once the relevant lifting laws and necessary technical facts have been proved.  In the first part of this section, we discuss Theorem \ref{introThm1}, while in the second part of this section we discuss Theorem \ref{introThm2}.

\subsection{Twisted orbit parametrization yielding quadratic rings}\label{subsec:statement1}  In this subsection we assume $R$ is an integral domain, with fraction field $F$ of characteristic $0$.  We will shortly insist that $R$ satisfies an additional property, but for now we only assume these properties of $R$.

\subsubsection{Algebraic preliminaries}  The first parametrization theorem concerns the nondegenerate orbits of $\GL_2(A)$ on $W_{A} = R \oplus A \oplus A \oplus R$, for certain $R$-algebras $A$.  We begin by describing which $R$-algebras $A$ are considered. We introduce one non-standard piece of terminology.

\begin{definition} Suppose $A$ is an associative $F$-algebra.  We say that a pair $(A, n_{A}:A \rightarrow F)$ is an \textbf{associative cubic norm structure} (ACNS) over $F$ if it is one of the following:
\begin{enumerate}
\item $A= F$, with $n_{A}: F \rightarrow F$ given by $x \mapsto x^3$;
\item $A = F \times F$, with $n_{A}: F \times F \rightarrow F$ given by $(x,y) \mapsto xy^2$;
\item $A = E$, where $E$ is a cubic \'{e}tale extension of $F$, with $n_{A} = n_{E/F}: E \rightarrow F$ the usual degree three norm map;
\item $A = F \times Q$, where $Q$ is a quaternion $F$-algebra, with $n_{A}: F \times Q \rightarrow F$ given by $(x,y) \mapsto xn_{Q/F}(y)$ where $n_{Q/F}: Q \rightarrow F$ is the (degree two) reduced norm map;
\item $A$ a central simple $F$ algebra of degree $3$, with $n_{A}: A \rightarrow F$ the (degree three) reduced norm map. \end{enumerate}
\end{definition}
When $R=F$ is a field, these are the set of $R$-algebras that will be considered in the orbit parametrizations.  

In all of the cases above, there is a quadratic polynomial map $\#: A \rightarrow A$, which we write $y \mapsto y^\#$, that is uniquely determined by the property $y^\# = n_{A}(y) y^{-1}$ when $y$ is invertible.  This map is called the \emph{adjoint}.  To state the explicit orbit parametrization Theorem \ref{twistThm1} below (and for nearly everything else besides), we will need to use this adjoint map.  To see that $y \mapsto n_A(y) y^{-1}$ extends to a polynomial map on all of $A$, one need only verify this when $A$ is split.  In these cases, the adjoint map is given by the following rule in each of the cases above:
\begin{enumerate}
\item $A= F$, with $x^\# = x^2$;
\item $A = F \times F$, with $(x,y)^\# = (y^2,xy)$
\item $A = F \times F \times F$, with $\#$ the map $(\lambda_1,\lambda_2, \lambda_3) \mapsto (\lambda_2 \lambda_3, \lambda_3\lambda_1, \lambda_1\lambda_2)$.
\item $A = F \times Q$, where $Q$ is a quaternion $F$-algebra, with $(x,y)^\# = (n_{Q}(y),xy^*)$.  Here $n_{Q}$ is the degree two reduced norm map on $Q$, and $y^*$ the conjugate of $y \in Q$, so that $yy^* = n_{Q}(y)$ for all $y \in Q$.
\item $A = M_3(F)$, with $x \mapsto x^\#$ the adjoint map on $3 \times 3$ matrices. \end{enumerate}

There is an analogous notion of an associative cubic norm structure over $R$.
\begin{definition} Suppose $A$ is an associative $R$-algebra, that is free of finite rank over $R$.  We say a pair $(A, n_{A}:A \rightarrow R)$ is an \textbf{associative cubic norm structure} over $R$ if $A_{F} = A\otimes_{R}F$, with the induced $n_{A_F}: A_{F} \rightarrow F$, is an ACNS over $F$, and $y^\# \in A \subseteq A_F$ for all $y \in A$.\end{definition}

For the rest of this subsection, $A$ denotes\footnote{The reader unfamiliar with cubic norm structures might take $R=\Z$ and $A$ an order in a degree three field extension of $\Q$ to get a sense for the definitions and constructions below.} an ACNS over $R$. For such an $A$, the algebra $M_2(A)$ has a degree $6$ (reduced) norm map $\det_6: M_2(A) \rightarrow R$, with $\GL_2(A)$ being the preimage of $R^\times$.  One can define a cubic polynomial action of $\GL_2(A)$ on the $R$-module $W_{A} := R \oplus A \oplus A \oplus R$, which will be discussed in detail in subsection \ref{sec:new}.  As mentioned in the introduction, this action makes $W_{A} \otimes_{R} F$ a PVS for the action of $\GL_2(A_F)$.

There is a quartic form $q: W_A \rightarrow R$, whose definition is essentially due to Freudenthal.  The action of $\GL_2(A_F)$ on $W_{A_F}$ preserves this quartic form, up to scaling: $q(v \cdot g) = \det_6(g)^2 q(v)$, where $\det_6: \GL_2(A_F) \rightarrow F^\times$ is the degree $6$ (reduced) norm map.  The open orbit $W_{A_F}^{open}$ for the action of $\GL_2(A_F)$ consists of those $v \in W_{A_F}$ with $q(v) \neq 0$.

\subsubsection{Orbit problem} We now begin the description of the orbit problem.  We will parametrize the orbits of a certain group $G$ on $W_{A}^{open}$ in terms of quadratic rings $S$ over $R$ and certain $S \otimes A$-modules.

To define the group $G$, set $G_1 = \GL_1(R) \times \GL_2(A)$.  Let $(\lambda,g)\in \GL_1(R) \times \GL_2(A)$ act on $v \in W_{A}$ by $v \mapsto \lambda^{-1} (v \cdot g)$.  Define $G \subseteq G_1$ to consist of the pairs $(\lambda,g)$ where $\lambda^3 = \det_6(g)$.  The group $G$ acts on $W_{A}$ by the restriction of the $G_1$ action.

We now briefly discuss quadratic rings $S$ over $R$.  For the rest of this subsection, we will additionally assume that the ring $R$ satisfies $\frac{x^2 + x}{2}$ is in $R$ for all $x \in R$. For such a ring, every quadratic ring $S$ over $R$, i.e., every commutative $R$-algebra that is free of rank two over $R$, is of the form $S = S_{D} = R[y]/(y^2 - Dy + \frac{D^2-D}{4})$, where $D$ is congruent to a square modulo $4 R$.  If $S$ is a quadratic ring over $R$, and $(1, \tau)$ is a basis of $S$, we say that $(1,\tau)$ is a \emph{good basis} of $S$ if $\tau$ satisfies a quadratic equation $\tau^2 - D\tau + \frac{D^2-D}{4}$ for some $D \in R$ (necessarily a square modulo $4R$).  We write $E := S_{D} \otimes_{R} F = F[y]/(y^2 -D)$, and denote by $\omega$ the image of $y$ in $E$.

\subsubsection{Balanced modules} As mentioned, the orbits of $G$ on $W_{A}$ will be parametrized in terms of certain $S \otimes A$ modules.  To describe which modules show up, we need to make a few definitions.  We consider a class of $S_D \otimes_R A = S \otimes_R A$ modules $I \subseteq A_E$ as follows. 
\begin{definition} If $I \subseteq E \otimes A$, we say that $I$ is an $S\otimes A$-\textbf{fractional ideal} if
\begin{itemize}
\item $I$ is closed under left multiplication by $S$ and right multiplication by $A$;
\item $I$ is free of rank $2$ as an $A$-module;
\item $I_F = E \otimes A$ is free of rank one as an $A_{E}$-module. \end{itemize}
\end{definition}

If a fractional $S\otimes A$-ideal $I$ comes equipped with an $A$-basis, $I = b_1 A \oplus b_2 A$, we may define its norm as follows.
\begin{definition} Suppose $(\tau,1)$ is a good basis of $S$, and $I = b_1A \oplus b_2A$, with $b_1, b_2 \in E\otimes A$.  Then, there is a unique $g \in M_2(A_F)$ so that $(b_1,b_2) = (\tau \otimes 1, 1\otimes 1)g$.  Define the \textbf{norm} of the data $(I,(b_1,b_2))$ as $N(I;\tau, (b_1,b_1)) = \det_6(g)$ where $g$ is as above.  Concretely, if $b_1 = \tau \otimes m_{11} + 1 \otimes m_{21}$ and $b_2 = \tau\otimes m_{12} + 1\otimes m_{22}$, with $m_{ij} \in A_F$, then the norm $N(I;\tau,(b_1,b_2))$ is defined to be $\det_6\left(\mm{m_{11}}{m_{12}}{m_{21}}{m_{22}}\right)$.  The norm of $I$ depends on the data $\tau, b_1, b_2$. \end{definition} 

The orbits of $G$ on $W_{A}^{open}$ will be parametrized by what are called \emph{balanced} fractional $S \otimes A$-ideals.  To define the notion of balanced, and to then state the explicit bijection Theorem \ref{twistThm1} below, we need to define a certain cubic polynomial map $!:A^2 \rightarrow W_{A}$.  Thus suppose $(x,y) \in A^2$.  Associated to the pair $(x,y)$ is the element 
\begin{equation}\label{shriekFormula} (x,y)^{!} := (n_{A}(x), x^{\#}y, y^\#x, n_{A}(y)) \in W_{A}.\end{equation}
Note that if $a \in A$, then $(ax,ay)^{!} = n_{A}(a)(x,y)^{!} \in W_{A}$.  The fact that $W_{A}$ receives a cubic polynomial map from $A^2$ with this property is the sense that in which $W_A$ can be thought of as twisted symmetric cube representation of $\GL_2(A)$.  

One always has $q((x,y)^{!}) = 0$, and in fact, $(x,y)^{!}$ is in the minimal nonzero orbit for the action $\GL_2(A_F)$ on $W_{A_F}$. Base-changing the map (\ref{shriekFormula}) from $A$ to $A_{E}$, one obtains a map $!: A_E^2 \rightarrow W_{A} \otimes E$ given by the same formula as in (\ref{shriekFormula}).  

We can now define the notion of a balanced fractional $S\otimes A$-ideal.
\begin{definition} Suppose $\beta \in E^\times$, and $I$ is an $S\otimes A$-fractional ideal with $A$-basis $b_1, b_2$, and $\tau \in S$ is such that $(\tau,1)$ is a good basis of $S$.  We say the data $(I,(b_1,b_2),\tau,\beta)$ is \textbf{balanced} if 
\begin{equation}\label{eqn:betaShriek}\beta^{-1}(x,y)^{!} = \beta^{-1}(n_{A_E}(x),x^\#y,y^\# x, n_{A_E}(y)) \in W_{A} \otimes S \subseteq W_{A} \otimes E\end{equation}
for all $x, y \in I$ and if $N(I;\tau,(b_1,b_2)) = N_{E/F}(\beta)$. Here $n_{A_E}$ denotes the $E$-linear extension of $n_A: A \rightarrow F$ to $n_{A_E}:A_E \rightarrow E$.\end{definition}
The condition $\beta^{-1}(x,y)^{!} \in W_{A} \otimes S$ for all $x, y \in I$ implies that $n_{A_E}(x) \in \beta S$ for all $x \in I$.  When $A = R \times R \times R$, $R \times M_2(R)$, and $M_3(R)$, the condition $n_{A_E}(x) \in \beta S$ for all $x \in I$ is in fact equivalent to the condition (\ref{eqn:betaShriek}).

It is a fact that $\beta^{-1}(x,y)^{!} \in W_{A} \otimes S$ for all $x, y \in I$ if and only if $\beta^{-1} (b_1,b_2)^{!} \in W_{A} \otimes S$.  Suppose that $I$ is a fractional $S \otimes A$-ideal, with $A$-basis $b = (b_1,b_2)$, and $\beta \in E^\times$. Set
\[X(I,b,\beta) := \beta^{-1} b^{!} \in W_{A_E} = W_A \otimes_{R} E.\]
The element $X(I,b,\beta)$ is the analogue of the pure tensor in (\ref{eqn:pureTensor1}).

Finally, we define an equivalence relation on data $(I,b,\beta)$.
\begin{definition} The triple $(I,b,\beta)$ is said to be \emph{equivalent} to the triple $(I',b',\beta')$ if there exists $x \in A_E^\times$ so that $I' = xI$, $b' = xb$, and $\beta' = n_{A_E}(x) \beta$.  We write $[(I,b,\beta)]$ for the associated equivalence class.\end{definition}
The non-degenerate orbits of $G$ on $W_{A}$ will be parametrized in terms of equivalence classes of balanced data $(I,b,\beta)$.

\subsubsection{The explicit bijection} Suppose that $\tau,1$ is a good basis for $S$ and that $(I,b,\beta)$ is a balanced fractional $S \otimes A$-ideal.  Then there are $v, v' \in W_{A}$ so that $\beta^{-1}b^{!} = X(I,b,\beta) = \tau v + v' \in W_{A}\otimes S$.  One thus obtains a map from balanced data $(I,b,\beta)$ to $v \in W_{A}$:
\[(I,b,\beta) \mapsto X(I,b,\beta) = \tau v + v' \mapsto v \in W_A.\]
Note the similarity of this map with those in \cite{bhargavaI} and the construction of subsection \ref{subsec:LLb1}.  It will be checked below that the $v$ so obtained is in $W_{A}^{open}$. We can now state the parametrization theorem.
\begin{theorem}\label{twistThm1} The map $(I,(b_1,b_2),\tau,\beta) \mapsto v$ induces a bijection between $v \in W_{A}$ with $q(v) = D \neq 0$ and balanced data $(I,(b_1,b_2),\tau,\beta)$ up to equivalence, with $\tau^2 - D\tau + \frac{D^2-D}{4} = 0$. This bijection is equivariant for the action of the group $G$, with $(\lambda,g) \in G$ acting on the data $((b_1,b_2),\omega,\beta)$ by
\[ [((b_1,b_2),\omega,\beta)] \mapsto [((b_1,b_2)g, \lambda \omega, \beta)].\] 
Here $[((b_1,b_2),\tau,\beta)]$ means that the data is taken up to equivalence. 
\end{theorem}

Note that Theorem \ref{twistThm1} implies Theorem \ref{introThm1} of the introduction, by restricting to the action of $\SL_2(A) \subseteq G$.  Below we will also describe explicitly the inverse map $v \mapsto [(I,(b_1,b_2),\tau,\beta)]$.  This will involve producing the element $X = \tau v + v'$ from the element $v \in W_{A}$, which will be the lifting law in this context.

\subsection{Twisted orbit parametrization yielding cubic rings} In this subsection we assume $R$ is an integral domain, with fraction field $F$ of characteristic $0$. (We no longer assume $x^2 + x \in 2R$ for all $x \in R$.)  We will state a precise and more general version of Theorem \ref{introThm2} of the introduction.  This theorem will parametrize the non-degenerate orbits of (essentially) $\GL_2(R) \times \GL_3(C)$ on $H_3(C)^2 = R^2 \otimes H_3(C)$, for an associative composition ring $C$.  The orbits will be parametrized in terms cubic $R$-algebras $T$ and certain $T \otimes C$-modules.  We discuss each of these pieces in turn: composition rings, the module $H_3(C)$, cubic $R$-algebras, and the $T \otimes C$-modules that arise.

\subsubsection{Composition rings} We begin by discussing associative composition algebras.

\begin{definition} Suppose $C$ is an associative $F$-algebra.  We say that $(C, n_C: C\rightarrow F)$ is an \textbf{associative composition algebra} (ACA) if $C$ is one of the following:
\begin{enumerate}
\item $C = F$, with $n_C: F \rightarrow F$ given by $x \mapsto x^2$;
\item $C = E$, with $E$ a quadratic \'{e}tale extension of $F$, and $n_C = n_{E/F}: E \rightarrow F$ the (degree two) norm map;
\item $C$ is a quaternion $F$-algebra, with $n_{C}: C \rightarrow F$ the (degree two) reduced norm map. \end{enumerate}
If $C$ is such an $F$-algebra, then $C$ has the standard involution $\sigma: C \rightarrow C$ that satisfies $n_C(x) = xx^\sigma$.  If $C$ is an associative $R$-algebra, that is free of finite rank as an $R$-module, then we say that $(C, n_C: C \rightarrow R)$ is an ACA over $R$ if $(C_F, n_{C_F}: C_F \rightarrow F)$ is an ACA over $F$ and $C$ is stable under the involution $\sigma$ on $C_F$.\end{definition}

Throughout the rest of this subsection, $C$ will denote\footnote{The reader may wish to take $R = \Z$ and $C$ an order in a quadratic extension of $\Q$ to get a sense for the definitions and constructions below.} an ACA over $R$.  We set $H_3(C) = \{h \in M_3(C): h^* = h\}$ the Hermitian $3 \times 3$ matrices with coefficients in $C$.  Here $*: M_3(C) \rightarrow M_3(C)$ is $x \mapsto \,^tx^\sigma$, the transpose conjugate.  In coordinates, $H_3(C)$ consists of the elements
\[ \left(\begin{array}{ccc} c_1 & a_3 & a_2^\sigma \\ a_3^\sigma & c_2 & a_1 \\ a_2 & a_1^\sigma & c_3 \end{array}\right)\]
with $c_1, c_2, c_3 \in R$ and $a_1, a_2, a_3 \in C$. There is an action of $\GL_3(C)$ on $H_3(C)$ via $h \cdot m = m^*hm$, for $h \in H_3(C)$ and $m \in \GL_3(C)$.  From now on, we will often abuse notation and write $*$, instead of $\sigma$, for the involution on $C$.

\subsubsection{Binary cubic forms and cubic rings} We now give various facts and reminders that concern the relationship between binary cubic forms and cubic rings. See, for instance, \cite{ggs} (and also subsubsection \ref{ssc:bcfs}) for more details.

Suppose we have the binary cubic form $f(x,y) = ax^3 + bx^2y+ cxy^2 + dy^3$ with $a,b,c,d$ in the base ring $R$.  The based cubic $R$-algebra $T$ determined by $f$ is a free $R$-module of rank $3$, with basis $1, \omega,\theta$.  The multiplication table in $T$ is
\begin{itemize}
\item $\omega \theta = -ad$;
\item $\omega^2 = -ac + a \theta - b \omega$;
\item $\theta^2 = -bd + c \theta - d \omega$. \end{itemize}

Denote $L = T \otimes_R F$, and set $\omega_0 = \omega + \frac{b}{3}$ and $\theta_0 = \theta - \frac{c}{3}$, so that $\tr_{L/F}(\omega_0) = \tr_{L/F}(\theta_0) = 0$.  The basis $(1,\omega,\theta)$ of $T$ is called a \emph{good basis}.

\subsubsection{Orbit problem} For the rest of this subsection, $R, F$ and $C$ will be as above.  The letter $T$ will denote a cubic ring over $R$, and $L = T \otimes_R F$.

Set $G_1 = \GL_1(R) \times \GL_2(R) \times \GL_3(C)$, which acts on the right of $V := V_2(R) \otimes H_3(C)$.  Here $V_2(R) = R^2$ is the row vectors with coefficients in $R$ and the action of $\GL_2$ action is the usual right action of matrices on row vectors.  The group $\GL_3(C)$ acts on the right of $H_3(C)$ by $h \mapsto m^*hm$ for $h \in H_3(C)$ and $m \in \GL_3(C)$. Finally, $\lambda \in \GL_1(R)$ acts on $V$ by scaling by $\lambda$. We denote by $G \subseteq G_1$ the group of triples $(\lambda,g,m) \in \GL_1(R) \times \GL_2(R) \times \GL_3(C)$ with $\lambda = \det(g)^{-1}$ and $\det(g)^2 = N_6(m)$, where $N_6(m)$ is the degree $6$ reduced norm on $M_3(C)$.  Then $G$ also acts on $V$ by restriction of the action of $G_1$.  One has that the pair $(G(F),V(F))$ is a PVS. We will consider the orbits of $G$ on $V^{open}$.

\subsubsection{Balanced modules} Let $C_{L} = C \otimes_{R} L$.  We consider $T \otimes C$ submodules $I \subseteq L \otimes C$.  We make a few definitions.
\begin{definition} If $I \subseteq L \otimes C$.  We say $I$ is a $T\otimes C$-\textbf{fractional ideal} if
\begin{itemize}
\item $I$ is closed under left multiplication by $T$ and right multiplication by $C$;
\item $I$ is free of rank $3$ as a $C$-module;
\item $I \otimes_{R}F$ is free of rank one as an $L \otimes C$-module. \end{itemize} \end{definition}

\begin{definition} Suppose $I$ is a $T \otimes C$ fractional ideal, and $b_1, b_2, b_3$ is an ordered basis of $I$ i.e., that $I = b_1C + b_2C + b_3 C$ inside $L \otimes C$, where $b_1, b_2, b_3$ are in $L \otimes C$.  Then $(b_1,b_2,b_3) = (1\otimes 1,\omega\otimes 1,\theta\otimes 1)g$ for some unique $g \in M_3(C_F)$.  We define the norm of the ideal $I$, $N(I;(b_1,b_2,b_3);(1,\omega,\theta))$, with respect to the data $(b_1,b_2, b_3)$ and $(1,\omega,\theta)$ to be $N_6(g)$.  Of course, this norm depends on the choice of the good basis $(1,\omega,\theta)$ and the choice of ordered basis $(b_1,b_2,b_3)$. \end{definition}

We now define when the based fractional ideal $I,(b_1,b_2,b_3)$ is balanced.
\begin{definition} Suppose $\beta \in L^\times$, and the good basis $(1,\omega,\theta)$ of $T$ is fixed.  The data $((I,(b_1,b_2,b_3)),\beta)$ is said to be \textbf{balanced} if $y^\sigma x \in \beta C \otimes_R T$ for all $x, y$ in $I$ and $N((I,(b_1,b_2,b_3))) = N_{L/F}(\beta)$. Here $\sigma: L \otimes C \rightarrow L \otimes C$ is the $L$-linear extension of $\sigma: C \rightarrow C$. \end{definition}

Set $b = (b_1,b_2,b_3) \in (C_L^3)$, a row vector.  If $I = b_1 C + b_2 C + b_3 C$ and $\beta \in L^\times$, we set $X_{I,b,\beta} = \beta^{-1} b^* b \in H_3(C) \otimes T = H_3(C_T)$.  In coordinates,
\[X_{I,b,\beta} = \beta^{-1} b^* b = \beta^{-1} \left(\begin{array}{ccc} n_{C}(b_1) & b_1^* b_2 & b_1^* b_3 \\ b_2^* b_1 & n_{C}(b_2) & b_2^*b_3 \\ b_3^* b_1 & b_3^* b_2 & n_{C}(b_3) \end{array}\right).\] 
One sees easily that $X_{I,b,\beta}$ is in $H_3(C) \otimes T \subseteq H_3(C) \otimes L$ if and only if $\beta^{-1} y^\sigma x \in C\otimes_{R} T$ for all $x,y \in I$.  The element $X_{I,b,\beta}$ is, by construction, a ``pure tensor'' in $H_3(C) \otimes L$, just like the element $X(I,b,\beta)$ was in subsection \ref{subsec:statement1}.

We now define an equivalence relation on based fractional $T \otimes C$ ideals.  Suppose $I =  b_1C +  b_2C + b_3C \subseteq C_L$ is a $C_T = C \otimes_{R} T$-fractional ideal, and $\beta \in L^\times$ is a unit.  We fix the good basis $(1,\omega,\theta)$ of $T$.  Now say that $(b_1,b_2,b_3, \beta)$ is equivalent to $(b_1',b_2',b_3',\beta')$ if there exists $x \in C_{L}^\times$ so that
\[(b_1',b_2',b_3',\beta') = (xb_1, xb_2,xb_3, n_{C}(x)\beta).\]
It is clear that if $(b_1,b_2,b_3, \beta)$ is equivalent to $(b_1',b_2',b_3',\beta')$, then $X_{I,b,\beta} = X_{I',b',\beta'}$.

If the data $(I,(b_1,b_2,b_3),\beta)$ is balanced, then since $X_{I,b,\beta}$ is in $H_3(C_{T}) = H_3(C) \otimes_{R} T$, there are $A,B, D$ in $H_3(C)$ so that 
\begin{equation}\label{getAB} X_{b,\beta} = -A\theta + B\omega + D.\end{equation}
Associated to a based-balanced $T \otimes C$-module $(I,(b_1,b_2,b_3),\beta)$, and the good basis $(1,\omega, \theta)$ of $T$, one thus obtains a pair $(A,B) \in H_3(C)^2$ by equality (\ref{getAB}). Again, it is clear that the association $(I,(b_1,b_2,b_3),\beta) \mapsto (A,B)$ descends to the level of equivalence classes.  It will be checked below that if $L = T\otimes_{R} F$ is \'{e}tale, then this pair $(A,B)$ is in $V^{open}$.

We can now state the orbit parametrization theorem.
\begin{theorem}\label{twistThm2} Suppose that $L = T_F$ is \'{e}tale.  Then the association
\[(T,(1,\omega,\theta),(I,(b_1,b_2,b_3)),\beta) \mapsto (A,B)\]
induces a bijection between based-balanced fractional $T \otimes C$-ideals and pairs $(A,B)$ in $V^{open}$.  These bijections are equivariant for the action of the group $G$: If $h = (\lambda, m,g) = (\det(g)^{-1},m,g)$ with $N_6(m) = \det(g)^2$ is in $G$, then $h$ acts on the data via
\[ [((\omega_0,\theta_0),(b_1,b_2,b_3),\beta)]\mapsto [((\omega_0,\theta_0)g,(b_1,b_2,b_3)m,\beta)].\]
Here, $[((\omega_0,\theta_0),(b_1,b_2,b_3),\beta)]$ means that the data $(b_1,b_2,b_3,\beta)$ is taken up to equivalence.  \end{theorem}

Note the similarity of the map above with those of \cite{bhargavaII}.  Below, we will also explicitly give the inverse map.  As with Theorem \ref{twistThm1}, constructing the inverse will involve producing the good basis $(1,\omega,\theta)$ and the element $X= -A\theta + B\omega + D$ out of the pair $(A,B)$, which will involve the lifting law in this context.

\begin{remark} In the context of Theorem \ref{twistThm1}, constructing $X = \tau v + v'$ from $v \in W_{A}$, and showing that $X$ is abstractly ``rank one'' (this is the general stand-in for being a pure tensor), is easy.  This is done in Theorem \ref{b1J}.  Similarly, in the context of Theorem \ref{twistThm2}, constructing $X = -A\theta + B\omega + D$ from the pair $(A,B)$ and proving that $X$ is rank one is again easy.  This is done in the (first part) of Theorem \ref{B2Jlift}.  Both Theorem \ref{b1J} and \ref{B2Jlift} occur in the context of arbitrary cubic norm structures.  Where we use the hypotheses that $A$ is an associative cubic norm structure, and $C$ is an associative composition algebra, is to write the rank one elements $X$ as ``pure tensors'' in a precise way.  Writing $X$ as a pure tensor is what is needed for the orbit parametrizations.  These \emph{refined} lifting laws, which write the rank one elements $X$ as pure tensors, are Theorem \ref{b1A} and Theorem \ref{B2liftC} below.\end{remark}

\section{Preliminaries}\label{preliminaries}  In this section, we discuss the algebraic preliminaries needed for the various other sections, such as composition algebras, cubic norm structures, certain reductive groups and representations related to these objects, and facts pertaining to the association of binary cubic forms with cubic rings.  To ease the burden on the reader, all the necessary preliminaries are given in this section, even though some of this is a repetition of the algebraic preliminaries given in prior sections.  

In subsection \ref{subsec:23rings} we discuss quadratic and cubic rings over a base ring $R$, and composition algebras.  (As always in this paper, $R$ is an integral domain with fraction field $F$ of characteristic $0$.)  In subsection \ref{subsec:CNS} we discuss cubic norm structures over $F$ and $R$, and in subsection \ref{subsec:Freud} we discuss the Freudenthal construction $W_J$ for general cubic norm structures $J$.  Finally, in subsection \ref{sec:new}, we discuss some facts and constructions related to the Freudenthal construction that apply when $J = A$ is an associative cubic norm structure.  Subsections \ref{subsec:23rings}, \ref{subsec:CNS} and \ref{subsec:Freud} consist entirely of a review of well-known material.  Subsection \ref{sec:new} contains some algebraic constructions and facts related to these objects that may be new, or at least, for which we were not able to find a reference\footnote{We have recently learned of a forthcoming work by Wei Ho and Matthew Satriano that will consider a generalization of the space $V_A$ constructed in subsection \ref{sec:new}.}.

The reader only interested in the results of section \ref{bII} can skip subsubsection \ref{sssec:QR}, subsection \ref{subsec:Freud} and subsection \ref{sec:new}.  The reader only interested in the results of section \ref{bI} can skip subsubsection \ref{ssc:bcfs}.

\subsection{Quadratic and cubic rings}\label{subsec:23rings} In this subsection we discuss quadratic and cubic rings over a base ring $R$, and composition algebras.

\subsubsection{Quadratic rings over $R$}\label{sssec:QR} We now discuss quadratic rings over base rings $R$, satisfying some property.
In this paragraph, we assume that $R$
\begin{itemize}
\item is an integral domain, with fraction field $F$ of characteristic $0$;
\item $\frac{x^2 +x}{2} \in R$ for all $x \in R$. \end{itemize}

\begin{definition} A commutative $R$-algebra $S$ is said to be a quadratic ring over $R$ if $S$ is a free of rank $2$ as an $R$-module. \end{definition}
We parametrize quadratic rings over $R$: Define
\[\mathrm{QuadRing}(R) :=\{D \in R: D \equiv \Box \text{ modulo } 4R\}/\left((R^\times)^2\right).\]
If $D \in \mathrm{QuadRing}(R)$, then $\frac{D^2-D}{4} \in R$ and set $S_{D} = R[x]/(x^2 - Dx + \frac{D^2-D}{4})$.  The $S_{D}$ are exactly the quadratic rings over $R$, up to isomorphism.

To see this, suppose $S = R[x]/(x^2 + ax +b)$ is some quadratic ring over $R$.  Set $D = a^2 - 4b$.  Then $D \in \mathrm{QuadRing}(R)$.  Set $y = x - 2b + \frac{a^2 +a}{2}$.  Then $y^2 - Dy + \frac{D^2-D}{4} = 0$, and thus $S \simeq S_{D}$, so every quadratic ring appears.  Furthermore, one can check $S_{D} \hookrightarrow S_{D'}$ if and only if $D = \alpha^2 D'$ for some $\alpha \in R$, and thus $S_{D} \cong S_{D'}$ implies $D$ and $D'$ define the same class in $\mathrm{QuadRing}(R)$.

\subsubsection{Composition algebras} We now discuss composition algebras.  A composition $F$-algebra $C$ is a not-necessarily associative unital $F$-algebra $C$, for which there exists a non-degenerate quadratic form $n_C: C \rightarrow F$ that satisfies the identity $n_C(xy) = n_C(x)n_C(y)$ for all $x,y \in C$.  If such an $n_C$ exists, it is unique, and is called the norm on the composition algebra.  Composition algebras have an involution, $x \mapsto x^*$ that is order-reversing and satisfies $x + x^* = \tr_C(x) \in F$ and $n_C(x) = xx^*$.  In a composition algebra, one always has the identity $\tr_{C}(a_1 (a_2 a_3)) = \tr_{C}((a_1 a_2) a_3)$, for $a_i \in C$, even though the multiplication may not be associative.

Composition $F$-algebras always have dimensions $1,2,4,$ or $8$, and are, respectively, $F$, a quadratic \'{e}tale extension of $F$, a quaternion algebra, or an octonion algebra.  The composition algebras $C$ for which the multiplication is associative are precisely the ones of dimensions $1,2$ or $4$.  For more on composition algebras see the books \cite{springerVeldkamp} or \cite{kmrt}.

By a composition ring $C$, we mean a unital $R$-algebra $C$ for which $C_F :=C\otimes_{R}F$ is a composition $F$-algebra, such that $C$ is closed under the involution $*$, and for which the norm and trace form $n_C$, $\tr_C$ on $C_F$ are $R$-valued on $C$.  We will also call such rings $C$ composition $R$-algebras.

\subsubsection{Binary cubic forms and cubic rings}\label{ssc:bcfs} We now give various facts and reminders that concern the relationship between binary cubic forms and cubic rings. See, for instance, \cite{ggs} for more details.

Suppose $f(x,y)$ is a binary cubic form over $R$, i.e., $f(x,y) = ax^3 + bx^2y+ cxy^2 + dy^3$ with $a,b,c,d$ in the base ring $R$.  Associated to $f$ is a based cubic $R$-algebra $T$.  The $R$-algebra $T$ determined by $f$ is a free $R$-module of rank $3$, with basis $1, \omega,\theta$, that satisfies the multiplication table
\begin{itemize}
\item $\omega \theta = -ad$;
\item $\omega^2 = -ac + a \theta - b \omega$;
\item $\theta^2 = -bd + c \theta - d \omega$. \end{itemize}
Set $L = T \otimes_{R} F$, a cubic ring over $F$.  For readers wanting to understand this multiplication table from the point of view of the Freudenthal construction discussed below, we note that the above algebraic relations in $T$ are determined by the fact that $(a,-\omega, \theta, d)$ is rank one in $W_L$, and $\tr(-\omega) = b$, $\tr(\theta) = c$.  This is essentially the lifting law in subsection \ref{subsec:LLbcfs}.

Set $\omega_0 = \omega + \frac{b}{3}$ and $\theta_0 = \theta - \frac{c}{3}$, so that $\tr_{L/F}(\omega_0) = \tr_{L/F}(\theta_0) = 0$.  We record the multiplication table for $(1, \omega_0,\theta_0)$ in $L$.  We have
\begin{itemize}
\item $\omega_0 \theta_0 = \frac{b}{3}\theta_0 - \frac{c}{3}\omega_0 + \left(\frac{bc}{9}-ad\right)$;
\item $\omega_0^2 = a \theta_0 - \frac{b}{3}\omega_0 + 2\left(\left(\frac{b}{3}\right)^2- a \frac{c}{3}\right)$;
\item $\theta_0^2 = \frac{c}{3}\theta_0 - d\omega_0 + 2\left(\left(\frac{c}{3}\right)^2-d\frac{b}{3}\right).$ \end{itemize}

The $F$-algebra $L$ possesses a unique quadratic polynomial map $\#: L \rightarrow L$, that satisfies $y^\# = n_{L/F}(y)y^{-1}$ for $y$ in $L$ invertible.  This fact was mentioned in section \ref{subsec:statement1} and will be discussed in much more generality in section \ref{subsec:CNS}.  For later use, we record that $\omega^\# = a\theta$ and $\theta^\#= -d\omega$.  Associated to the quadratic polynomial map $\#$ is a symmetric bilinear map $\times: L \otimes L \rightarrow L$ defined by $x \times y = (x+y)^\#-x^\#-y^\#$.  Again, this map will be put in much more context in subsection \ref{subsec:CNS}.  For now, we also record that $\omega \times \theta = ad - bc + b\theta - c\omega$.  To compute $\omega \times \theta$, one uses the identity $x^\# \times x = \tr(x^2)x-x^3$.  Furthermore,
\begin{align*} \omega_0^\# &= (\omega+b/3)^\# = a\theta - \frac{b}{3}\omega - \frac{2}{9}b^2 \\ &= a \theta_0 - \frac{b}{3}\omega_0 + a\left(\frac{c}{3}\right) - \left(\frac{b}{3}\right)^2, \end{align*}
\begin{align*} \theta_0^\# &= (\theta - c/3)^\# = \frac{c}{3}\theta - d\omega - \frac{2}{9}c^2 \\ &= \frac{c}{3}\theta_0 - d\omega_0 + d\left(\frac{b}{3}\right) - \left(\frac{c}{3}\right)^2, \end{align*}
and
\begin{align*} \omega_0 \times \theta_0 &= (\omega+b/3)\times (\theta -c/3) = \frac{2}{3}b \theta - \frac{2}{3}c \omega + ad - \frac{5}{9}bc \\ &= \frac{2}{3} b\theta_0 - \frac{2}{3}c \omega_0 + ad - \left(\frac{b}{3}\right)\left(\frac{c}{3}\right).\end{align*}

Binary cubic forms possess a quartic invariant, the discriminant.  If $f(x,y) = ax^3 + bx^2y + cxy^2 + dy^3$, then the discriminant of $f$ is
\[Q(f) = -27 a^2d^2 + 18adbc + b^2c^2-4ac^3-4db^3.\]
We remark that this discriminant can be easily understood in the context of the Freudenthal construction, see section \ref{subsec:Freud}.  Namely, the binary cubic $f(x,y) = ax^3 + bx^2y + cxy^2 + dy^3$ corresponds to the element $(a,\frac{b}{3},\frac{c}{3},d)$ in the Freudenthal construction $W_F$, for the trivial cubic norm structure on $F$.  (This is the structure from Example \ref{ex:deg}, case (3).) We set $q_f = q((a,b/3,c/3,d))$.  Then the discriminant $Q(f) = -27q_f$.  Indeed,
\begin{align*} -27q_f &= -\frac{1}{3}q((3a,b,c,3d)) \\ &= -\frac{1}{3}\left\{(9ad - 3bc)^2 + 12ac^3 + 12db^3 - 12b^2c^2\right\} \\ &= -3(3ad-bc)^2- 4ac^3-4db^3+4b^2c^2 \\ &= -27 a^2d^2 + 18adbc + b^2c^2-4ac^3-4db^3.\end{align*}
One has the equality
\[\mathrm{disc}(1,\omega,\theta) := \det \left(\begin{array}{ccc} \tr(1) &\tr(\omega)&\tr(\theta) \\ \tr(\omega) &\tr(\omega^2) &\tr(\omega \theta) \\ \tr(\theta) &\tr(\theta \omega) &\tr(\theta^2) \end{array}\right) = \det \left(\begin{array}{ccc} 3&-b&c \\ -b&b^2-3ac&-3ad \\ c&-3ad &c^2-2bd \end{array}\right) = Q(f).\]

The left action of $\GL_2$ on binary cubic forms is $(g \cdot f)(x,y) = \det(g)^{-1}f((x,y)g)$.  With this action, one has $q(g \cdot f) = \det(g)^2q(f)$, and $\left(\begin{array}{c} \omega_0\\ \theta_0\end{array}\right) \mapsto g \left(\begin{array}{c} \omega_0 \\ \theta_0\end{array}\right)$.  This latter identity, $\left(\begin{array}{c} \omega_0\\ \theta_0\end{array}\right) \mapsto g \left(\begin{array}{c} \omega_0 \\ \theta_0\end{array}\right)$, has the following meaning:  Set $f' = (g \cdot f)(x,y) = a'x^3 + b'x^2y+c'xy^2 + d'y^3$, denote $(T',(1,\omega',\theta'))$ the based cubic ring associated to $f'$, and set $\omega_0' = \omega' + \frac{b'}{3}$, $\theta_0' = \theta'-\frac{c'}{3}$.  Then one can identify $T$ with $T'$ so that $\left(\begin{array}{c} \omega_0'\\ \theta_0'\end{array}\right) = g \left(\begin{array}{c} \omega_0 \\ \theta_0\end{array}\right)$ inside $L$.

\subsection{Cubic norm structures}\label{subsec:CNS} In this subsection we discuss cubic norm structures. For completeness, and to fix notation, we begin with the definition of a cubic norm structure.  However, readers not familiar with cubic norm structures might wish to proceed first to the examples, and only then come back to the definition if desired.

\begin{definition}(See for example \cite{peterssonRacine1}.) Let $J$ be a finite dimensional $F$ vector space.  A \emph{cubic norm structure} on $J$ consists of the data of an element $1 \in J$, a cubic polynomial map $N: J \rightarrow F$, a quadratic polynomial map $\#: J \rightarrow J$, and a non-degenerate symmetric pairing $(\quad, \quad): J \otimes J \rightarrow F$, that are subject to the following requirements.
\begin{itemize}
\item $(x^\#)^\# = N(x)x$;
\item $N(1) = 1$, $1^\# = 1$, $1 \times x = (1,x) - x$ for all $x \in J$;
\item One has $(x,y) = (1,1,x)(1,1,y) - (1,x,y)$ and 
\[N(x+y) = N(x) + (x^\#,y) + (x,y^\#) + N(y)\]
for all $x,y \in J$.\end{itemize}
Here $x \times y := (x+y)^\# - x^\# - y^\#$, and $(x,y,z)$ is the polarization of the norm form $N$, i.e., 
\[(x,y,z) := N(x+y+z) - N(x+y)-N(x+z) - N(y+z) +N(x) +N(y) +N(y)\]
is the unique symmetric trilinear form on $J$ satisfying $(x,x,x) = 6N(x)$. The notation $\tr(x) :=(1,x)$ and $\tr(x,y) :=(x,y)$ is also used.  The map $N$ is called the norm on $J$, while the map $\#$ is called the adjoint.

It follows from these requirements that $(x \times y, z) = (x,y,z) = (x,y\times z)$.  For $x,y$ in $J$, one sets
\[U_xy := - x^\# \times y + (x,y)x.\]
One has $N(U_{x}y) = N(x)^2N(y)$. \end{definition}
One can also replace $F$ by a subring $R$, and $J$ by a finite, free $R$-module.  We will sometimes be in this slightly more general situation.

\begin{remark} Suppose $J$ is an $F$ vector space, and one has a quadratic polynomial map $\#: J \rightarrow J$, a cubic polynomial map $n: J \rightarrow F$, a non-degenerate symmetric pairing $J \otimes_F J \rightarrow F$, and an element $1 \in J$.  As above, define $x\times y = (x+y)^\# - x^\# - y^\#$.  To check that this data defines a cubic norm structure on $J$, it suffices to check that
\begin{itemize}
\item $1 \times x = (1,x) - x$;
\item $(x^\#)^\# = n(x)x$;
\item $(x,x^\#) = 3n(x)$;
\item the trilinear form $(x, y\times z)$ is symmetric in $x,y,z$. \end{itemize}
\end{remark}

The key examples of cubic norm structures are as follows.
\begin{example}\label{ex:Acsa} Suppose $J = A$ is a central simple $F$-algebra of degree $3$.  One defines $1 \in J$ to be the identity in $A$, and $N(x)$ to be the reduced norm of $x \in A$.  The pairing $(x,y)$ is $(x,y) = \tr(xy)$, where the trace on the right is the reduced trace of $A$.  For $x \in A$ invertible, $x^\# = N(x)x^{-1}$, and in fact this map extends to a quadratic polynomial map on all of $A$.  For $x,y \in A$, one has the identity $U_{x}y = xyx$. \end{example}

\begin{example}\label{ex:H3C} Suppose $C$ is a composition algebra.  Set $J = H_3(C)$, the $3 \times 3$ Hermitian matrices with elements in $C$.  Elements of $J$ are of the form
\[x = \left(\begin{array}{ccc} c_1 & a_3 & a_2^* \\ a_3^* & c_2 & a_1 \\ a_2 & a_1^* & c_3\end{array}\right)\]
with $c_i \in F$ and $a_i \in C$.  The element $1$ for $J$ is just the $3 \times 3$ identity matrix $1_3$.  The norm on $J$ is
\[N(x) = c_1 c_2 c_3 - c_1n_{C}(a_1) - c_2n_C(a_2) - c_3n_C(a_3) + \tr_{C}(a_1 a_2 a_3)\]
and the adjoint on $J$ is
\[x^\# = \left(\begin{array}{ccc} c_2c_3 - n_C(a_1) &a_2^*a_1^*- c_3 a_3 & a_3 a_1 - c_2 a_2^*\\ a_1 a_2 - c_3 a_3^* & c_1 c_3 - n(a_2) & a_3^*a_2^*-c_1a_1\\ a_1^*a_3^*-c_2a_2 & a_2a_3 - c_1 a_1^* & c_1 c_2 - n(a_3)\end{array}\right).\]
One has $\tr(x) = (1,x) = c_1 + c_2 +c_3$ and $(x,x') = \tr(xx'+x'x)/2 = \sum_{i}{c_i c_i'} + \sum_{j}{(a_j,a_j')}$ where $(a_j,a_j') = n_C(a_j + a_j') - n_C(a_j) - n_C(a_j') = \tr_C(a_j^* a_j')$.  If $C$ is associative, then for $x, y \in J = H_3(C)$ one has the identity $U_{x}y = xyx$.
\end{example}

\begin{example}\label{JBK} Suppose $K = F$ or a quadratic \'{e}tale extension of $F$, and $B$ is a central simple $K$ algebra, with involution of the second kind $*$.  That is, we assume $*: B\rightarrow B$ is an order-reversing involution that induces the unique involution on $K$ over $F$.  (So, if $K = F$, this involution on $K$ is the identity.) Set $J$ to be $B^{* = 1}$, the $F$-subspace of $B$ fixed by the involution $*$.  From example \ref{ex:Acsa}, $B$ is a cubic norm structure over $K$.  Restricting the norm, adjoint, and pairing to $J$ makes $J$ a cubic norm structure over $F$.  If $C = K$ is a commutative associative composition algebra, $B = M_3(C)$, with involution $*$ being conjugate transpose, then $J$ is the cubic norm structure $H_3(C)$ from example \ref{ex:H3C}. \end{example}

There are also degenerate versions of some of the above examples.
\begin{example}\label{ex:deg} Suppose $J = A$ is one of the following associative $F$-algebras:
\begin{enumerate} 
\item $A = E$, a cubic \'{e}tale $F$-algebra;
\item $A = F \times C$, for $C$ an associative composition algebra;
\item $A = F$. \end{enumerate}
(When $C = K$ is a quadratic \'{e}tale $F$-algebra, there is overlap between cases (1) and (2).).  In each case, let $1$ for $J$ be the identity of the algebra $A$.  In case (1), let the trace pairing and norm on $E$ be the usual trace pairing and norm, and $x^\# = N(x)x^{-1}$ for $x$ invertible.  (Again, this extends as a polynomial map to all of $E$.)  In case (2), set $N((\alpha,x)) = \alpha n_{C}(x)$ for $\alpha \in F$ and $x \in C$, $(\alpha,x)^\# = (n_C(x),\alpha x^*)$, and $((\alpha,x),(\beta,y)) = \alpha \beta + (x,y)$.  In case (3), define $N(x) = x^3, x^\# = x^2$, and $(x,y) = 3xy$, where here $x, y \in F$. \end{example}

We introduce one non-standard piece of terminology.
\begin{definition} Suppose $J$ is a cubic norm structure.  We say $J$ is an \emph{associative cubic norm structure} if $J = A$ is as in Example \ref{ex:Acsa} or Example \ref{ex:deg}.  In this case, one has the identities $xx^\# = x^\# x = N(x)$, $(xy)^\# = y^\#x^\#$, $N(xy) = N(x)N(y)$, and $U_{x}y = xyx$. \end{definition}
All of the above constructions and definitions have generalizations to the case when $F$ is replaced by a subring $R$, and one assumes $J$ is finite and free over $R$.

There is also a weaker notion a \emph{cubic norm pair}, which we now define.  (This notion will only be used in section \ref{bII}.)

\begin{definition} Suppose $J$, $J^\vee$ are finite dimensional $F$ vector spaces in perfect duality by a pairing $(\,,\,): J \otimes J^\vee \rightarrow F$.  The structure of a cubic norm pair on $J, J^\vee$ is the data of cubic polynomial maps $N_J : J \rightarrow F$, $N_{J^\vee}: J^\vee \rightarrow F$, and quadratic polynomial maps $\#_J: J \rightarrow J^\vee$ and $\#_{J^\vee}: J^\vee \rightarrow J$ that are subject to the following requirements:
\begin{enumerate}
\item One has
\[N_J(x + y) = N_J(x) + (y,x^{\#_J}) + (x,y^{\#_J}) + N_J(y)\]
for all $x, y \in J$ and similarly for $N_{J^\vee}$. 
\item One has $(x^{\#_J})^{\#_{J^\vee}} = N_J(x) x$ for all $x \in J$ and similarly $(y^{\#_{J^\vee}})^{\#_{J}} = N_{J^\vee}(y)y$ for all $y \in J^\vee$. \end{enumerate}
One sets 
\[x \times_{J} y := (x+y)^{\#_{J}} - x^{\#_J} - y^{\#_J}\]
if $x, y \in J$ and similarly for $J^\vee$.  One has $(x, y \times z) = (z,x\times y) = (x,y,z)$ for $x, y, z \in J$ if $(x,y,z)$ denotes the polarization of the norm form, and again similarly for $J^\vee$.
\end{definition}
Obviously, every cubic norm structure gives rise to a cubic norm pair with $J^\vee = J$.  

We now define the rank of elements of $J$.
\begin{definition} Suppose $J$ is a cubic norm structure.  All elements of $J$ have rank at most $3$.  An element $x$ has rank at most $2$ if $N_J(x) = 0$.  An element $x$ has rank at most $1$ if $x^\# = 0$.  Finally, $0$ is the unique element of rank $0$. \end{definition}

For more on cubic norm structures, see, for example, \cite{kmrt} or \cite{mccrimmon}.

\subsubsection{Identities for cubic norm structures}\label{special_Id} We give several identities valid in cubic norm structures.  These identities will be used throughout the paper.  First, one has the relations
\begin{align*} x \times (x^\# \times y) &= n(x)y + (x,y)x^\#, \\ x^\# \times (x\times y) &= n(x)y + (x^\#,y)x, \\ (x\times y)^\# + x^\# \times y^\# &= (x,y^\#)x + (x^\#,y)y. \end{align*}

Recall that the cubic norm structure $J$ is said to be \emph{special} if there exists an associative algebra $W$, and an inclusion $J \subseteq W$ so that $U_{x}y = xyx$.  Here the multiplication on the right hand side of this equality is that of $W$.  If $J$ is special, then one has $xx^\# = x^\#x = n(x)$.  Linearizing this identity gives $x(x\times y) + yx^\# = (x^\#,y)$ and similarly $(x\times y)x + x^\# y = (x^\#,y)$, one consequence of which is the identity $y^\#x(x\times y) = (x^\#,y)y^\# - n(y)x^\#$.  One also has the identity
\begin{equation}\label{eqn:Uyz} x \times (y \times z) = (x,y)z + (x,z)y - (yxz+zxy)\end{equation}
for special $J$.

\subsubsection{The group $M_J$} We now define a reductive group $M_J$ that acts on $J$ if $J$ is a cubic norm structure.  Namely, define $M_J$ to be the group of $(g, \lambda) \in \GL(J) \times \GL_1$ so that $N_J(g x) = \lambda N_J(x)$ for all $x \in J$.  The group $M_J$ preserves the rank of elements of $J$, and its action makes $J$ into a prehomogeneous vector space.  The open orbit consists of the elements of rank $3$.

More generally, suppose $J, J^\vee$ form a cubic norm pair over $F$.  There is an analogue of the group $M_J$ for the pair $(J,J^\vee)$.  Define $\widetilde{M}(J, J^\vee)$ to be the group of $(\alpha,t,t^\vee, \delta) \in \GL_1 \times \GL(J) \times \GL(J^\vee) \times \GL_1$ that satisfy the following identities, for all $b \in J, c \in J^\vee$:
\begin{itemize}
\item $(t(b), t^\vee(c)) = \alpha \delta (b,c)$;
\item $N_{J^\vee}(t^\vee(c)) = \alpha \delta^2 N_{J^\vee}(c)$;
\item $N_{J}(t(b)) = \alpha^2 \delta N_{J}(b)$;
\item $t(b)^{\#_J} = \alpha t^\vee(b^{\#_J})$ and $t^\vee(c)^{\#_{J^\vee}} = \delta t(c^{\#_{J^\vee}})$. \end{itemize}
Now, define the group $M(J,J^\vee)$ to be the subgroup of $\widetilde{M}(J)$ where $\alpha \delta = 1$.  If $J$ is a cubic norm structure, then $M_J = M(J,J)$, which makes sense because $J, J$ is a cubic norm pair. 

\subsection{The Freudenthal construction}\label{subsec:Freud} We now discuss the so-called Freudenthal construction associated to a cubic norm structure $J$.  That is, to such a $J$, we define the space the space $W_J$ together with its natural quartic form and symplectic pairing.  We aslo discuss the group $H(W_J)$, that acts on $W_J$ preserving these structures.

Suppose $J$ is a cubic norm structure.  We set 
\[W_J = F \oplus J \oplus J^\vee \oplus F = F \oplus J \oplus J \oplus F.\]
We write typical elements of $W_J$ as $(a,b,c,d)$.  The space $W_J$ carries a symplectic pairing
\[\langle (a,b,c,d), (a',b',c',d') \rangle := ad' - (b,c') + (c,b')-da'.\]
The space $W_J$ also carries a natural quartic form (due to Freudenthal), given by
\[q((a,b,c,d)) = (ad - (b,c))^2 + 4an(c) + 4dn(b) - 4(b^\#,c^\#).\]

Polarizing the quartic form $q$, there is a symmetric $4$-linear form $(w,x,y,z)$ on $W_J$ normalized by the condition $(v,v,v,v) = 2q(v)$ for $v$ in $W_J$.  Since the symplectic pairing is non-degenerate, there is a symmetric trilinear form $t: W_J \times W_J \times W_J \rightarrow W_J$ normalized by the identity 
\[\langle w, t(x,y,z) \rangle = (w,x,y,z).\]
Finally, for $v \in W_J$, set $v^\flat = t(v,v,v)$.  If $v = (a,b,c,d)$, then $v^\flat = (a^\flat,b^\flat,c^\flat,d^\flat)$, where
\begin{itemize}
\item $a^\flat = -a^2d +a(b,c)-2n(b)$;
\item $b^\flat = - 2 c \times b^\# + 2ac^\# - (ad-(b,c))b$;
\item $c^\flat = 2b \times c^\# - 2db^\#+(ad-(b,c))c$;
\item $d^\flat = ad^2 - d(b,c)+2n(c)$; \end{itemize}

One has the following fact.
\begin{lemma}\label{lem:vvflat} For $v \in W_J$, one has $\langle v, v^\flat \rangle = 2q(v)$ and $(v^\flat)^\flat = - q(v)^2 v$. \end{lemma}

We now define the rank of elements of $W_J$.
\begin{definition}  All elements of $W_J$ have rank at most $4$.  An element $v$ of $W_J$ has rank at most $3$ if $q(v) = 0$.  An element $v$ of $W_J$ has rank at most $2$ if $(v,v,v,w) = 0$ for all $w \in W_J$.  Equivalently, $v$ has rank at most $2$ if $v^\flat = 0$.  An element $v$ has rank at most $1$ if $(v,v,v',w) = 0$ for all $w \in W_J$ and $v' \in W_J$ satisfying $\langle v,v' \rangle = 0$.  Equivalently, $v$ has rank at most $1$ if $t(v,v,w) \in Fv$ for all $w \in W_J$.  Finally, $0$ is the unique element of rank $0$. \end{definition}

\subsubsection{The group $H(W_J)$} Let $W_J$ be as above.  We now define a group $H(W_J)$ that acts on $W_J$.  We define the group $H(W_J)$ to be the subgroup of $g \in \GSp(W_J; \langle \,,\, \rangle)$ satisfying $q(gv) = \nu(g)^2q(v)$ for all $v\in W_J$.  Here $\nu(g) \in \GL_1$ is the similitude of $g$ in $\GSp(W_J)$.

We define some maps on $W_J$ that are in $H(W_J)$.
\begin{itemize}
\item For $X \in J$, define 
\[n_J(X)(a,b,c,d) = (a,b+aX,c + b \times X + aX^\#, d+(c,X)+(b,X^\#) + an(X))\]
and similarly for $Y \in J^\vee = J$ define 
\[\overline{n}_J(Y) = n_{J^\vee}(Y)(a,b,c,d)= (a + (b,Y) + (c,Y^\#) + dn(Y),b+ c \times Y + dY^\#, c+ dY, d).\]
\item For $\lambda \in \GL_1$, define
\[m(\lambda)(a,b,c,d) = (\lambda^2 a, \lambda b,c, \lambda^{-1}d).\]
\item More generally, for $m =(\alpha, t, t^\vee, \delta) \in \widetilde{M}(J,J)$, define
\[m(a,b,c,d) = (\alpha a, t(b), t^\vee(c),\delta d).\]
\item Define $w_J(a,b,c,d) = (d,-c,b,-a).$
\end{itemize}
The following proposition is well-known, the only nontrivial piece being for the maps $n_J(X), n_{J^\vee}(Y)$.  (See \cite[Lemma 2.3]{springer} for this nontrivial piece.)
\begin{proposition}\label{HWJmaps} The above maps are all elements of $H(W_J)$. \end{proposition}

The action of $H(W_J)$ preserves the ranks of elements of $W_J$, and makes it into a prehomogeneous vector space.  The elements of rank $4$ of $W_J$ form the open orbit.  All elements of rank $1$ of $W_J$ are in the same $H(W_J)$ orbit.

One has the following fact.
\begin{lemma} If $v=(a,b,c,d) \in W_J$ is rank one, then $b^\# = ac$, $c^\# = db$, and $(b,c) = 3ad$.  If either $a$ or $d$ is nonzero, then the first two conditions imply $v$ is rank one. \end{lemma}

All rank one elements of $W_J$ are in the $H(W_J)$-orbit of the element $(1,0,0,0)$.  More generally, arbitrary nonzero elements $v$ of $H(W_J)$ are in the $H(W_J)$-orbit of elements of the form $(1,0,c,d)$.  
\begin{lemma} Every nonzero element $v$ of $W_J$ can be moved to one of the form $(1,0,c,d)$ by applying operators $n_{J}(X)$, $\overline{n}_J(Y)$, and $m(\lambda)$. \end{lemma}
A proof of the previous two lemmas may be found in \cite[Proposition 11.2, Corollary 11.3]{ganSavinMin}.

From the definitions, if $g \in H(W_J)$, one obtains $(gv)^\flat = \nu(g) g (v^\flat)$ and similarly $t(gv_1, gv_2,gv_3) = \nu(g) g t(v_1,v_2,v_3)$.

\subsection{Particulars for associative cubic norm structures}\label{sec:new} In this subsection we give more preliminaries that apply specifically to the case that $J = A$ is an associative cubic norm structure.  Throughout this subsection, $A$ is assumed to be an associative cubic norm structure.  We write $n$ or $n_A$ for the cubic norm map $A \rightarrow F$.

We first note the following.  Suppose $u,v\in A$ and $n(u)n(v) \in F^\times$.  Then if one sets $\alpha = n(u), \delta = n(v)$, $t(b) = vbu^\#$, $t^\vee(c) = ucv^\#$, the map $L(u,v) := m(\alpha,t,t^\vee,\delta)$ is in $\widetilde{M}(A,A)$.

\subsubsection{The quadratic map $R$} Define a map $R: W_A \rightarrow M_2(A)$ via
\[v = (a,b,c,d) \mapsto R(v) = \left(\begin{array}{cc} ad + 2cb - (c,b) & 2b^\# - 2ac \\ 2db - 2c^\# & -ad + (b,c) - 2bc \end{array}\right).\]
\begin{proposition}\label{prop:Rsquared} One has $R(v)^2 = q(v) 1_2$ for all $v \in W_A$. \end{proposition}
\begin{proof} Computing the $(1,2)$ entry of $R(v)^2$, one obtains 
\begin{align*} R(v)^2_{12} &= (ad + 2cb - (b,c))(2b^\# - 2ac) + (2b^\# - 2ac)(-ad + (b,c) - 2bc) \\  &= 2cb(2b^\# - 2ac) + (2b^\# - 2ac)(-2bc) \\ &= 0. \end{align*}
Similarly, the $(2,1)$ entry of $R(v)^2 = 0$.  We thus must compute the diagonal entries of $R(v)^2$.  The $(1,1)$ entry is
\begin{align*} R(v)^2_{11} &= (ad - (b,c) + 2cb)^2 + 4(b^\# -ac)(db - c^\#) \\ &= (ad - (b,c))^2 + 4E \end{align*}
where
\begin{align*} E &= (ad - (b,c))cb + cbcb + (b^\# - ac)(db - c^\#) \\ &= an(c) + dn(b) - \left(\tr(cb)cb - (cb)^2 + (cb)^\#\right) \\&= an(c) + dn(b) - \tr((cb)^\#) \end{align*}
since for all $x$ in $A$ we have 
\[x^2 - \tr(x)x + \tr(x^\#) - x^\# = 0.\]
Thus $R(v)^2_{11} = q(v)$.  Similarly, $R(v)^2_{22} = q(v)$.  This completes the proof.\end{proof}

We will sometimes have desire to use the more symmetrical-looking matrix $S(v):=\frac{1}{2}R(v)\mm{}{-1}{1}{}$.  If $v = (a,b,c,d)$, then
\begin{equation}\label{eqn:Sdef} S(v):=\frac{1}{2}R(v) \left(\begin{array}{cc} &-1\\1& \end{array}\right) = \left(\begin{array}{cc} b^\# - ac & ad-cb - \tr(ad-cb)/2\\ ad-bc-\tr(ad-bc)/2 & c^\# -db \end{array}\right).\end{equation}

From Proposition \ref{prop:Rsquared} one obtains the identity $S(v)J_2S(v) = -\frac{q(v)}{4} J_2$, where $J_2 = \mm{}{1}{-1}{}$.

\begin{example} Suppose $A = F \times F \times F$ and $v = (a,b,c,d)$, with $b = (b_1,b_2,b_3)$, $c = (c_1,c_2,c_3)$.  The element $v$ is a $2 \times 2 \times 2$ cube, as considered in \cite{bhargavaI}.  Then $S(v) \in M_2(A) = M_2(F) \times M_2(F) \times M_2(F)$, and one has $S(v) = (q_1,q_2,q_3)$ with
\begin{align*} q_1 &= \left(\begin{array}{cc} b_2b_3 - ac_1 & \frac{-b_1c_1+b_2 c_2 + b_3 c_3  - ad}{2} \\ \frac{-b_1c_1+b_2 c_2 + b_3 c_3  - ad}{2} & c_2c_3 - db_1 \end{array}\right)\\ q_2 &=\left(\begin{array}{cc} b_3b_1 - ac_2 & \frac{-b_2c_2+b_3 c_3 + b_1 c_1  - ad}{2} \\ \frac{-b_2c_2+b_3 c_3 + b_1 c_1  - ad}{2} & c_3c_1 - db_2 \end{array}\right),\\ q_3 &= \left(\begin{array}{cc} b_1b_2 - ac_3 & \frac{-b_3c_3+b_1 c_1 + b_2 c_2  - ad}{2} \\ \frac{-b_3c_3+b_1 c_1 + b_2 c_2  - ad}{2} & c_1c_2 - db_3 \end{array}\right).\end{align*}
The $q_i$ are $2 \times 2$ symmetric matrices, or equivalently, binary quadratic forms.  These are (essentially) the $3$ quadratic forms associated to $v$ from \cite{bhargavaI}.  Since one always has $mJ_2 m^{t} = \det(m) J_2$ for $m\in M_2(F)$, the identity $S(v)JS(v) = -\frac{q(v)}{4}J$ is the statement that the $3$ quadratic forms $q_i$ have the same discriminant, which is equal to the discriminant of $v$. \end{example}

We record how $R(v)$ behaves under the action of certain elements of $H(W_A)$.
\begin{proposition}\label{Requiv} Suppose $v \in W_A$, $X \in A$, and $m,n\in A$ with $n_A(m)n_A(m) \in R^\times$.  Then $R(n(X) v) = \mm{1}{}{X}{1}R(v)\mm{1}{}{-X}{1}$, $R(w_A(v)) = \mm{}{1}{-1}{}R(v)\mm{}{-1}{1}{}$, and 
\begin{align*} R(L(m,n)v) &= n_A(m)n_{A}(n)\left(\begin{array}{cc}m& \\ &n\end{array}\right) R(v)\left(\begin{array}{cc}m^{-1}& \\ &n^{-1}\end{array}\right) \\ &= \left(\begin{array}{cc}m& \\ &n\end{array}\right) R(v)\left(\begin{array}{cc}n_A(n)m^{\#}& \\ &n_{A}(m)n^{\#}\end{array}\right).\end{align*}\end{proposition}
\begin{proof} These are all direct computations. \end{proof}

The vanishing of $R(v)$, or equivalently $S(v)$, characterizes elements of $W_A$ of rank at most $1$.
\begin{lemma}\label{Rvan} For $v \in W_A$, one has $S(v) = 0$ if and only if $v$ has rank at most $1$. \end{lemma}
\begin{proof} Clearly $S(0) = 0$, and one computes $S((1,0,c,d)) = \mm{-c}{-d/2}{-d/2}{c^\#}$.  The lemma then follows by the equivariance result of Proposition \ref{Requiv}. The lemma also follows right away from \cite[Proposition 11.2]{ganSavinMin}.\end{proof}

\subsubsection{Rank one elements} We now say a bit more about rank one elements of $W_A$.  Suppose $\ell =(s,t) \in A^2$ is a row vector.  We define $\ell^{!} \in W_A$ to be the element
\[\ell^{!} = (s,t)^{!} := (n(s), s^\#t,t^\# s, n(t)).\]
Similarly, if $\eta \in A^2$ is a column vector, we define
\[\eta^{!} = \left(\begin{array}{c}u\\v\end{array}\right)^{!} = (n(u),vu^\#,uv^\#,n(v)).\]
Note that the order of multiplication in the $``b"$ and $``c"$-components has been switched.  We have the following lemma.
\begin{lemma}\label{lem:nRC} For $\ell \in A^2$ a row vector, and $\eta \in A^2$ a column vector, $\ell^{!}$ and $\eta^{!}$ are rank at most one, and one has $\langle \ell^{!},\eta^{!}\rangle = n_{A}(\ell \mm{}{1}{-1}{} \eta)$. \end{lemma}
\begin{proof} That $\ell^{!}$ and $\eta^{!}$ are rank at most one follows from Lemma \ref{Rvan}.  Suppose $\ell = (s,t)$ and $\eta = \left(\begin{array}{c}u\\v\end{array}\right)$.  Then
\begin{align*} \langle \ell^{!},\eta^{!} \rangle &= \langle (n(s), s^\#t,t^\#s,n(t)), (n(u),vu^\#,uv^\#,n(v)) \rangle \\ &= n(s)n(v) - (s^\#t,uv^\#) + (t^\#s,vu^\#) - n(t)n(u) \\ &= n(sv)-((sv)^\#,tu)+(sv,(tu)^\#)-n(tu) \\&= n(sv-tu)\end{align*}
and the lemma follows.\end{proof}

\subsubsection{The $\GL_2(A)$-action on $W_A$} In this paragraph we define a right and left $\GL_2(A)$ action on $W_A$, that preserves the symplectic and quartic forms, up to similitude.  We will give the details for the left action; the right action is completely analogous.

We begin by constructing a space $V_A$, with a left $\GL_2(A)$ action, that comes equipped with a map $V_A \rightarrow W_A$.  We will then show that this map is an isomorphism.

We consider $W_2 \otimes A$ to be the $2 \times 1$ column vectors with coefficients in $A$, an $A$ bi-module.
\begin{definition} We first define $V_{A_F}$.  Denote by $S_3$ the symmetric group on three letters. The vector space $V_{A_F}$ is defined to be the $S_3$ invariants of the quotient of $(W_2 \otimes A_F)^{\otimes 3}$ by the subspace spanned by $xa \otimes xa \otimes xa - n(a) x \otimes x \otimes x$ for $x \in W_2 \otimes A_F$ and $a \in A_F$.  That is, set
\[I_{A_F} := \left(\langle xa \otimes xa \otimes xa - n(a) x \otimes x \otimes x: x \in W_2 \otimes A_F, a \in A_F \rangle \right),\]
the subspace spanned by $xa \otimes xa \otimes xa - n(a) x \otimes x \otimes x$ for $x \in W_2 \otimes A_F$ and $a \in A_F$, and 
\[V_{A_F} := \left(\left((W_2 \otimes A_F) \otimes (W_2 \otimes A_F) \otimes (W_2 \otimes A_F)\right)\slash I_{A_F}\right)^{S_3}.\]
We define $V_{A}$ to be the image of $\left(A^2 \otimes A^2 \otimes A^2\right)^{S_3}$ in $V_{A_F}$.\end{definition}

Note that since $\GL_2(A_F)$ preserves the submodule $I_{A_F}$, and since the $\GL_2(A_F)$ and $S_3$ actions commute on $(W_2 \otimes A_F)^{\otimes 3}/I_{A_F}$, $\GL_2(A_F)$ acts on the left of $V_{A_F}$.  The group $\GL_2(A)$ preserves the module $V_A$ inside $V_{A_F}$.

A map $V_A \rightarrow W_A$ is defined as follows.  We map $x \otimes x \otimes x$ to $x^{!}$.  Linearizing this, suppose $x_i = \left(\begin{array}{c}u_i \\ v_i \end{array}\right)$ in $W_2 \otimes A$.  Recall the symmetric trilinear form on $A$ that satisfies the identities $(x,x,x) = 6n(x)$, $(x,y,z) = \tr(x \times y,z)$.  Then we map
\[\sum_{\sigma \in S_3}{\sigma \left(x_1 \otimes x_2 \otimes x_3\right)} \mapsto (\alpha(x_1,x_2,x_3),\beta(x_1,x_2,x_3),\gamma(x_1,x_2,x_3),\delta(x_1,x_2,x_3))\]
where
\begin{itemize}
\item $\alpha(x_1,x_2,x_3) = (u_1,u_2,u_3)$
\item $\beta(x_1,x_2,x_3) = v_1(u_2 \times u_3)+v_2(u_3 \times u_1) + v_3(u_1 \times u_2)$
\item $\gamma(x_1,x_2,x_3) = u_1(v_2 \times v_3) + u_2(v_3 \times v_1) + u_3(v_1 \times v_2)$
\item $\delta(x_1,x_2,x_3) = (v_1,v_2,v_3)$. \end{itemize}
Note that this does define a map on $V_A$, since $\alpha(x_1a,x_2 a, x_3a) = n(a) \alpha(x_1,x_2,x_3)$, $\beta(x_1 a,x_2 a, x_3 a) = n(a)\beta(x_1,x_2,x_3)$ etcetera.

\begin{lemma}\label{relnsInVA} Suppose $x,y, z$ are in $W_2 \otimes A$ and $a_1, a_2, a_3, a$ are in $A$.  Then in $(W_2 \otimes A_F)^{\otimes 3}/I_{A_F}$, one has the following equalities:
\begin{enumerate}
\item $\sum_{\sigma \in S_3}{\sigma \left(xa_1 \otimes xa_2 \otimes xa_3\right)} = (a_1,a_2,a_3) x \otimes x \otimes x$;
\item $xa \otimes ya \otimes z = x \otimes y \otimes za^\#$;
\item $xa_1 \otimes ya_2 \otimes z + xa_2 \otimes ya_1 \otimes z = x \otimes y \otimes z(a_1 \times a_2)$. \end{enumerate}
\end{lemma}
\begin{proof} The first item follows from linearizing the identity $xa \otimes xa \otimes xa = n(a) x \otimes x \otimes x$, and the third item follows from the second by linearization.  Thus, we prove the second statement.

To do this, fix $x,y,z$ in $W_2 \otimes A$, and consider the map $A \rightarrow (W_2 \otimes A_F)^{\otimes 3}/I_{A_F}$ given by
\[a \mapsto xa \otimes ya \otimes z - x \otimes y \otimes za^\#.\]
Since this is a polynomial map from $A$ to a finite dimensional $F$ vector space, to check that it is identically $0$, it suffices to check that it is $0$ on the Zariski dense set of $a$ with $n(a) \neq 0$.  But now in $(W_2 \otimes A_F)^{\otimes 3}/I_{A_F}$, one has
\[n(a) \left(x \otimes y \otimes za^\#\right) = xa \otimes ya \otimes za^\# a = n(a) \left(xa \otimes ya \otimes z\right),\]
and the lemma follows. \end{proof}

Fix the standard basis $e = \left(\begin{array}{c}1 \\ 0 \end{array}\right)$, $f = \left(\begin{array}{c}0 \\ 1 \end{array}\right)$ of $W_2$.
\begin{proposition}\label{prop:VAWA} The map $V_{A_F} \rightarrow W_{A_F}$ is a linear isomorphism, which induces an isomorphism of $R$-modules $V_A \rightarrow W_A$. \end{proposition}
\begin{proof} We check that the map $V_A \rightarrow W_A$ is an isomorphism by seeing that it is surjective, and checking the bound $\dim_F V_{A_F} \leq 2 + 2\dim_F A_F$.

To see the surjectivity, note that for $a \in R$ and $b \in A$, $a e \otimes e \otimes e \mapsto (a,0,0,0)$, and
\[bf \otimes e \otimes e + e \otimes bf \otimes e + e \otimes e \otimes bf \mapsto (0,b,0,0)\]
and similarly the elements $(0,0,c,0)$, $(0,0,0,d)$ are in the image of $V_A \rightarrow W_A$. 

It follows from Lemma \ref{relnsInVA} that the elements
\begin{itemize}
\item $a e\otimes e \otimes e, a \in F$;
\item $bf \otimes e \otimes e + e \otimes bf \otimes e + e \otimes e \otimes bf$, $b \in A_F$;
\item $ce \otimes f \otimes f + f \otimes ce \otimes f + f \otimes f \otimes ce$, $c \in A_F$;
\item $df \otimes f \otimes f$, $d \in F$ \end{itemize}
span $V_{A_F}$.  The dimension bound, and thus the proposition, follows.\end{proof}

As a consequence of Proposition \ref{prop:VAWA}, we obtain a left $\GL_2(A)$ action on $W_{A}$ by transport of structure.  Note that by the definition of this action, we have $(g \eta)^{!} = g \cdot \eta^{!}$, if $g \in \GL_2(A)$ and $\eta \in A^2$ is a column vector.

We next need to check that the induced action of $\GL_2(A)$ on $W_A$ preserves the symplectic and quartic form.  First, a lemma.
\begin{lemma}\label{Paction} The element $\mm{1}{}{X}{1}$ of $\GL_2(A)$ acts on $W_A$ as $n(X)$, $\mm{m}{}{}{n}$ acts on $W_A$ as $L(m,n)$, and $\mm{}{1}{-1}{}$ acts on $W_A$ as $w_A$. \end{lemma}
\begin{proof} These are all direct computations that follow from Lemma \ref{relnsInVA} and the definitions.  For example, $\mm{m}{}{}{n}$ acts on $(0,b,0,0)$ as
\[\sum_{cyc}{ e \otimes e \otimes bf} \mapsto \sum_{cyc}{me \otimes me \otimes nbf} = \sum_{cyc}{ e\otimes e \otimes nbm^\# f} = (0,nbm^\#,0,0).\]
\end{proof}

Denote by $P \subseteq \GL_2(A_F)$ the subgroup consisting of elements of the form $\mm{*}{0}{*}{*}$.
\begin{lemma}\label{Pgens} The group $\GL_2(A_F)$ is generated by $P$ and $J_2 = \mm{}{1}{-1}{}$. \end{lemma}
\begin{proof} We give a direct proof.  Suppose $g = \mm{a}{b}{c}{d}$ is in $\GL_2(A_F)$.  We first check that we may assume $c$ is invertible.  Multiplying $g$ on the right by $J_2$, we may assume $\mathrm{rank} d \geq \mathrm{rank} c$.  Now, consider $g^{-1} = \mm{p}{q}{r}{s}$.  We get that $cq + ds =1$, and thus
\[1 = n(cq+ds) = n(c)n(q) + (c^\#d,sq^\#) + (d^\#c,qs^\#) + n(d)n(s).\]
Now, if $d$ is rank $3$, we are done, and if $c = 0$, then $d$ is rank three.  If both $c$ and $d$ are rank at most $1$, then $n(cq+ds) = 0$, so we cannot have this case. Thus we may assume $d$ is rank two, and $c$ is rank two or rank one.  We conclude that not both of $c^\#d$ and $d^\# c$ are $0$.  It follows that we can find $y \in A_F$ and $\lambda \in \GL_1(F)$ so that $\lambda(c^\#d,y)+ \lambda^2 (d^\#c,y^\#) \neq 0$.  Indeed, if this is $0$ for all $\lambda$ and fixed $y$, then $(c^\#d,y) = 0$ and $(d^\#c,y^\#) = 0$, and if these are $0$ for all $y$, then $c^\#d = 0$ and $d^\#c = 0$.  Hence, there exists $x = \lambda y$ for which $n(c+dx) \neq 0$, and thus $\mm{a}{b}{c}{d} \mm{1}{}{x}{1} = \mm{*}{*}{c+dx}{*}$ has an invertible entry in the bottom row.

Thus, we've shown that we can multiply $g$ by elements of $P$ and $J_2$ so that the bottom right entry of $g$ is invertible.  Now, multiply $g = \mm{a}{b}{c}{d}$ on the right by $\mm{1}{}{x}{1}$, where $x = -d^{-1}c$.  One obtains an element in $P^{op} = \mm{*}{*}{0}{*}$.  Conjugation by $J_2$ moves $P^{op}$ to $P$, thus completing the lemma. \end{proof}

For $m \in M_2(A)$, denote by $\det(m)$ the degree $6$ reduced norm on $M_2(A)$.
\begin{corollary} The action of $\GL_2(A)$ on $W_A$ induced by the isomorphism $V_A \rightarrow W_A$ and the natural action of $\GL_2(A)$ on $V_A$ preserves the symplectic and quartic form on $W_A$. That is, $\langle g \cdot v, g \cdot w\rangle = \det(g) \langle v, w \rangle$ and $q(g \cdot v) = \det(g)^2 q(v)$ for all $v, w \in W_{A}$.\end{corollary}
\begin{proof} This follows immediately from Lemmas \ref{Paction} and \ref{Pgens}. \end{proof}

Note that since $V_A$ is defined polynomially, it is a module for $M_2(A)$.   We will require the following fact below.
\begin{corollary} Suppose $m \in M_2(A)$, and $x, y \in V_A$.  Then $\langle mx, my \rangle = \det(m) \langle x,y \rangle$ and $q(mx) = \det(m)^2q(x)$. \end{corollary}
\begin{proof} The difference $\langle mx, my \rangle - \det(m)\langle x,y\rangle$ is $0$ for all $m \in \GL_2(A) \subseteq M_2(A)$.  Since $\GL_2(A)$ is Zariski dense in $M_2(A)$, the identity follows.  The second identity is similar. \end{proof}

Recall the element $R(v) \in M_2(A)$ defined for $v \in W_A$.
\begin{lemma}\label{lem:Requiv} For $g \in \GL_2(A)$, we have $R(g \cdot v) = \det(g) gR(v)g^{-1}$. Here $g \cdot v$ is the left action of $\GL_2(A)$ on $W_{A}$.\end{lemma}
\begin{proof} This follows immediately from Lemma \ref{Paction} and Proposition \ref{Requiv}.\end{proof}

All of the above constructions and proofs and can be made analogously for the right action of $\GL_2(A)$ on row vectors $A^2$, which then induces a right action of $\GL_2(A)$ on $W_A$.  We state explicitly the ``right" version of the previous lemma. Set $R_{r}(v) = J_2R(v)J_2^{-1}$, where $J_2 = \mm{}{1}{-1}{}$.  We have the following.
\begin{lemma}\label{lem:rightAct} The element $\mm{1}{X}{}{1}$ acts on the right of $W_A$ as the map $n(X)$; the element $J_2=\mm{}{1}{-1}{}$ acts on the right of $W_{A}$ as $-w_{A}$, i.e., $(a,b,c,d)\cdot J_2 = (-d,c,-b,a)$; and the element $\mm{m}{}{}{n}$ acts on the right of $W_A$ as the map $R(m,n)$, with 
\[(a,b,c,d)\cdot R(m,n) = (n_{A}(m)a, m^\#bn,n^\#cm, n_{A}(n)d).\]
For $g \in \GL_2(A)$, we have $R_{r}(v\cdot g) = \det(g) g^{-1}R_{r}(v)g$. Here $v\cdot g$ is the right action of $\GL_2(A)$ on $W_{A}$.\end{lemma}

\begin{definition} Suppose $R$ is a semisimple ring, by which we mean $R$ is a subring of a finite product of fields $k = \prod_{i}{k_i}$.  Assume $A = A_R$ is an associative cubic norm structure over $R$.  Then, we define $V_{A_k} = \prod_{i}{V_{A_{k_i}}}$together with its associated $\GL_2(A_k) = \prod_{i}{\GL_2(A_{k_i})}$-action, and we have a $k$-linear isomorphism of $V_{A_k}$ with $W_{A_k} = \prod_{i}{W_{A_{k_i}}}$.  We define $V_{A_R}$ to be the image of $\left(W_2(A_R)^{\otimes 3}/I_A\right)^{S_3}$ in $V_{A_k}$, and then we similarly have an $R$-linear isomorphism of $V_{A_R}$ with $W_{A_R}$.\end{definition}

\subsubsection{An invariant of rank one elements of $W_{A_F}$} All the rank one \emph{lines} in $W_{A_F}$ are in the same $\GL_2(A_F)$ orbit, but that there is slight extra invariant of the rank one \emph{elements} of $W_{A_F}$, that distinguishes different $\GL_2(A_F)$-orbits.  This invariant measures how far a rank one element is from one of the form $\ell^{!}$ or $\eta^{!}$.  We will now define this invariant.

Suppose $k$ is a field, $A = A_{k}$ is an associative cubic norm structure over $k$, and $v \in W_A$ is rank one. Then there is a row vector $\ell \in A^2$ so that $\langle \ell^{!}, v \rangle \neq 0$, and similarly a column vector $\eta \in A^2$ so that $\langle v, \eta^{!} \rangle \neq 0$.  The following lemma says that these elements of $k^\times$ give well-defined classes in $k^{\times}/n(A^\times)$, and in fact this class is the same if one uses row vectors or column vectors.
\begin{lemma}\label{lem:lambdaInv} Suppose $k$ is a field, $A = A_{k}$ is an associative cubic norm structure over $k$, and $v \in W_{A}$ is rank one.  Then, there an element $\lambda \in k^{\times}$ so that
\begin{equation}\label{lambdaInv}\{\langle \ell^{!}, v \rangle: \ell \in A^2 \text{ row vector }\} \cap k^\times = \lambda n(A^\times) = \{\langle v, \eta^{!} \rangle: \eta \in A^2 \text{ column vector }\} \cap k^\times.\end{equation} \end{lemma}
As a consequence of the lemma, we can make the following definition.
\begin{definition}\label{def:lambda} Suppose $k$ is a field, $A$ is an associative cubic norm structure over $k$, and $v$ in $W_A$ is rank one.  We define $\lambda(v) \in k^\times/n(A^\times)$ to be the class of $\lambda$ determined by Lemma \ref{lem:lambdaInv}.  Similarly, if $k = \prod_{i}{k_i}$ is a finite product of fields, $A$ is an associative cubic norm structure over $k$, and $v = \prod_{i}{v_i} \in W_A$ is rank one in every component, then we define $\lambda(v) = \prod_i{\lambda(v_i)}$ in $\prod_{i}{k_i^\times/n(A_{k_i}^\times)} = k^\times/n(A^\times)$.\end{definition}

\begin{proof}[Proof of Lemma \ref{lem:lambdaInv}] Suppose $r \in A$.  Since $(r\ell)^{!} = n(r) \ell^{!}$ and $(\eta r)^{!} = n(r) \eta^{!}$, the sets on the left and right side of (\ref{lambdaInv}) are $n(A^\times)$-cosets.  

Since all rank one lines are in the same $\GL_2(A)$-orbit, there is $\lambda_1, \lambda_2 \in k^\times$, and $\eta_1, \ell_2$ a column vector and a row vector in $A^2$, so that $v = \lambda_1 \eta_1^{!} = \lambda_2 \ell_2^{!}$.  Now, suppose $v = \lambda_1 \eta_1^{!}$.  Then for any $\ell \in A^2$ a row vector, $\langle \ell^{!},v \rangle = \lambda_1 n(\ell J_2 \eta_1)$, where $J_2 = \mm{}{1}{-1}{}$.  Hence it follows that the set on the left of (\ref{lambdaInv}) forms a single $n(A^\times)$-coset, and this coset is represented by $\lambda_1$.  Similarly for the set on the right of (\ref{lambdaInv}); it forms a single $n(A^\times)$-coset represented by $\lambda_2$.

It remains to prove that $\lambda_1$ and $\lambda_2$ represent the same class in $k^\times/n(A^\times)$.   This is clear if either the $``a"$ or $``d"$-components of $v$ are nonzero, and in general it follows from this and the $\GL_2(A)$ action.\end{proof}

In the following lemma, and below, a primitive vector of $W_2(A_F)$ is a vector which may be completed to an $A_F$-basis of $W_2(A_F)$.
\begin{lemma} Suppose $\eta \in W_2(A_F)$, and $\eta^{!} \neq 0$ in $W_{A_F}$.  Then $\eta$ is primitive. \end{lemma}
\begin{proof} Write $\eta = \left(\begin{array}{c} u \\ v\end{array}\right)$.  Since $\eta^{!} \neq 0$, by acting by $\GL_2(A_F)$, we may assume $\eta^{!} = \gamma (1,0,0,0)$ for some $\gamma \in F^\times$.  But then $n(u) \in F^\times$, so $\eta$ is primitive. \end{proof}

\section{The first lifting law for $W_J$}\label{bI}
Suppose $A$ is an associative cubic norm structure, over the field $F$ or over a subring $R$ of $F$.  Then the group $\GL_2(A)$ acts on $W_{A}$, as shown in the previous section.  Denote by $\SL_2(A)$ the subgroup of $\GL_2(A)$ that is kernel of the degree $6$ norm map.  Equivalently, $\SL_2(A)$ is the subgroup of $\GL_2(A)$ that acts on $W_{A}$ with similitude $1$.  The purpose of this section is to parametrize the orbits of $\SL_2(A)$ on $W_A$, over the field $F$ and over a subring $R$ of $F$.  We begin by proving two lifting laws, one for general cubic norm structures $J$ and another more precise lifting law for associative cubic norm structures $A$.  We then use the second of these lifting laws to parametrize the orbits of $\SL_2(A)$ on $W_A$. 

The lifting laws are proved in subsection \ref{subsec:LLB1}.  In subsection \ref{subsec:B1F}, we parametrize the orbits of $\SL_2(A)$ on $W_A$ over a field $F$, and in subsection \ref{subsec:integralB1} we parametrize the orbits over a subring $R$ of $F$. The main result of this section is Corollary \ref{cor:equivG1}, which implies Theorem \ref{twistThm1}.
 
\subsection{The lifting law}\label{subsec:LLB1} In this subsection, we state the two lifting laws.  The first lifting law applies to the general case of $W_J$ with $J$ a cubic norm structure, while the second lifting law applies to the case $W_A$ with $A$ an associative cubic norm structure, but gives more information.

Here is the first lifting law.  It is likely known to experts in Freudenthal triple systems.
\begin{theorem}\label{b1J} Let $J$ be a cubic norm structure over the field $F$.  Suppose $v \in W_J$ has $q(v) \neq 0$, and suppose $F \subseteq k$ is an extension of fields, and $\omega \in k$ satisfies $\omega^2 = q(v)$ in $F$.  The elements $\omega v \pm v^\flat$ in $W_J \otimes_F k = W_{J_k}$ are rank one.  In fact, one has the equality
\begin{equation}\label{eq:b1J}3t(\omega v+v^\flat, \omega v + v^\flat, x) = \langle x, \omega v+v^\flat \rangle (\omega v + v^\flat)\end{equation}
for all $x \in W_{J_k}$ and similarly for $\omega v - v^\flat$.
\end{theorem}
\begin{proof}  Note that the fact that $\omega v + v^\flat$ is rank one follows immediately from (\ref{eq:b1J}), which gives more information.  So, we must check (\ref{eq:b1J}).

As mentioned, this result is likely known to experts in Freudenthal triple systems.  For instance, it follows by a one-line computation from the definition of a Freudenthal triple system and \cite[Lemma 4.4]{springer}.  One can also give a simple proof using prehomogeneity and equivariance for the action of group $H(W_{J_k})$.  Indeed, it is immediately seen that both sides of (\ref{eq:b1J}) change in the same way under the action of $H(W_{J_k})$, thus it suffices to check the equality (\ref{eq:b1J}) at a single non-degenerate $v$ in $W_J$.  For instance, if one takes $v = (a,0,0,d)$, both sides of (\ref{eq:b1J}) are easy to compute, and they are seen to be equal. \end{proof}

For $v, \omega$ as in Theorem \ref{b1J}, set $X(v) = (\omega v + v^\flat)/2$ and $\overline{X(v)} = (-\omega v + v^\flat)/2$. Note that $\langle X(v),\overline{X(v)}\rangle = \omega q(v)$.  

We give a complement to the lifting law, which explains the way in which the rank one lift $X(v)$ of $v$ is unique.
\begin{lemma}\label{uniqueLifts} Suppose $k$ is a field, $E$ is an \'{e}tale extension of $k$, $\omega \in E$ satisfies $\omega^2 \in k$ and $1, \omega$ are linearly independent over $k$.  Additionally, suppose $J$ is a cubic norm structure over $k$, $v_1, v_2 \in W_J$ with $q(v_1) \neq 0$, and $X =\omega v_1 + v_2  \in W_J\otimes_k E$ is rank one.  Then there is $t \in k^\times$ so that $(t^{-1}\omega)^2 = q(v_1)$ and $v_2 = tv_1^\flat$.  Consequently, $t^{-1}X = (t^{-1}\omega)v_1 + v_1^\flat.$\end{lemma}
\begin{proof} By equivariance and passing to the algebraic closure, we may assume $v_1 = (a,0,0,d)$ with $a,d \in k^\times$.  Write $v_2 = (a',b',c',d')$, so $X = (\omega a +a',b',c',\omega d + d')$.  Then since $X$ is rank one, $(b')^\# = (\omega a + a')c'$, $(c')^\# = (\omega d +d')b'$, and $3(\omega a + a')(\omega d + d') = (b',c')$.

From the first equation, we deduce $c' = 0$, and from the second that $b' = 0$.  We then get from the third equation that $\omega(ad'+a'd) + (a'd'+ \omega^2 (ad)) = 0$.  Since $a,d$ are in $k^\times$, there exists $t \in k$ so that $a' = -t a^2 d$, $d' = t ad^2$, and we then get that $\omega^2 = t^2(ad)^2$. Note that if $t$ is in some extension field of $k$, and $v_2 = t v_1^\flat$, then necessarily $t \in k$.  Since $(a,0,0,d)^\flat = (-a^2d, 0,0,ad^2)$ and $q((a,0,0,d))=(ad)^2$, we obtain $\omega^2 = t^2q(v_1)$ and $v_2 = tv_1^\flat$.  Because $1,\omega$ are assumed linearly independent, $t \neq 0$, and thus $t \in k^\times$.  The lemma follows. \end{proof}

Suppose $F$ is a field, $J$ is a cubic norm structure over $F$, $E$ is an \'{e}tale quadratic extension of $F$, and $X \in W_{J}\otimes_F E = W_{J_E}$ is rank $1$.  Say $X$ is \emph{admissible} if $X = (\omega v+v^\flat)/2$ for some $\omega \in E^\times \setminus F^\times$ and $v \in W_J$ with $\omega^2 = q(v)$.  Below we will need to understand when $\lambda X$ is admissible, for $\lambda \in E^\times$.
\begin{lemma}\label{cor:admis} Let the notation be as above, and assume $X = (\omega v + v^\flat)/2$ is admissible.  If $\mu \in E^\times$, $\mu = \alpha + \beta\omega$ with $\alpha, \beta \in F$, then 
\[\mu n_{E/F}(\mu)X = (n_{E/F}(\mu)\omega)(\alpha v+ \beta v^\flat)/2 + n_{E/F}(\mu)(\alpha v^\flat + q(v)\beta v)/2\] is admissible.  In particular, one has the identities
\begin{equation}\label{eq:qv}q(\alpha v+\beta v^\flat) = (\alpha^2 - q(v)\beta^2)^2q(v)\end{equation}
and
\begin{equation}\label{eq:vflat}(\alpha v + \beta v^\flat)^\flat = (\alpha^2 - q(v)\beta^2)(\alpha v^\flat + q(v)\beta v).\end{equation}
\end{lemma}
\begin{proof} The admissibility of $\mu n_{E/F}(\mu)X$ follows from the identities (\ref{eq:qv}) and (\ref{eq:vflat}), and these identities are familiar from the theory of Freudenthal triple systems.  For example, (\ref{eq:vflat}) is \cite[Lemma 4.5]{springer}.  The reader can check these identities themselves by observing that both sides of (\ref{eq:qv}) and (\ref{eq:vflat}) are equivariant under the action of $g \in H(W_J)$ with $\nu(g) =1$, and then checking them for $v = (a,0,0,d)$.  (This is essentially how (\ref{eq:vflat}) is proved in \cite{springer}.)  For $v = (a,0,0,d)$, 
\[\alpha v + \beta v^\flat = ((\alpha - \beta ad)a, 0,0,(\alpha + \beta ad)d),\]
and (\ref{eq:qv}), (\ref{eq:vflat}) follow.\end{proof}

We now give the more precise lifting law that applies when $A$ is an associative cubic norm structure.
\begin{theorem}\label{b1A} Let $A$ be an ACNS over $F$. Suppose $v \in W_A$ has $q(v) \neq 0$, $E = F[x]/(x^2-q(v))$, and denote by $\omega$ the image of $x$ in $E$.  Set $X(v) = (\omega v+v^\flat)/2$ in $W_{A} \otimes_F E$, and $\overline{X(v)} = (-\omega v + v^\flat)/2$, and recall the element $R(v) \in M_2(A)$ that satisfies $R(v)^2 = q(v)$.  Then for all column vectors $\eta \in A^2_{E}$, one has
\[\left(\left(\frac{\omega + R(v)}{2}\right)\eta\right)^{!} = \langle X(v), \eta^{!} \rangle \overline{X(v)}\]
in $W_{A} \otimes E$. Similarly, if $\ell \in A_E^2$ is a row vector, and $R_{r}(v) = J_2 R(v)J_2^{-1}$ (see Lemma \ref{lem:rightAct} and the discussion preceding it), then
\[\left(\ell\left(\frac{\omega + R_{r}(v)}{2}\right)\right)^{!} = \langle \ell^{!}, \overline{X(v)}\rangle X(v).\]
\end{theorem}

\begin{remark} In Theorem \ref{b1J}, which takes place in the context of arbitrary cubic norm structures, the element $X(v)$ was proved to be rank one. Theorem \ref{b1A} proves a stronger statement in the special case when the cubic norm structure is an ACNS.  Namely, Theorem \ref{b1A} gives an explicit decomposition of $X(v)$ as a ``pure tensor'', i.e., as an explicit scalar times the shriek $(!)$ of an element of $A_{E}^2$.\end{remark}

\begin{proof}[Proof of Theorem \ref{b1A}] The equality in question is polynomial in nature, and thus by prehomogeneity it suffices to check equivariance of both sides and to check it for particular $v$'s.  The equivariance is immediate using $R(gv) = gR(v)g^{-1}, g X(v) = X(gv)$ for $g \in SL_2(A)$.  A short computation verifies the identity when $v = v_0 = (a,0,0,d)$, in which case $R(v) = \mm{ad}{}{}{-ad}$. \end{proof}

\subsection{Orbits over a field}\label{subsec:B1F} In this subsection, we assume $F$ is a field, and $A$ is an associative cubic norm structure over $F$.  We parametrize the orbits of $\SL_2(A)$ on the elements of $W_A$ with rank $4$.

More precisely, suppose $v \in W_A$ has $q(v) \neq 0$.  Define $E = E_{q(v)} = F[x]/(x^2-q(v))$, denote $\omega$ the image of $x$ in $E$, and denote by $y \mapsto \overline{y}$ the nontrivial involution on $E$, so that $\overline{\omega} = -\omega$.  Now, let $X(v) = (\omega v+ v^\flat)/2$ in $W_A \otimes_F E = W_{A_E}$.  Associate to $v$ the invariants $q(v)$ and $\lambda(X(v)) \in E^\times/n(A_E^\times)$.  This latter invariant $\lambda$ was defined in Definition \ref{def:lambda}. The purpose of this subsection is to prove the following theorem.  Many of the instances of this theorem are closely related to results of Wright-Yukie \cite{wrightYukie} and Kable-Yukie \cite{kableYukie}, which were proven by means of Galois cohomology.

\begin{theorem}\label{B1thmField} For $v \in W_A$ with $q(v) \neq 0$, the invariant $\lambda(v):=\lambda(X(v))$ satisfies $n_{E/F}(\lambda) \in F^\times \cap n_{A_E}(A_E^\times)$.  The association $v \mapsto (E_{q(v)},\omega, \lambda(X(v)))$ defines a bijection between $\SL_2(A)$ orbits on the rank $4$ elements of $W_A$ and triples $(E,\omega,\lambda)$ of this kind, up to isomorphism.  An isomorphism of a triple $(E,\omega, \lambda)$ with $(E',\omega',\lambda')$ is an $F$-algebra isomorphism $\phi: E \rightarrow E'$ satisfying $\phi(\omega) = \omega'$ and $\phi(\lambda) = \lambda'$. \end{theorem}
\begin{proof} Note that we have 
\[\langle X(v), \overline{X(v)}\rangle = \frac{1}{4}\langle \omega v + v^\flat, -\omega v + v^\flat \rangle = \frac{2\omega}{4}\langle v, v^\flat\rangle = \omega q(v) \in E^\times.\]
It follows that $X(v)$ is nonzero in every component of $W_{A_E}$, and thus the invariant $\lambda(v)$ is defined.

That $n_{E/F}(\lambda(v))$ is a norm from $n(A_E)$ follows quickly from the lifting law.  Indeed, set $U = \frac{\omega + R(v)}{2}$ in $M_2(A_E)$, and suppose $\eta \in A_E^2$ is a column vector with $\langle X(v),\eta^{!} \rangle = \lambda(v) \in E^\times$.  Then since $\lambda(\overline{X(v)}) = \overline{\lambda(X(v))}$, there is a row vector $\ell \in A_E^2$ with $\langle \ell^{!}, \overline{X(v)}\rangle = \overline{\lambda(v)}$.  Hence
\[n(\ell J_2 U \eta) = \langle \ell^{!}, (U\eta)^{!}\rangle = \langle X(v),\eta^{!}\rangle \langle \ell^{!}, \overline{X(v)}\rangle = \lambda(v) \overline{\lambda(v)},\]
proving that $n_{E/F}(\lambda(v))$ is in $F^\times \cap n_{A_E}(A_E^\times)$. Here we have used Lemma \ref{lem:nRC} for the first equality and the lifting law Theorem \ref{b1A} for the second.

Suppose now $v' = g \cdot v$ with $g \in \SL_2(A)$.  Then $q(v) = q(v')$, and $(v')^\flat = (g \cdot v)^\flat = \nu(g) \left(g \cdot v^\flat\right) = g \cdot v^\flat$.  Hence $g\cdot X(v) = \omega g\cdot v + g \cdot v^\flat = X(v')$.  It follows that $v$ and $v'$ have isomorphic invariants.

Before we complete the proof of the theorem, we give some necessary intermediate results. \end{proof}
\begin{proposition}\label{prop:B1prim} Suppose $v \in W_A$ is non-degenerate, i.e., $q(v) \neq 0$.  Then, there exists $\lambda \in E^\times$ and a primitive vector $v_0$ in $W_2(A_E)$ for which $X(v) = \lambda v_0^{!}$.  Furthermore, if one sets
\[ v_1 = \left(\frac{\omega - R(v)}{2\omega}\right)v_0,\]
then $R(v)v_1 = -\omega v_1$ and $\lambda v_1^{!} = X(v)$. \end{proposition}
\begin{proof} We have already established the first part of the proposition.  Now, let $v_1$ be as in the statement.  Then $R(v) v_1 = -\omega v_1$, and by the lifting law,
\begin{align*} \lambda v_1^{!} &= -\frac{\lambda}{\omega^3}\left(\left(\frac{-\omega + R(v)}{2}\right)v_0\right)^{!} \\ &=-\frac{\lambda}{\omega^3} \langle \overline{X(v)},v_0^{!}\rangle X(v) \\ &= \frac{1}{\omega^3} \langle X(v), \overline{X(v)} \rangle X(v) \\ &= X(v). \end{align*}
This completes the proof. \end{proof}
We can now establish the fact that if $v, v' \in W_A$ have isomorphic invariants, then they are in the same $\SL_2(A)$ orbit, as stated in the following lemma.
\begin{lemma} Suppose $v, v' \in W_A^{rk = 4}$ have isomorphic invariants, so that in particular $q(v) = q(v') \neq 0$.  Then, there is $g \in \SL_2(A)$ with $g \cdot v = v'$. \end{lemma}
\begin{proof} Since $v, v'$ have isomorphic invariants, we may write $\lambda v_1^{!} = \lambda v_0^{!} = X(v)$ and $\lambda (v_1')^{!} = \lambda (v_0')^{!}= X(v')$, where $v_0, v_0' \in W_2(A_E)$ are primitive, $\lambda \in E^\times$, and $v_1, v_1'$ are as in the Proposition \ref{prop:B1prim}.

Since $v_0, v_0'$ are primitive, there exists $\widetilde{g} \in \GL_2(A_E)$ satisfying $\widetilde{g} v_0 = v_0'$.  Write $\widetilde{g} = g_1 - \omega g_\omega$, with $g_{*} \in M_2(A)$, and set $g = g_1 + g_\omega R(v)$.  Then $g \in M_2(A)$.  We have 
\begin{align*} g \cdot X(v) &= \lambda g \cdot v_0^{!} = \lambda g \cdot v_1^{!} \\ &= \lambda \left((g_1 + g_\omega R(v))v_1\right)^{!} \\ &= \lambda \left((g_1 - g_\omega \omega)v_1\right)^{!} \\ &= \lambda \widetilde{g} \cdot v_1^{!} = \widetilde{g} \cdot X(v) \\ &= X(v'). \end{align*}
Hence $g \cdot v = v'$ and $g \cdot v^\flat = (v')^\flat$.  Thus, since $q(v) = q(v')$,
\[2q(v) = 2q(v') = \langle v', (v')^\flat \rangle = \langle g \cdot v, g \cdot v^\flat \rangle = \det(g) \langle v, v^\flat \rangle = 2\det(g) q(v).\]
This proves $\det(g) = 1$, and in particular $g \in M_2(A)$ is actually invertible and in $\SL_2(A)$.  This completes the proof. \end{proof}

Finally, to complete the proof of Theorem \ref{B1thmField}, we must check that every $(E,\omega,\lambda)$ with $n_{E/F}(\lambda) \in F^\times \cap n(A_E^\times)$ arises as an actual invariant of some $v \in W_{A}$ of rank $4$.  To see that these invariants do all arise, first set $v = (0,1,0,d)$, with $d\in F^\times$.  Write $d^{1/2}$ for the image of $x$ in $F[x]/(x^2-d)$; one has $\omega = 2d^{1/2}$.  Then $X(v) = (1,d^{1/2},d,d^{3/2})$, and thus $\lambda(X(v)) = 1$.  Hence, all invariants $(E,\omega,1)$ appear.

Now, suppose $\lambda \in E^\times$ satisfies $n_{E/F}(\lambda) \in F^\times \cap n(A_E^\times)$.  Then from Corollary \ref{cor:admis}, we know $\lambda n(\lambda) X(v) = X(v')$ for some $v' \in W_A$ with $q(v') \neq 0$.  But $\lambda(X(v')) = \lambda n(\lambda)$, which represents the same class as $\lambda$ since $n(\lambda) \in n(A_E^\times)$.  Hence all classes $(E,\omega,\lambda)$ arise, as desired.  This completes the proof of Theorem \ref{B1thmField}.

\subsection{Balanced modules and integral orbits}\label{subsec:integralB1} In this subsection we parametrize the orbits of essentially $\GL_2(A)$ on $W_{A}$, when $A$ is an ACNS over the integral domain $R$.  More precisely, we prove Theorem \ref{twistThm1}, and thus Theorem \ref{introThm1} of the introduction.

In this subsection, we assume $R$ is an integral domain, with fraction field $F$ of characteristic $0$, and $R$ satisfies $\frac{x^2 + x}{2}$ is in $R$ for all $x \in R$.  The characterization of quadratic rings over $R$ is as discussed in paragraph \ref{sssec:QR}: these quadratic rings are of the form $S_{D} = R[y]/(y^2 - Dy + \frac{D^2-D}{4})$, where $D$ is congruent to a square modulo $4R$; we denote $\tau_D$ or $\tau$ the image of $y$ in $S_D$.  If $S$ is a quadratic ring over $R$, and $(1, \tau)$ is a basis of $S$ as an $R$-module, we say that $(1,\tau)$ is a \emph{good basis} of $S$ if $\tau$ satisfies a quadratic equation $\tau^2 - D\tau + \frac{D^2-D}{4}$ for some $D \in R$ (necessarily a square modulo $4R$).

The letter $A$ denotes an ACNS over $R$.  There is a cubic polynomial action of $\GL_2(A)$ on $W_A$, which we take to be the action on the right.  Set $G_1 = \GL_1(R) \times \GL_2(A)$.  Let $(\lambda,g)\in \GL_1(R) \times \GL_2(A)$ act on $v \in W_{A}$ by $v \mapsto \lambda^{-1} (v \cdot g)$.  We define $G \subseteq \GL_2(A)$ to be the pairs $(\lambda,g)$ where $\lambda^3 = \det_6(g)$.

\subsubsection{Preliminaries} We now give the various preliminaries that we will need, much of which is a repetition of facts and notations from above.

Suppose $v \in W_{A}$ has $q(v) = D \neq 0$.  We set $E = F[x]/(x^2-q(v))$ and let $\omega$ denote the image of $x$ in $E$.  Of course, $S_{D}$ is a subring of $E$.  We define $X(v) \in W_{A} \otimes_{R} E$ as
\[X(v,\omega) = X(v) = \frac{\omega v + v^\flat}{2}  = \frac{q(v) + \omega}{2}v + \frac{v^\flat -q(v)v}{2} = \tau_{D} v + \frac{v^\flat - q(v)v}{2}.\]
By using the formula for $v^\flat$ from Lemma \ref{lem:vvflat}, one checks easily that the assumption that $x^2 + x \in 2R$ for all $x \in R$ implies $\frac{v^\flat - q(v)v}{2}$ is in $W_A$.  Hence, $X(v) \in W_{A} \otimes S_{D}$.  Similarly, one defines $X(v,-\omega) = \overline{X(v)} = \frac{-\omega v + v^\flat}{2}$, which is also in $W_A \otimes S_{D}$.  One has the equality $\langle X(v,\omega), X(v,-\omega)\rangle = \omega^3$.

Set $S(v) \in \frac{1}{2} M_2(A)$ as in subsection \ref{sec:new}, and $R_r(v) = 2 J S(v)$, where $J = \mm{}{1}{-1}{}$.  One has $R_r(v)^2 = q(v) = D$.  We also record the equivariance property of $R_r(v)$: $R_r(v \cdot g) = \det(g) g^{-1}R(v)g$ for $g \in \GL_2(A)$.  

We need two other pieces of notation: if $\ell = (s,t) \in A^2$ is a row vector, set $\ell^{!} = (n(s), s^\#t, t^\#s, n(t))$ in $W_A$.  Furthermore, set
\[\epsilon = \epsilon(v) = \frac{\omega + R_{r}(v)}{2\omega} = \frac{1}{2} + \frac{1}{2\omega} R_r(v).\]
Finally, note that
\[ \omega \epsilon = \frac{\omega + R_r(v)}{2} = \frac{\omega + q(v)}{2} + \frac{R_r(v)-q(v)}{2} = \tau_{D} + \frac{R_r(v)-q(v)}{2} = \tau_{D} -\frac{D}{2} + JS(v).\]
Using again that $x^2 + x \in 2R$ for all $x \in R$, one checks that $\frac{R_r(v)-q(v)}{2} = -\frac{D}{2} + JS(v)$ is in $M_2(A)$.

We these notations behind us, we can restate the lifting law, Theorem \ref{b1A}: For all row vectors $\ell \in A_{F}^2$,
\[ \omega^3 (\ell \epsilon(v))^{!} = \langle \ell^{!}, X(v,-\omega) \rangle X(v)\]
in $W_A \otimes E$.  If $\ell \in A^2$, then both sides are in $W_A \otimes S_{D}$.

Set $\Omega = \Omega(v) = \frac{q(v) + R_r(v)}{2}$.  Then $\Omega(v) \in M_2(A)$, and satisfies $\Omega^2 - D\Omega + \frac{D^2 -D}{4} = 0$.  Thus $\Omega$ defines an action of $S_D$ on the column vectors $A^2$.

\subsubsection{Balanced modules} We consider a class of $S_D \otimes A = S \otimes A$ modules $I \subseteq A_E$ as follows. 
\begin{definition} If $I \subseteq E \otimes A$, we say that $I$ is an $S\otimes A$-\textbf{fractional ideal} if
\begin{itemize}
\item $I$ is closed under left multiplication by $S$ and right multiplication by $A$;
\item $I$ is free of rank $2$ as an $A$-module;
\item $I_F = E \otimes A$ is free of rank one as an $A_{E}$-module. \end{itemize}
If $(\tau,1)$ is a good basis of $S$, and $I = b_1A \oplus b_2A$ with $b_1, b_2 \in E\otimes A$, define the \textbf{norm} of $I$ $N(I;\tau, (b_1,b_1)) = \det_6(g)$ where $g \in M_2(A_F)$ is the unique matrix with $(b_1,b_2) = (\tau\otimes 1,1\otimes 1)g$. Explicitly, if $b_1 = \tau m_{11} + m_{21}$ and $b_2 =\tau m_{12}+ m_{22}$, with $m_{ij} \in A_{F}$, define $N(I;\tau,(b_1,b_2)) = \det_6 \mm{m_{11}}{m_{12}}{m_{21}}{m_{22}}$.  This definition does depend on the good basis $(\tau,1)$ of $S$ and on the $A$-basis $(b_1,b_2)$ of $I$.

If $\beta \in E^\times$, and $I,\tau, (b_1,b_2)$ are as above, we say the data $(I,(b_1,b_2),\tau,\beta)$ is \textbf{balanced} if $\beta^{-1}(x,y)^{!} \in W_{A} \otimes S$ for all $x, y \in I$ and $N(I;\tau,(b_1,b_2)) = N_{E/F}(\beta)$. \end{definition}

Suppose $(I,(b_1,b_2),\beta,\tau)$ is data as above.  Set
\[X(I,b,\beta) := \beta^{-1} b^{!} \in W_{A_E} = W_A \otimes_{R} E.\]
It is clear that $X(I,b,\beta) \in W_{A} \otimes S$ if and only if $\beta^{-1}(x,y)^{!} \in W_{A}\otimes S$ for all $x,y \in I$.  We next define an equivalence relation on data $(I,b,\beta)$.
\begin{definition} The triple $(I,b,\beta)$ is said to be \emph{equivalent} to the triple $(I',b',\beta')$ if there exists $x \in A_E^\times$ so that $I' = xI$, $b' = xb$, and $\beta' = n_{A_E}(x) \beta$. \end{definition}

It is clear that if $(I,b,\beta)$ is equivalent to $(I',b',\beta')$, then $X(I,b,\beta) = X(I',b',\beta')$.  The converse is also true, as the following lemma proves.
\begin{lemma}\label{equivModA} If $X(I,b,\beta) = X(I',b',\beta')$, then there exists $x \in A_E^\times$ so that $b' = xb$ and $\beta' = n_{A_E}(x)\beta$. \end{lemma}
\begin{proof} Since $b = (\tau,1)g$, $b$ is primitive, and thus there is $h \in \GL_2(A_E)$ so that $b_0 = bh = (1,0)$.  Set $b_1 = b'h = (s,t)$.  Then $\beta^{-1} b_0^{!} = (\beta')^{-1} b_1^{!}$ yields that $\beta' = n(s)\beta$, so $s \in A_E^\times$, and then that $t = 0$.  Hence $\beta' = n(s) \beta$ and $b_1 = s b_0$.  The lemma follows. \end{proof}

The following equivalent criterion for a module to be balanced is useful.
\begin{lemma}\label{equivBal} Suppose $I$ is a fractional $S \otimes A$-ideal, $b = (b_1,b_2)$ is an $A$ basis of $I$, and $\beta \in E^\times.$  Suppose moreover that $(1,\tau)$ is a good basis of $S$, and set $\omega = \tau - \overline{\tau}$.  Recall $X = X(I,b,\beta) = \beta^{-1}b^{!}$.  Then $\langle X, \overline{X}\rangle = \omega^3 \frac{\det_6(g)}{N_{E/F}(\beta)}$.  It follows that if $X \in W_A \otimes S$, the triple $(I,b,\beta)$ is balanced if and only if $\langle X, \overline{X}\rangle = \omega^3$. \end{lemma}
\begin{proof} We have
\[\langle X, \overline{X}\rangle = \frac{1}{N_{E/F}(\beta)} \langle (\tau,1)^{!} \cdot g, (\overline{\tau},1)^{!} \cdot g \rangle = \frac{\det_6(g)}{N_{E/F}(\beta)}(\tau-\overline{\tau})^3 = \omega^3 \frac{\det_6(g)}{N_{E/F}(\beta)}.\]
The lemma follows. \end{proof}

The elements of $W_A^{open}$ parametrize pairs $(S,M)$, where $S$ is a quadratic ring over $R$ and $M$ is an equivalence class of based balanced modules $(I,b,\beta)$.  First, a technical lemma.
\begin{lemma}\label{goodL} Suppose $v \in W_A$ is rank $4$ and recall the element $\epsilon = \epsilon(v)$ of $M_2(A_{E})$.  Then, there is $\eta \in A_F^2$ (column vectors) and $\ell \in A_F^2$ (row vectors) so that $\ell \epsilon \eta$, in $A_E$, is actually in $A_E^\times$. \end{lemma}
\begin{proof}  For all row vectors $x$ and column vectors $y$ in $A_E^2$, one has $n_{A_E}(x J_2 y) = \langle x^{!},y^{!} \rangle.$  Thus
\[ \omega^3 n_{A_E}(\ell \epsilon \eta) = \langle \omega^3 (\ell \epsilon)^{!}, (J_2^{-1}\eta)^{!} \rangle = \langle \ell^{!}, X(v,-\omega)\rangle \langle X(v,\omega), (J_2^{-1}\eta)^{!} \rangle.\]
This latter term can be made to lie in $E^\times$ for appropriate choices of $\ell, \eta$, since $\langle X(v,\omega), X(v,-\omega) \rangle \in E^\times$. To see this, first note that since $\langle X(v,\omega), X(v,-\omega) \rangle \in E^\times$ there is $\ell_1 \in A_{E}^2$ with $\langle \ell_1^{!}, X(v,\omega)\rangle \in E^\times$.  Now, set $\ell = \tr_{E/F}(\ell_1 \epsilon)$, taking the trace $\tr_{E/F}(\cdot)$ individually on each of the two entries of $\ell_{1}\epsilon$.  Then one computes $\ell \epsilon = \ell_1 \epsilon$.  Hence
\[\langle \ell^{!}, X(v,-\omega) \rangle X(v) = \omega^{3}(\ell \epsilon)^{!} = \omega^3 (\ell_1 \epsilon)^{!} = \langle \ell_1^{!}, X(v,-\omega) \rangle X(v),\]
from which it follows that $\langle \ell^{!}, X(v,-\omega) \rangle \in E^\times$.  Similarly one finds an $\eta \in A_{F}^2$ with $\langle X(v,\omega), (J^{-1}\eta)^{!} \rangle \in E^\times$.  The lemma follows. \end{proof}

Given a row vector $\ell \in A_F^2$, one defines a map $\Phi^{\ell}: A^2 \rightarrow A_E$ from the column vectors $A^2$ to $A_E$ via $\Phi^\ell(\eta) = \ell \epsilon \eta$.  Note that $\Phi^{\ell}$ is an $S\otimes A$ module map, with $S$ acting on the left of the column vectors $A^2$ via $\Omega(v)$ and $A$ acting on the right of them by multiplication $A\otimes A \rightarrow A$.  The previous lemma implies that $\ell$ can be chosen so that $\Phi^{\ell}(A_F^2)$ contains a unit of $A_E$.  By dimension count, it follows that for this $\ell$, $\Phi^{\ell}$ is an $A_E$-module isomorphism $A_F^2 \rightarrow A_E$.

Choose any such $\ell$ as above, and set $I(\ell) = \Phi^{\ell}(A^2)$, $b(\ell) = \ell \epsilon = (b_1,b_2)$, so that $b_1 = \Phi^{\ell}((1,0)^t)$ and $b_2 = \Phi^{\ell}((0,1)^t)$.  Set $\beta(\ell) = \omega^{-3} \langle \ell^{!}, X(v,-\omega)\rangle$.  One has $\beta(\ell) \in E^\times$ (from the proof of Lemma \ref{goodL}, for example).  We claim $(I(\ell),b(\ell),\beta(\ell))$ is balanced.  Indeed, the lifting law gives $\beta(\ell)^{-1} (b(\ell))^{!} = X(v)$ is in $W_A \otimes S$, and thus that $(I(\ell),b(\ell),\beta(\ell))$ is balanced follows from Lemma \ref{equivBal}.

From the above remarks, and an application of the lifting law and Lemma \ref{equivModA}, we see that $v \mapsto (I(\ell),b(\ell),\beta(\ell))$ yields an equivalence class of based balanced modules.  There is also a map in the other direction 
\[(I,b,\beta) \mapsto X(I,b,\beta) = \tau v + v' \mapsto v \in W_A.\]
\begin{lemma}\label{admisLem} If $(I,b,\beta)$ as above is balanced, and $X(I,b,\beta) = \tau v + v'$, then $q(v) = \omega^2$ and in particular $v$ is rank $4$. \end{lemma}
\begin{proof} Since $(I,b,\beta)$ is a balanced triple, one has $\langle X(I,b,\beta), \overline{X(I,b,\beta)} \rangle = \omega^3$.  This implies that $X(I,b,\beta) = \frac{\omega v + v_2}{2}$ is \emph{admissible}, meaning that $q(v) = \omega^2$ and $v_2 = v^\flat$.  

To see this, note that by construction, $X = X(I,b,\beta) = (\omega v + v_2)/2$ is rank one.  We have $\langle X, \overline{X}\rangle = \frac{\omega}{2}\langle v, v_2\rangle = \omega^3$.  Since $\omega v = X(v) - \overline{X(v)}$ is the difference of two rank one elements with pairing landing in $E^\times$, $v$ is rank $4$.  Thus we may apply Lemma \ref{uniqueLifts} to obtain $\omega^2 = t^2 q(v)$ and $v_2 = tv^{\flat}$ for some $t \in F^\times$.  We obtain
\[\omega^3 = \langle X(v), \overline{X(v)}\rangle = \frac{\omega}{2}\langle v, v_2\rangle = \frac{\omega}{2}\langle v, tv^\flat\rangle = t\omega q(v) = \omega^3/t.\]
Hence $t = 1$ and thus $X$ is admissible, as claimed.
\end{proof}

From the lifting law, the composition of the two maps is the identity, if one starts with $v$ in $W_A$.  We will check that the composition of the two maps is also the identity if you start with a based balanced module $(I,b,\beta)$.
\begin{theorem} The maps $v \mapsto (I,(b_1,b_2),\tau,\beta)$ and $(I,(b_1,b_2),\tau,\beta) \mapsto v$ define inverse bijections between $v \in W_{A}$ with $q(v) = D \neq 0$ and data $(I,(b_1,b_2),\tau,\beta)$ up to equivalence, with $\tau^2 - D\tau + \frac{D^2-D}{4} = 0$. This bijection is equivariant for the action of $G_1$, with $(\lambda,g) \in \GL_1(R) \times \GL_2(A)$ acting on the data via
\[((b_1,b_2),\omega,\beta;\ell) \mapsto ((b_1,b_2)g, \det_6(g) \lambda^{-2} \omega, \lambda^{3}\det_6(g)^{-1}\beta; \ell g).\]
Here $((b_1,b_2),\tau,\beta;\ell)$ means that the data is computed using the row vector $\ell \in A_{F}^2$.
\end{theorem}
\begin{proof} From Lemma \ref{admisLem}, if $X(I,b,\beta) = \tau v + v'$, then $X(I,b,\beta) = X(v,\omega)$ and $q(v) = D$. Now, take $\ell \in A_F^2$ so that $\beta(\ell) = \omega^{-3} \langle \ell^{!}, X(v,-\omega)\rangle$ is in $E^\times$.  Then $(I(\ell), b(\ell),\beta(\ell))$ is a balanced triple.  We have
\[\beta(\ell)^{-1} (\ell \epsilon(v))^{!} = X(v) = X(I,b,\beta) = \beta^{-1} b^{!}.\]
It follows from Lemma \ref{equivModA} then that $b = x \ell \epsilon(v)$ for some $x \in A_E^\times$.

We check that the action maps agree.  Let $R_{I} \in M_2(A)$ denote the action of $\omega$ on $I$.  Then, since $b_1, b_2$ is a basis for the free $A$-module $I$, $R_{I}$ is the unique matrix in $M_2(A)$ satisfying $\omega b = bR_{I}$.  But we have
\[bR_{I} = \omega b = \omega x \ell \epsilon(v) = x \ell \epsilon(v) R(v) = bR(v).\]
It follows that $R_{I} = R(v)$, so that the action maps $R_{I}$, $R(v)$ agree, so we write $\epsilon = \epsilon(v)$ for the common element of $M_2(A_E)$.

Now, set $b_0 = \tr_{E/F}(b) \in A_F^2$.  Then $b_0 \epsilon = b$, as is immediately checked.  Thus by the lifting law,
\[\beta X(v) = \beta X(I,b,\beta) = b^{!} = (b_0 \epsilon)^{!} = \beta(b_0) X(v).\]
It follows that $\beta(b_0) = \beta \in E^\times$.  Thus the triple $(I(\ell),b(\ell),\beta(\ell))$ is equivalent to the triple $(I(b_0),b(b_0),\beta(b_0))$, and this latter triple is equal to $(I,b,\beta)$.  Thus the composition $(I,b,\beta) \mapsto v \mapsto (I',b',\beta')$ is the identity as well.  

Now for the equivariance.  Set $h = (\lambda, g) \in G_1$. From $R_r( \lambda^{-1} (v \cdot g)) = \lambda^{-2} \det_6(g) g^{-1}R_{r}(v)g$, one computes $\epsilon(v \cdot h) = g^{-1}\epsilon(v) g$. From the definition of the action of $h$ on $\omega$, one gets $X(\omega,v) \mapsto \lambda^{-3}\det_6(g)(X(\omega,v) \cdot g)$.  The results on $(b_1,b_2)$ and $\beta$ follow.  This completes the proof of the theorem. \end{proof}

The following corollary is immediate, and implies Theorem \ref{twistThm1}.
\begin{corollary}\label{cor:equivG1}  The action of $G \subseteq G_1$ on the data $((b_1,b_2),\omega,\beta)$ is
\[((b_1,b_2),\omega,\beta;\ell) \mapsto ((b_1,b_2)g, \lambda \omega, \beta; \ell g).\]
\end{corollary}

\section{The lifting law for $J \oplus J$}\label{bII}
Suppose $C$ is an associative composition algebra over $F$.  The purpose of this section is to parametrize the orbits of $\SL_3(C)$ on $H_3(C)^2$,  and to solve the analogous problem over a subring $R$ of $F$.  Over a field, a closely-related parametrization problem was solved in \cite{taniguchi,kableYukie,wrightYukie}, albeit by very different methods.  The integral theory involves a notion of ``balanced" module, following \cite{bhargavaII}, and also uses some ideas from \cite{wood}.  We begin by proving a general lifting law, that applies to the prehomogeneous vector space $J\oplus J = V_2 \otimes J$ for an arbitrary cubic norm structure $J$.  Then when $J = H_3(C)$, we prove a refined lifting law.  We use the refined lifting law to solve the orbit parametrization problems.  

The lifting laws are proved in subsection \ref{subsec:LLB2}.  In subsection \ref{subsec:B2F}, we parametrize the orbits over a field $F$, and in subsection \ref{subsec:integralB2}, we parametrize the orbits over a subring $R$ of $F$.  The main result of this section is Corollary \ref{cor:equivG2}, which implies Theorem \ref{twistThm2}.

\subsection{The lifting law}\label{subsec:LLB2} In this subsection we give the lifting law for the prehomogeneous vector space $J^2$, where $J$ is a cubic norm structure over the ring $R$.  We also give a more explicit lifting law when $J = H_3(C)$ for an associative composition ring $C$ over $R$.

\subsubsection{The representation space} We begin with a description of the representation space $J^2 = V_2 \otimes J$.  We write a typical element of $J^2$ as $(A,B)$, and then the (left) $\GL_2$ action is
\[g \cdot \left(\begin{array}{c} A \\ B \end{array}\right) = \left(\begin{array}{cc}p&q\\r&s\end{array}\right)\left(\begin{array}{c} A\\ B \end{array}\right) = \left(\begin{array}{c} pA+qB\\ rA+sB \end{array}\right)\]
if $g = \mm{p}{q}{r}{s}$.

Recall the group $M_J$, which is the group of $(g,\lambda) \in \GL(J) \times \GL_1$ satisfying $n(gX) = \lambda n(X)$ for all $X \in J$.  Then $M_J$ also acts on $J^2$ via $g \cdot (A,B) = (gA, gB)$, and this action commutes with the action of $\GL_2$ to give a $\GL_2 \times M_J$ action on $J^2$, under which $J^2$ becomes a prehomogeneous vector space.

\subsubsection{The discriminant invariant} To the element $(A,B)$ of $J^2$, we assign the binary cubic form
\[f_{(A,B)}(x,y) = n_J\left((x,y)\left(\begin{array}{c} A\\B\end{array}\right)\right) = n_J(Ax+By) = n(A)x^3 + (A^\#,B)x^2y + (A,B^\#)xy^2 + n(B)y^3,\]
and the degree $12$ discriminant polynomial $Q((A,B)) := Q(f_{(A,B)})$; see subsubsection \ref{ssc:bcfs}.  Here $n_J: J \rightarrow R$ is the cubic norm on $J$.  If $g \in \GL_2$, then $f_{g\cdot (A,B)} = \det(g) (g \cdot f_{(A,B)})$ (recall that the action $g \cdot f$ is twisted by $\det(g)^{-1}$.)  Hence $Q(g \cdot (A,B)) = \det(g)^6 Q((A,B)).$ If $m = (g,\lambda)\in M_J$, then $f_{(m \cdot (A,B))}(x,y) = \lambda f_{(A,B)}(x,y)$ and hence $Q(m\cdot(A,B)) = \lambda^4 Q((A,B))$.

Associated to the binary cubic $f_{(A,B)}$ is a cubic $R$-ring $T$ with good basis $(1,\omega,\theta)$, and thus to every pair $(A,B)$ in $J^2$ is associated a cubic ring with good basis.

\subsubsection{The lifting law for $J$} We now give the lifting law for general cubic norm structures $J$.  Suppose $(A,B) \in J^2$ has associated based cubic ring $T, (1,\omega,\theta)$.  Consider $J_T = J \otimes_R T$, and put on $J_T$ the cubic norm structure over $T$, with norm and adjoint extended $T$-linearly from $J$ to $J_T$.  That is, $(U\otimes \lambda)^\# = U^\# \otimes \lambda^2$ and $n(U\otimes \lambda) = \lambda^3 n(U)$, for $U \in J$ and $\lambda \in T$. Set
\[X= X(A,B,\omega,\theta):=-A\theta + B\omega + A^\#\times B^\#,\]
an element of $J_T$.  The first part of the lifting law below will be that $X$ is rank one in $J_T$.

\begin{remark}\label{rmk:katoYukie} With $X = X(A,B,\omega,\theta)$ as above, we have
\[\tr_{T/R}(X) = 3 A^\# \times B^\# - (A,B^\#)A - (A^\#,B)B.\]
In Kato-Yukie \cite{katoYukie}, the authors parametrize the orbits of $\GL_2 \times M_J$ on $J\oplus J$, where $J = H_3(C)$ with $C$ an octonion algebra.  Associated to the non-degenerate pair $(A,B)$, Kato-Yukie assign an embedding $L :=T\otimes_{R}F \hookrightarrow J^{\dagger}$, where $J^\dagger$ is an isotope of $J$.  This isotope $J^\dagger$ of \cite{katoYukie} is precisely the isotope specified by the element $\tr_{T/R}(X)$ of $J$. \end{remark}

We will define an element $Y= Y(A,B,\omega,\theta)$ of $J_T^\vee$ associated to $A,B,\omega,\theta$, which will also turn out to be rank one.  To do this, first define an $R$-bilinear map $\times_T: J_T \otimes J_T \rightarrow J_T^\vee$ via 
\[(U_1 \otimes \lambda_1) \times_T (U_2 \otimes \lambda_2) = (U_1 \times U_2) \otimes (\lambda_1 \times \lambda_2).\]
We set $Y = Y(A,B,\omega,\theta) := \frac{1}{2} X(A,B,\omega,\theta) \times_T X(A,B,\omega,\theta)$, an element of $J_T^\vee$.

More precisely, if $V$ is any $R$-module, the map $\times: J \otimes_R J \rightarrow J^\vee$ extends to a map 
\begin{equation}\label{eqn:JVV}(J \otimes_R V) \otimes_R (J \otimes_R V) \rightarrow J^\vee \otimes_R (V \otimes_R V).\end{equation}
Then $\times: J_T \otimes J_T \rightarrow J_T^\vee$ is the composition of the map (\ref{eqn:JVV}) for $V = T$ with the multiplication map $T\otimes_R T \rightarrow T$, and $\times_T$ is the composition of the map (\ref{eqn:JVV}) with the map $\times: T \otimes_R T \rightarrow T$.  We similarly have maps $\#: J_T^\vee \rightarrow J_T$ and $\times_T: J_T^\vee \otimes J_T^\vee \rightarrow J_T$.

Here is the lifting law.
\begin{theorem}\label{B2Jlift} The elements $X(A,B,\omega,\theta) \in J_T$ and $Y(A,B,\omega,\theta) \in J_T^\vee$ are rank one.  Furthermore, the pairing $(X(A,B,\omega,\theta),Y(A,B,\omega,\theta))$, a priori in $T$, lies in $R$ and one has 
\[(X(A,B,\omega,\theta),Y(A,B,\omega,\theta)) = Q((A,B)).\]
\end{theorem}

Before proving the lifting law, we record the following lemma, which will be used in the proof of Theorem \ref{B2Jlift}.
\begin{lemma}\label{comAdj} Define two $4$-linear maps $J_T \otimes J_T \otimes J_T \otimes J_T \rightarrow J_T$ via 
\[m_1(w,x,y,z) = (w \times x) \times_{T} (y \times z) + (w \times z) \times_{T} (y \times x)\]
and 
\[m_2(w,x,y,z) = (x \times_T z) \times (w \times_T y).\]
Then $m_1 = m_2$. Consequently, if $x \in J_T$, then $(x \times_T x)^\# = 4 x^\# \times_T x^\#$.\end{lemma}
\begin{proof} The first part of the lemma follows immediately from the identity
\[(\lambda_1 \lambda_2) \times (\lambda_3\lambda_4) + (\lambda_1 \lambda_4)\times (\lambda_3\lambda_2) = (\lambda_2 \times \lambda_4)(\lambda_1 \times \lambda_3)\]
for $\lambda_1, \lambda_2, \lambda_3, \lambda_4 \in T$.

The second claim of the lemma is a direct consequence of the first, as $(x \times_{T} x)^\# = \frac{1}{2} m_2(x,x,x,x)$ while $4 x^\#\times_{T} x^\# = \frac{1}{2}m_1(x,x,x,x)$.\end{proof}

We compute $Y(A,B,\omega,\theta)$ more explicitly.
\begin{lemma} Suppose $f_{(A,B)}(x,y) = ax^3 + bx^2y+cxy^2+dy^3$, and that this binary cubic form corresponds to the based cubic ring $T$, $(1,\omega,\theta)$.  Then
\begin{equation}\label{Yexp}Y(A,B,\omega,\theta) = \left(c\theta - 3d\omega - c^2 + bd\right)A^\# + \left(-b\theta + c\omega + bc - 3ad\right) A\times B + \left(3a\theta - b\omega -b^2+ac\right)B^\#.\end{equation}
\end{lemma}
\begin{proof} One has
\begin{align*} Y(A,B,\omega,\theta) &= \frac{1}{2} X(A,B,\omega,\theta) \times_{T} X(A,B,\omega,\theta) \\ &= \frac{1}{2} (-A \theta + B\omega + A^\# \times B^\#) \times_{T} (-A \theta + B\omega + A^\# \times B^\#) \\ &= 2 A^\# \theta^\# - (A \times B)(\omega \times \theta) - (A \times (A^\# \times B^\#))(\theta \times 1) + 2B^\# \omega^\# \\ &\quad + (B \times (A^\# \times B^\#))(\omega \times 1) + 2 (A^\# \times B^\#)^\# \\ &= 2 A^\#(-d\omega)- (A \times B)(ad - bc + b\theta - c\omega) - (aB^\# + cA^\#)(c-\theta) + 2B^\# (a\theta) \\ &\quad - (dA^\#+bB^\#)(b+\omega) + 2(-ad (A\times B) + bd A^\# + acB^\#).\end{align*}
Here we have used the identities $W \times (W^\# \times Z) = n(W)Z + (W,Z)W^\#$ and $(u\times v)^\# + u^\# \times v^\# = (u,v^\#)u + (u^\#,v)v$.  Hence we obtain the formula (\ref{Yexp}). \end{proof}

\begin{proof}[Proof of Theorem \ref{B2Jlift}] That $X$ is rank one follows by direct computation from the multiplication table of $T$.  Indeed, one has
\begin{align*} X^\# &= (A^\# \times B^\#)^\# + ad A \times B - bd A^\# - ac B^\# \\ &\quad + \left(c A^\# + a B^\# - A \times (A^\# \times B^\#)\right)\theta + \left(-dA^\# - bB^\# + B \times (A^\# \times B^\#)\right)\omega.\end{align*}
This then gives $0$ by the general identities $W \times (W^\# \times Z) = n(W)Z + (W,Z)W^\#$ and $(u\times v)^\# + u^\# \times v^\# = (u,v^\#)u + (u^\#,v)v$ used above.

That $Y$ is rank one then follows from Lemma \ref{comAdj}.  We must still evaluate $(X,Y)$. This can be proved by a somewhat tedious, but entirely straightforward computation using the multiplication table of $T$.\end{proof} 

The following two lemmas provide a complement to the lifting law.  They explain the sense in which the rank one ``lift'' $X$ of $(A,B)$ is unique.
\begin{lemma}\label{B2uniqueDelta} Suppose $T, (1,\omega,\theta)$ is a good-based cubic ring, $A, B \in J$, and $X' = -A\theta + B\omega + C$ is rank one in $J_T$.  Suppose $n(Ax+By)$ is not identically $0$ and $\delta \in J$ is such that $X' +\delta$ is rank one.  Then, $\delta =0$.\end{lemma}
\begin{proof} We have $0 = (X'+\delta)^\# = X' \times \delta + \delta^\#$.  Hence, since $1, \omega, \theta$ are $F$-linearly independent, $A \times \delta = 0 = B \times \delta$.  Hence $(Ax + By) \times \delta = 0$ in $J\otimes_{F}F(x,y)$, where $F(x,y)$ is the field of rational functions in $x,y$.  But $Ax+By$ is invertible in $J\otimes_{F}F(x,y)$, since $n(Ax+By) \neq 0$.  Thus the lemma follows from the fact that $z \times \delta = 0$ implies $\delta = 0$, when $z$ is invertible.

This latter fact is a consequence of the identity 
\[z^\# \times (z \times \delta) = n(z)\delta + (z^\#,\delta)z = n(z)\delta + \frac{1}{2}(z, z \times \delta)z.\]
\end{proof}

\begin{lemma}\label{B2lem:uniqueLift}  Suppose $T, (1,\omega,\theta)$ is a good-based cubic ring, $A, B \in J$, and $X' = -A\theta + B\omega + C$ is rank one in $J_T$.  Set $f(x,y) = ax^3 + bx^2y + cxy^2 + dy^3$ the binary cubic form associated to $\omega, \theta$, $f'(x,y) = n(Ax+By)$ the binary cubic associated to the pair $(A,B)$.  Assume $A,B, C$ are linearly independent in $J$, and $f(x,y) \neq 0$. Then there exists $t \in F$ so that $f'(x,y) = tf(x,y)$.\end{lemma}
In other words, if $-A\theta + B\omega + C$ is rank one in $J_T$, then the data $(T, (1,\omega,\theta))$ is essentially the based cubic ring associated to $(A,B)$.

\begin{proof}[Proof of Lemma \ref{B2lem:uniqueLift}] To check that $f' = tf$, it suffices to first act by $\SL_2(F)$ on $(\omega_0,\theta_0)$ and $(A,B)$ so that the coefficients of $x^3$ and $y^3$ in $f(x,y)$ are both not zero and $X'$ stays invariant.  (Every nonzero orbit of $\SL_2(F)$ on binary cubics contains an element with $ad \neq 0$.) So by using the action of $\SL_2(F)$, we may assume $ad \neq 0$.

Now, using the multiplication table for $(1,\omega,\theta)$ and taking $(X')^\#$, one finds 
\begin{align*} (X')^\# &= \left(cA^\# + aB^\# - A \times C\right)\theta + \left(-dA^\# - bB^\# + B \times C\right)\omega \\ &\quad + \left(C^\# - bd A^\# - ac B^\# + (ad) (A \times B)\right)1.\end{align*}
Since $1,\omega,\theta$ are linearly independent, each of the terms in parentheses is $0$.  Consequently,
\begin{align} \nonumber n(A)C + (A^\#,C)A &= A^\# \times (A \times C) \\ \nonumber &= A^\# \times (cA^\# + aB^\#) \\&= \label{eqn:2cn} 2cn(A)A + a A^\# \times B^\# \end{align}
and similarly
\begin{equation}\label{eqn:nBC} n(B)C + (B^\#,C)B = 2bn(B)B + d A^\# \times B^\#.\end{equation}
Multiplying (\ref{eqn:2cn}) by $d$, (\ref{eqn:nBC}) by $a$ and subtracting, one obtains
\begin{equation}\label{eq:ABC}(n(A)d - n(B)a)C + d((A^\#,C) - 2cn(A))A + a(2bn(B) - (B^\#,C))B = 0.\end{equation}
Since $A,B,C$ are assumed linearly independent in $J$, we deduce $n(A)d = n(B)a$.  Pairing the coefficient of $\theta$ in $(X')^\#$ with $A$, we get
\[2cn(A) + a(A,B^\#) = 2(A^\#,C)\]
and similarly pairing the coefficient of $\omega$ in $(X')^\#$ with $B$ we get
\[3bn(B) + d(A^\#,B) = 2(B^\#,C).\]
Plugging these into the coefficients of $A$ and $B$ in (\ref{eq:ABC}) we get
\[2d(a(A,B^\#)-cn(A)) = 0, \qquad a(bn(B) - d(A^\#,B)) = 0.\]
Since $ad \neq 0$, this completes the proof. \end{proof}

\subsubsection{The lifting law for $J = H_3(C)$} Suppose now that $C$ is a composition ring over $R$, and that $C$ is associative.  Then $J = H_3(C)$ is a cubic norm structure, and $\GL_3(C)$ acts on $J$ via $m \cdot h = mhm^*$, for $m \in \GL_3(C)$ and $m^*$ the conjugate transpose of $m$.  We now extend the lifting law Theorem \ref{B2Jlift} to a more precise lifting law.

Let $(A,B) \in J^2$ be fixed.  Define a map $S_r:T \rightarrow M_3(C)$ by $S_r(1) = 1$, $S_r(\omega) = -A^\#B$, $S_r(\theta) = B^\#A$ and extending $R$-linearly, and another $R$-linear map $S_{\ell}:T\rightarrow M_3(C)$ via $S_{\ell}(\lambda) = S_{r}(\lambda)^*$ for $\lambda \in T$.  The maps $S_r, S_{\ell}$ are $R$-algebra homomorphisms, and in fact $X = X(A,B,\omega,\theta)$ and $Y = Y(A,B,\omega,\theta)$ are eigenvectors for the action of $S_{?}(T)$.  This claim is verified in the following proposition.  The maps $S_{?}$ are closely related to constructions in \cite{wood}.
\begin{proposition}\label{prop:Srevec} The maps $S_r, S_{\ell}: T \rightarrow M_3(C)$ are $R$ algebra homomorphisms, and if $\lambda \in T$, then $S_{\ell}(\lambda)X = \lambda X = XS_{r}(\lambda)$. \end{proposition}
\begin{proof} First we check that $S_r$ is a ring homomorphism.  It is clear that $S_{r}(\omega)S_{r}(\theta) = -ad = S_{r}(\theta)S_{r}(\omega)$.  Next we check that $S_{r}(\omega)^2 = S_{r}(\omega^2)$.  This amounts to the identity
\begin{equation}\label{Somega2} A^\#BA^\#B = -n(A)(A,B^\#) + n(A)B^\#A + (A^\#,B)A^\#B.\end{equation}
Equation (\ref{Somega2}) is a polynomial identity, and thus by Zariski density it suffices to prove it when $n(A) \neq 0$.  When this is the case, (\ref{Somega2}) is equivalent to
\begin{equation}\label{Somega2:2} BA^\#B + (A,B^\#)A = AB^\#A + (A^\#,B)B\end{equation}
by multiplying (\ref{Somega2}) on the left by $A$, dividing by $n(A)$, and rearranging.  But now both sides of (\ref{Somega2:2}) are equal to $(A\times B)^\#$, since one has $(x\times y)^\# = (x^\#,y)y + U_{x}y^\#$, and $U_{x}z = xzx$. The proof that $S_{r}(\theta)^2 = S_r(\theta^2)$ is essentially identical.

Now, since $T$ is commutative and $S_{\ell} = S_{r}^*$, $S_{\ell}$ is also a ring map.  To prove that $\lambda X = XS_r(\lambda)$ for $\lambda \in T$, one must only check this for $\lambda = \omega$ and $\lambda = \theta$.  From (\ref{eqn:Uyz}), one obtains
\begin{equation}\label{eqAB} A^\# \times B^\# =\left((A,B^\#) - AB^\#\right)A = \left((A^\#,B) - BA^\#\right)B.\end{equation}
Hence,
\begin{align*} XS_{r}(\omega) &= \left(A\theta - B\omega - (A,B^\#)A + AB^\# A\right)(A^\#B) \\ &= aB\theta - BA^\#B\omega - acB + ad A. \end{align*}
Likewise one computes
\[\omega X = \omega (-A\theta + B\omega + bB - BA^\#B) = ad A -ac B + aB\theta - BA^\#B\omega,\]
so $\omega X = XS_r(\omega)$.  The proof that $\theta X = XS_r(\theta)$ is similar.

Since $X = X^*$, the equation $\lambda X = XS_r(\lambda)$ implies $\lambda X = S_{\ell}(\lambda)X$. \end{proof}

Before giving the extended lifting law for $J=H_3(C)$, we give one additional preparation.  \emph{From now on we assume that the pair $(A,B)$ is non-degenerate}, i.e. that $Q((A,B)) \neq 0$, or equivalently that $L = T\otimes_{R} F$ is a cubic \'{e}tale $F$-algebra.
\begin{definition}\label{epDef} Suppose given two copies $T_1, T_2$ of the $R$-algebra $T$, by which we mean given two $R$-algebras $T_1, T_2$ and $R$-algebra isomorphisms $\iota_{j}: T\rightarrow T_i$, for $j = 1,2$.  Let $\{v_{\alpha}\}$ be a basis of $T$, and $\{w_{\alpha}\}$ the dual basis in $L= T \otimes F$ for the trace form.  Define
\[\epsilon(T_1,T_2) := \sum_{\alpha}{\iota_1(v_\alpha) \otimes \iota_2(w_\alpha)} \in T_1 \otimes_R T_2 \otimes F.\]
Then $\epsilon(T_1,T_2)$ is independent of the choice of basis $\{v_{\alpha}\}$. When $T_1 = S_{r}(T), T_2 = T$, $\iota_1 = S_{r}$ and $\iota_2$ is the identity, we set $\epsilon := \epsilon(S_{r}(T_1),T)$, so $\epsilon \in M_3(C)\otimes L$.\end{definition}

With this definition, one has the following lemma.
\begin{lemma}\label{epsilonLemma} Suppose $T_1, T_2, \iota_1, \iota_2$ are as in Definition \ref{epDef}.  Then $\iota_1(x) \epsilon(T_1,T_2) = \epsilon(T_1,  T_2) \iota_2(x)$ for all $x \in T$. \end{lemma}
\begin{proof} It suffices to check the identity for the elements of a basis of $T \otimes F$ over $F$, and thus we may assume $x$ is invertible in $T \otimes F$.  Now, let $v_{\alpha}$, $w_{\alpha}$ be as in Definition \ref{epDef}. Then we have
\begin{align*} \iota_1(x) \epsilon(T_1,T_2) &= \sum_{\alpha}{ \iota_1(x) \iota_1(v_{\alpha}) \otimes \iota_2(w_{\alpha})} \\ &= \sum_{\alpha}{ \iota_1(xv_{\alpha}) \otimes \iota_2(x^{-1}w_{\alpha})\iota_2(x)} \\ &= \epsilon(T_1, T_2) \iota_2(x) \end{align*}
since $x^{-1}w_{\alpha}$ is the basis dual to $x v_{\alpha}$, and $\epsilon(T_1,T_2)$ is independent of the choice of basis. \end{proof}

Finally, with these preparations, we can state the lifting law.  Write $V_3(C) = C^3$ to be row vectors with entries in our associative composition ring $C$.  In the lifting law Theorem \ref{B2Jlift}, $X(A,B,\omega,\theta)$ is shown to be rank one.  In Theorem \ref{B2liftC}, we now show explicitly the way in which $X(A,B,\omega,\theta)$ is a ``pure tensor''.
\begin{theorem}\label{B2liftC}Let $\epsilon$, $X = X(A,B,\omega,\theta)$, $Y = Y(A,B,\omega,\theta)$, and $Q((A,B))$ be as above.  Then for all $v \in V_3(C)$, one has 
\begin{equation}\label{B2liftEq} Q((A,B))(v\epsilon)^* (v\epsilon) = (v Y v^*) X.\end{equation} \end{theorem}
Before giving the proof, we give various preparations.  Observe that 
\[ - A\theta_0 + B\omega_0 = \left(\begin{array}{cc} \omega_0 & \theta_0\end{array}\right)\left(\begin{array}{cc} &1\\ -1& \end{array}\right) \left(\begin{array}{c} A \\ B \end{array}\right).\]
Hence under the action of $\GL_2$, $g \cdot (A,B) = (A',B')$, which induces $(\omega_0,\theta_0) \mapsto (\omega_0',\theta_0')$, the expression $- A\theta_0 + B\omega_0$ gets multiplied by $\det(g)^2$.  Similarly one has that the quantity $-AS_{r}(\theta_0) + BS_r(\omega_0)$, which is in $H_3(C_F)$, gets multiplied by $\det(g)^2$ under the action of $\GL_2$ on $A,B$. It is thus convenient to write $X$ and $Y$ above in terms of $- A\theta_0 + B\omega_0$ and $-AS_{r}(\theta_0) + BS_r(\omega_0)$, which we will now do.
\begin{lemma}\label{X0} One has the equality
\[X = -\frac{1}{2}\left( A S(\theta_0) - B S(\omega_0)\right) - A \theta_0 + B \omega_0.\]
\end{lemma}
\begin{proof} From (\ref{eqAB}) one gets 
\begin{equation}\label{eqAsharpBsharp} A^\# \times B^\# = \frac{1}{2}(cA - AB^\#A) + \frac{1}{2}(bB - BA^\#B),\end{equation}
and thus
\begin{align*} X &= -A(\theta_0 + c/3) + B(\omega_0 -b/3) + \frac{1}{2}(cA - AB^\#A) + \frac{1}{2}(bB - BA^\#B)\\ &= -A\theta_0 + B\omega_0 -\frac{1}{2} AB^\#A + \frac{c}{6} A -\frac{1}{2} BA^\#B + \frac{b}{6}B \\ &= -\frac{1}{2}\left( A S(\theta_0) - B S(\omega_0)\right) - A \theta_0 + B \omega_0.\end{align*}
This completes the proof.
\end{proof}

Define
\[Y_0 = Y_0(A,B,\omega,\theta) = \left(3(A \theta_0 - B\omega_0)\right)^\# + \left(3(A \theta_0 - B\omega_0)\right) \times \left(3(A S(\theta_0) - BS(\omega_0))\right).\]
\begin{lemma} One has $Y_0 = - 3 Y$. \end{lemma}
\begin{proof} We have
\begin{align*} \left(A(3\theta_0)-B(3\omega_0)\right)^\# &= A^\#\left(3c\theta_0-9d\omega_0 + 2c^2 -6db\right) - (A\times B)\left(3b \theta_0 -3c\omega_0 +bc-9ad\right) \\ &\quad + B^\#\left(9a\theta_0-3b\omega_0 + 2b^2 -6ac\right) \end{align*}

Additionally, we have $S(3\theta_0) = 3B^\#A -c$, $S(3\omega_0) = -3A^\#B + b$, $AS(3\theta_0) - BS(3\omega_0) = 3(AB^\#A + BA^\#B) - (cA+bB)$.  Applying (\ref{eqAsharpBsharp}) yields
\[AS(3\theta_0) - BS(3\omega_0) = - 6A^\# \times B^\# + 2cA +2bB.\]
Using this, one gets
\[ A \times (AS(3\theta_0) - BS(3\omega_0)) = -2cA^\# + 2b A\times B - 6aB^\#\]
and
\[ B \times (AS(3\theta_0) - BS(3\omega_0)) = -6dA^\# + 2c A\times B - 2bB^\#.\]
We obtain
\begin{align*} Y_0 &= A^\#\left(-3c\theta_0 + 9d\omega_0 +2c^2-6bd\right) - A\times B\left(-3b\theta_0+3c\omega_0 + bc-9ad\right) \\ &\quad + B^\#\left(-9a\theta_0+3b\omega_0+2b^2-6ac\right) \\ &= A^\#\left(-3c\theta + 9d\omega +3c^2-3bd\right) - A\times B\left(-3b\theta+3c\omega + 3bc-9ad\right) \\ &\quad + B^\#\left(-9a\theta+3b\omega+3b^2-3ac\right).\end{align*}
Comparing with (\ref{Yexp}), we get $Y_0 = -3 Y$, as desired.
\end{proof}
If $m \in \GL_3(C)$, and $\mu \in \GL_1$, we get a map on $H_3(C) = J$ via $h \mapsto \mu mhm^*$.  Let $[\mu,m]$ denote the corresponding element of $M_J$.  Denote by $N(m)$ the degree $6$ reduced norm on $H_3(C)$.  One has $N(m) = n(mm^*)$.
\begin{proposition}\label{equivariance} The element $[\mu,m] \in M_J$ changes the quantities $f(x,y)= n(Ax + By), \omega_0, \theta_0, X,$ $Y_0, \epsilon$ as follows:
\begin{enumerate}
\item $f(x,y) \mapsto \mu^3 N(m) f(x,y)$;
\item $\omega_0 \mapsto \mu^3 N(m) \omega_0$;
\item $\theta_0 \mapsto \mu^3 N(m) \theta_0$;
\item $X \mapsto \mu^4 N(m) m X m^*$;
\item $Y_0 \mapsto \mu^8 N(m)^3 (m^*)^{-1} Y_0 m^{-1}$;
\item $\epsilon \mapsto (m^*)^{-1} \epsilon m^*$. \end{enumerate}
Under the action of $g \in \GL_2$,  the quantities $f(x,y)= n(Ax + By), \omega_0, \theta_0, X, Y_0, \epsilon$ change as follows:
\begin{enumerate}
\item $f(x,y) \mapsto f((x,y)g) = \det(g) \left(\det(g)^{-1} f((x,y)g)\right)$;
\item $\left(\begin{array}{c} \omega_0 \\ \theta_0 \end{array}\right) \mapsto \det(g) g \left(\begin{array}{c} \omega_0 \\ \theta_0 \end{array}\right)$;
\item $X \mapsto \det(g)^2 X$;
\item $Y_0 \mapsto \det(g)^4 Y_0$;
\item $\epsilon \mapsto \epsilon$. \end{enumerate}
\end{proposition}
\begin{proof} These are straightforward computations, especially by using the equivariance properties of the expression $A\theta_0 - B \omega_0$. \end{proof}

Finally, set 
\[V=V(1,\omega,\theta) = \tr_{L/F}\left(\left(\begin{array}{c}1\\ \omega \\ \theta \end{array}\right)\left(\begin{array}{ccc} 1 &\omega &\theta \end{array}\right)\right) = \left(\begin{array}{ccc} \tr(1) &\tr(\omega)&\tr(\theta) \\ \tr(\omega) &\tr(\omega^2) &\tr(\omega \theta) \\ \tr(\theta) &\tr(\theta \omega) &\tr(\theta^2) \end{array}\right).\]
Suppose $x_i, y_j$ are in $T$, $(1,\omega,\theta)m = (x_1,x_2,x_3)$ and $(1,\omega,\theta)n = (y_1,y_2,y_3)$, with $m, n \in M_3(R)$.  Then $(\tr_{T/R}(x_iy_j)) = \,^tm V n$.  It follows that $Q(f) \epsilon = \det(V)\epsilon$ is a polynomial in $1,\omega,\theta, A,$ and $B$.
\begin{proof}[Proof of Theorem \ref{B2liftC}] Multiplying both sides of (\ref{B2liftEq}) by $Q(f)$, this identity becomes a polynomial identity in $A, B, \omega, \theta$ and $v$.  Thus to prove the lifting law, it suffices to check that both sides of (\ref{B2liftEq}) behave the same way under the action of $\GL_2 \times M_J$, and that one has the equality for all $v$ in a particular non-degenerate case.  

For the equivariance, suppose $(g, [\mu,m]) \in \GL_2 \times M_J$ acts on $A,B$ as defined above.  Define an action of $\GL_3(C)$ on $V_3(C)$ by $v \mapsto vm^*$.  Under these actions, one finds using Proposition \ref{equivariance} that both of (\ref{B2liftEq}) change via
\[Z \mapsto \det(g)^{6} \mu^{12} N(m)^{4} mZm^*.\]

For the non-degenerate case, take $A = 1$ and $B = \mathrm{diag}(d,-d,0)$ for some nonzero $d \in F$.  Then $f(x,y) = x^3 - d^{2}xy^2 = x(x-dy)(x+dy)$, $Q(f) = 4d^6$, so the pair $(A,B)$ is non-degenerate.  We have $S_{r}(\omega) = \diag(-d,d,0)$, $S_{r}(\theta) = \diag(0,0,-d^2)$. We identity $L$ with $F \times F \times F$ via the map that takes $\omega \mapsto (-d,d,0)$ and $\theta \mapsto (0,0,-d^2)$.  This is an $F$-algebra isomorphism.

Since we can compute $\epsilon$ in any basis, take the basis $\epsilon_1 = (1,0,0)$, $\epsilon_2 = (0,1,0)$, $\epsilon_3 = (0,0,1)$ of $L$. This basis is self-dual, and we have $S_{r}(\epsilon_1) = \diag(1, 0,0)$, $S_{r}(\epsilon_2) = \diag(0,1,0)$, $S_{r}(\epsilon_3) = \diag(0,0,1)$.  We obtain
\[ \epsilon = \left(\begin{array}{ccc} (1,0,0) && \\ & (0,1,0)& \\ & & (0,0,1) \end{array}\right) =: E.\]
We now compute $X$.  One obtains
\begin{align*} X &= (0,0,d^2) 1 + \left(\begin{array}{ccc} (-d^2,d^2,0) & &\\ & (d^2,-d^2,0) & \\ & & 0 \end{array}\right) - d^2\left(\begin{array}{ccc} 1 & & \\ & 1 & \\ & & 0 \end{array}\right) \\ &= d^2(-2,-2,1)E. \end{align*}
Thus 
\[Y = \frac{1}{2} X \times_T X = d^4 (-2,-2,1)^\# \frac{1}{2} E \times_{T} E = d^4(-2,-2,4) E.\]

Suppose $v = (v_1, v_2, v_3)$, with $v_1, v_2, v_3$ in the composition algebra $C$.  For our choice of $(A,B)$, we see that both $Q(f) (v\epsilon)^* (v\epsilon)$ and $(v Y v^*) X$ are equal to $4 d^6 (n(v_1), n(v_2), n(v_3))E$. This completes the proof. \end{proof}

In the next subsection, we will consider the orbit problem for the action of $\GL_2(R) \times \SL_3(C)$ on $H_3(C)^2$ when $R = F$ is a field.  Before doing this, we finish this subsection with a preparatory result.

The following proposition says that $Y$ is also an eigenvector for the action of $S(T)$.
\begin{proposition}\label{Y_eig} Let $R, T, Y$ be as above, and drop (for this proposition) our running assumption that $Q((A,B)) \neq 0$, so that instead $T$ can be arbitrary.  One has $S_r(\lambda) Y = \lambda Y = YS_{\ell}(\lambda)$ for all $\lambda$ in $T$. \end{proposition}
Note that right and left have been switched.
\begin{proof}[Proof of Proposition \ref{Y_eig}] Since $YS_{\ell}(\lambda) = (S_{r}(\lambda)Y)^*$, it suffices to check that $S_{r}(\theta)Y = \theta Y$ and $S_{r}(\omega)Y = \omega Y$.  

From the identities of subsection \ref{special_Id}, one has $B^\#A(A \times B) = bB^\# - dA^\#$ and $B^\#AB^\# =-d (A \times B) + cB^\#$.  Thus
\begin{align*} S_{r}(\theta)Y &= \left(c\theta - 3d\omega -c^2+bd\right)(aB^\#) + \left(-b\theta + c\omega +bc - 3ad\right)(bB^\#-dA^\#)\\ &\quad + \left(3a\theta - b\omega - b^2+ ac\right)(-d (A\times B)+ cB^\#) \\ &= \left(db\theta - dc\omega -dbc + 3ad^2\right) A^\# + \left(-3ad\theta + db\omega + db^2 -dac\right) A\times B \\ &\quad + \left((4ac - b^2)\theta - 3ad \omega - 2abd\right)B^\#.\end{align*}
Computing $\theta Y$ using $\theta^2 = - bd + c\theta - d\omega$ and $\omega \theta = -ad$, one obtains the same expression, and thus $\theta Y = S_r(\theta)Y$, as desired.

The proof of the equality $S_{r}(\omega)Y = \omega Y$ is similar; for completeness, we give the details.  From the identities of subsection \ref{special_Id}, one obtains $-A^\#BA^\# = -bA^\# + a(A\times B)$, $-A^\#B(A\times B) = -cA^\# + aB^\#$, $-A^\#B(B^\#) = -dA^\#$.  Thus
\begin{align*} S_r(\omega)Y &= \left(c\theta - 3d\omega -c^2+bd\right)(-bA^\# + a (A\times B)) + \left(-b\theta + c\omega +bc - 3ad\right)(-cA^\#+aB^\#)\\ &\quad + \left(3a\theta - b\omega - b^2+ ac\right)(-dA^\#) \\ &= \left(-3ad\theta + (4bd-c^2)\omega + 2ad c\right)A^\# + \left(ac\theta - 3ad\omega - ac^2 + abd\right) A\times B \\ &\quad + \left(-ab\theta + ac\omega + abc - 3a^2d\right)B^\#.\end{align*}
Using now that $\omega^2 = -ac + a\theta - b\omega$, one finds for $\omega Y$ the same expression, and thus $S_{r}(\omega)Y = \omega Y$.  This completes the proof. \end{proof}

\subsection{Orbits over a field}\label{subsec:B2F} In this subsection, $R=F$ is a field.  An element $(A,B) \in H_3(C)^2$ is said to be non-degenerate if $Q((A,B)) \neq 0$. In this subsection we use the lifting law to parametrize the orbits of $\SL_3(C)$ on the non-degenerate elements of $H_3(C)^2$.  

\subsubsection{The invariant $\mu(X)$} Before stating the parametrization theorem, we associate to a non-degenerate pair $(A,B)$ two invariants.  The first invariant is the cubic \'{e}tale $F$-algebra $L$, determined by $f_{(A,B)}(x,y) = n(Ax+By)$, together with its good basis $1,\omega,\theta$.  The second invariant involves the rank one element $X = X(A,B,\omega,\theta) = -A\theta + B\omega + A^\#\times B^\#$.  Note that if $v \in V_3(C_L)$, then $v^* v$ in $H_3(C_L)$ is rank one; the second invariant associated to $(A,B)$ measures how far $X$ is from a ``pure tensor" $v^*v$ in $H_3(C_L)$.

In the following lemma and also below, we say a vector $v$ in $V_3(A)$, for an associative composition algebra $A$, is \emph{primitive} if $v$ can be extended to a basis of $V_3(A)$.
\begin{lemma}\label{lem:rk1Field} Suppose $k$ is a field, $A$ is an associative composition algebra over $k$, and $Y$ in $H_3(A)$ is rank one.  Then, there exists a primitive column vector $v_0 \in V_3(A)$ and an element $\mu \in k^\times$ so that $Y = \mu v_0 v_0^*$.  This element $\mu$ is unique up to multiplication by an element of $n(A^\times)$. \end{lemma}
\begin{proof} First assume that a primitive $v_0$ and a $\mu \in k^\times$ has been found; we check the uniqueness of $\mu$.  To do this, note that $v^*Yv = \mu n(v^*v_0)$.  Since $v_0$ is primitive, by definition $v_0 = g \cdot (1,0,0)^*$ for some $g \in \GL_3(A)$, and thus there is $v$ so that $n(v^*v_0) = 1$.  Hence $\{v^*Yv: v \in V_3(A)\} \cap k^\times$ is nonempty.  The formula $v^*Yv = \mu n(v^*v_0)$ then shows that
\[\{v^*Yv: v \in V_3(A)\} \cap k^\times = \mu n(A^\times),\]
which implies the desired uniqueness property.

For the existence, we show that there exists $g \in \SL_3(A)$ so that
\[gYg^* = Z = \mu e_{11} = \mu \left(\begin{array}{c} 1 \\ 0 \\ 0\end{array}\right) \left(\begin{array}{ccc} 1 &0 &0\end{array}\right).\]
Hence setting $v_0 = g^{-1} (1,0,0)^*$ gives the lemma, once $g$ is found.

To find such a $g$, one can use uppertriangular unipotent matrices in $\SL_3(A)$ to move $Y$ to an element $Z$ with some $Z_{ii} \neq 0$, and thus in $k^\times$ since $k$ is a field.  Using the permutation matrices $S_3 \subseteq \GL_3(A)$, one can assume $Z_{11} = \mu \in k^\times$.  Then again using unipotent matrices, one can clear out the top row of $Z$, so that $Z_{12} = Z_{13} = Z_{21} = Z_{31} = 0$.  Since $Y$ and thus $Z$ is rank one, we then get that $Z = \mu e_{11}$, as desired.\end{proof}

We require one more lemma before defining the invariant $\mu((A,B)) = \mu(X(A,B,\omega,\theta))$.
\begin{lemma}\label{nX0} Assume $Q(f) \neq 0$, so that $L$ is \'{e}tale, and $X = X(A,B,\omega,\theta)$.  Suppose $\sigma: L \rightarrow k$ is a map of $F$-algebras to a field $k$.  Then $\sigma(X) \neq 0$.  Furthermore, $n(\tr_{L/F}(X))= Q(f)$, so in particular $\tr_{L/F}(X)$ in $J$ is invertible. \end{lemma}
\begin{proof} We first make some general remarks.  Assume first that $X$ is an arbitrary element of $J\otimes_FL$, for some cubic norm structure $J$.  Extending scalars from $F$ to some algebraic closure $k$ of $L$, we can assume $L = k \times k \times k$ and $X = (X_1,X_2,X_3) \in J_k\times J_k \times J_k$.  In other words, $X =  X_1 \epsilon_1 + X_2 \epsilon_2 + X_3 \epsilon_3$, with $X_i \in J_k$ and $\epsilon_1 = (1,0,0)$, $\epsilon_2 = (0,1,0)$, $\epsilon_3 = (0,0,1)$ the primitive idempotents in $L$.

Set $Y = \frac{1}{2} X\times_{L} X$.  We find
\[Y = \frac{1}{2} \sum_{i,j}{ X_i \times X_j \otimes \epsilon_i \times \epsilon_j} = (X_2 \times X_3)\epsilon_1 + (X_3 \times X_1) \epsilon_2 + (X_1 \times X_2)\epsilon_3\]
since $\epsilon_1 \times \epsilon_2 = \epsilon_3$, $\epsilon_1^\# =0$ etc.  Thus
\[(Y,X) = (X_1, X_2\times X_3)\epsilon_1 + (X_2,X_3 \times X_1)\epsilon_2 + (X_3, X_1 \times X_2)\epsilon_3 = (X_1,X_2,X_3).\]
Suppose now that $X$ is rank one.  Then $X_i^\# = 0$ for all $i$, and hence $n(X_1 + X_2 + X_3) = (X_1,X_2 \times X_3) = (X_1,X_2,X_3)$.  Thus, if $X$ is rank one and $Y = \frac{1}{2} X\times_{L} X$, we have shown $(X,Y) = n(\tr_{L/F}(X))$.

Applying this fact to $X = X(A,B,\omega,\theta)$, and using that for this $X$, $(X,Y) = Q(f)$, we obtain 
\[n\left(\tr_{L/F}(X)\right) = (X_1,X_2,X_3) = Q(f) \neq 0.\]
In particular, each $X_i \neq 0$.  It follows that $\sigma(X) \neq 0$ for any $\sigma: L\rightarrow k$ as in the statement of the lemma.\end{proof}

We record for later use the following corollary.
\begin{corollary}\label{B2:linInd} Suppose the pair $(A,B)$ is non-degenerate, i.e., $Q(f) \neq 0$.  Then $A, B,$ and $A^\# \times B^\#$ are linearly independent in $J$. \end{corollary}
\begin{proof} $A,B,A^\#\times B^\#$ are linearly dependent if and only if there exists a nonzero linear map $\ell: L \rightarrow F$ such that $\ell(X) = 0$.  Since $L$ is \'{e}tale, every linear map $L \rightarrow F$ is of the form $\tr_{L/F}(\mu \cdot)$ for some $\mu \in L$.

Now, let $k$ be an algebraic closure of $F$, and $X_1, X_2, X_3$ as in Lemma \ref{nX0}.  To show $\tr_{L/F}(\mu X) = 0$ implies $\mu = 0$, it suffices to check that $X_1, X_2, X_3$ are linearly independent in $k$.  But if one had a dependence relation between the $X_i$, say $X_3 = \alpha X_1 + \beta X_2$, then
\[Q(f) = (X_1, X_2,X_3) = \alpha (X_1, X_2, X_1) + \beta(X_1,X_2,X_2) = 0\]
since $X_1, X_2$ are rank one.  Since $Q(f) \neq 0$, the corollary follows. \end{proof}

Using Lemmas \ref{nX0} and \ref{lem:rk1Field} we can now see that $X$ is of the form $v_0 v_0^*$ times a unit of $L$.
\begin{corollary} Suppose the pair $(A,B)$ is non-degenerate.  Then, there is $\mu \in L^\times$ and a primitive column vector $v_0 \in V_3(C_L)$ so that $X = \mu v_0 v_0^*$. The element $\mu$ is unique up to multiplication by $n(C_L^\times)$, and thus associated to $X$ is a well-defined class $[\mu_X] \in L^\times/(n(C_L^\times)).$\end{corollary}
\begin{proof} Since $(A,B)$ is non-degenerate, $L$ is an \'{e}tale $F$-algebra, and thus a product of fields.  Suppose $\sigma: L \rightarrow k$ is an $F$-algebra map to a field $k$.  Then by Lemma \ref{nX0}, $\sigma(X) \neq 0$, and thus $\sigma(X)$ is rank one.  It follows that we map apply Lemma \ref{lem:rk1Field} to each of the connected components of $Spec(L)$, to get the existence of $\mu \in L^\times$ and $v_0$ in $V_3(C_L)$ a primitive vector so that $X = \mu v_0 v_0^*$.  Uniqueness of $\mu$ modulo $n(C_L^\times)$ also follows. \end{proof}

\begin{definition}\label{B2FieldInv} To a non-degenerate pair $(A,B) \in J^2$, we associate its invariant, which is the triple $(L,(1,\omega,\theta), [\mu_X])$.  Here $(L,(1,\omega,\theta))$ is the \'{e}tale cubic $F$-algebra, together with its good basis $(1,\omega,\theta)$, determined by the binary cubic form $n(Ax + By)$, and the element $[\mu_X] \in L^\times/n(C_L^\times)$ is that for which $X_{(A,B)} = \mu_X v_0 v_0^*$ for some primitive vector $v_0 \in V_3(C_L)$. We say a triple $(L,(1,\omega,\theta),[\mu])$ is isomorphic to a triple $(L',(1,\omega',\theta'),[\mu'])$ if there is an $F$-algebra isomophism $\phi: L \rightarrow L'$ for which $\phi(\omega) = \omega'$, $\phi(\theta) = \theta'$, and $\phi([\mu]) = [\mu']$. \end{definition}

We can now state the orbit parametrization theorem.  Here $J = H_3(C)$ and $G = \SL_3(C)$, where $\SL_3(C) = \{m \in \GL_3(C): N(m) = 1\}$.  This result is closely related to work of Wright-Yukie \cite{wrightYukie}, Kable-Yukie \cite{kableYukie}, and Taniguchi \cite{taniguchi}.
\begin{theorem}\label{B2F} If the pair $(A,B)$ is non-degenerate, then $n_{L/F}(\mu(X(A,B))) \in n(C^\times)$.  The invariant map of Definition \ref{B2FieldInv}, which associates to $(A,B)$ the invariant $(L,(1,\omega,\theta),[\mu(X(A,B))])$, descends to a bijection between $G$ orbits on $J^2$ and isomorphism classes of pairs $(L,(1,\omega,\theta),[\mu])$, where $L, (1,\omega,\theta)$ is an \'{e}tale cubic $F$-algebra with good basis and $[\mu]$ is an element of the quotient group $L^\times\slash n(C_L)^\times$ for which $n_{L/F}(\mu) \in n(C^\times)$. \end{theorem}

Note that $\mu \in L^\times$ is only well-defined up to multiplication by an element $n_{C}(C_L^\times)$.  Thus implicit in the statement of Theorem \ref{B2F} is the fact that $n_{L/F}(n_{C}(C_L^\times)) \in n_C(C^\times)$.  This will be proved directly below.

We will prove Theorem \ref{B2F} in steps.  We begin with the following lemma.
\begin{lemma} If $h \in G$, then the invariant associated to $h \cdot (A,B)$ is isomorphic to the invariant associated to $(A,B)$. \end{lemma}
\begin{proof} If $m \in \GL_3(C)$ and $N(m) = 1$, then the cubic \'{e}tale $F$-algebra $L$ and its good basis does not change, and furthermore $X_{m \cdot (A,B)} = mX_{(A,B)}m^*$.  Hence $[\mu_X]$ does not change. \end{proof}

Thus the map from $G$-orbits on $J^2$ to arithmetic invariants $(L,(1,\omega,\theta),[\mu])$ is well-defined.  We now check that this map is injective.
\begin{proposition} Suppose $(A,B)$ and $(A',B')$ are non-degenerate elements of $J^2$ with invariants $(L,(1,\omega,\theta),[\mu])$ and $(L',(1,\omega',\theta'),[\mu'])$, and suppose that these invariants are isomorphic.  Then $(A,B)$ and $(A',B')$ are in the same $G$-orbit. \end{proposition}
\begin{proof} Write $X$, $X'$ for the rank one elements of $J \otimes L$ associated to the pairs $(A,B)$ and $(A',B')$, respectively.  Since $L,(1,\omega,\theta)$ is isomorphic to $L',(1,\omega',\theta')$  we have $n(Ax+By) = n(A'x + B'y)$.  Thus we may assume $L = L'$.  We have $X = \mu v_0 v_0^*$ and $X' = \mu' v_0' (v_0')^*$, and by changing $v_0'$ if necessary, we may assume $\mu = \mu'$.  Now, since $v_0$ and $v_0'$ are primitive, there is $\tilde{m} \in \GL_3(C_L)$ so that $\tilde{m}v_0 = v_0'$.  Hence $\tilde{m}X\tilde{m}^* = X'$.  Now, write $\tilde{m} = m_1 + m_\omega \omega + m_\theta \theta$, with $m_{*} \in M_3(C)$, and set $m = m_1 + m_\omega S_{\ell}(\omega) + m_\theta S_{\ell}(\theta)$.  Then $m \in M_3(C)$, and $\tilde{m}X\tilde{m}^* = mXm^*$ by Proposition \ref{prop:Srevec}, from which we deduce $A' = mAm^*$ and $B' = mBm^*$.  Since $n(Ax + By) = n(A'x+B'y) = N(m)n(Ax+By)$, we get $N(m) = 1$ and $m$ is invertible.  The proposition follows. \end{proof}

Say a rank one element $X \in J_L$ is \emph{admissible} if $X = X(A,B,\omega,\theta)$ for some $A,B$ in $J$.  Note that if $1,\omega',\theta'$ is a good basis of $L$, and $X = -A'\theta' + B'\omega' + C'$ for some $A',B', C'$ in $L$, then by Lemma \ref{B2uniqueDelta}, $X$ is admissible for $(A',B', \omega',\theta')$ if and only if $n(A'x+B'y) = f_{\omega',\theta'}(x,y)$.  Here $f_{\omega',\theta'}(x,y)$ is the binary cubic form associated to the good basis $1,\omega',\theta'$ of $L$.
\begin{lemma}\label{B2admis} Suppose $\mu \in L^\times$, and $X = X(A,B,\omega,\theta)$ is admissible.  Then $\mu n_{L/F}(\mu)X$ is admissible. \end{lemma}
\begin{proof} Set $X' = \mu n(\mu) X$.  We have 
\begin{align*}X' &= n(\mu)XS_{r}(\mu) = -AS_r(\mu)(n(\mu) \theta) + BS_r(\mu)(n(\mu)\omega) + n(\mu)CS_r(\mu) \\ &= -A' \theta' + B'\omega' + C'.\end{align*}
Here $A' = AS_r(\mu)$, $\theta' = n(\mu)\theta$ etc.  Note that since the left-hand side is in $J_L \subseteq M_3(C)_L$, so is the right-hand side, and thus $A', B', C' \in J$.  We have $f_{(\omega',\theta')}(x,y) = n(\mu)f_{(\omega,\theta)}(x,y)$.

Set $Y' = \frac{1}{2} X'\times_{L} X'$.  Then $Y' = \mu^\# n(\mu)^2 Y$, and thus $(X',Y') = n(\mu)^4 (X,Y) = n(\mu)^4 Q(f) \neq 0$.  It follows as in the proof of Corollary \ref{B2:linInd} that $A', B', C'$ are linearly independent.  Hence we may apply Lemma \ref{B2lem:uniqueLift} to deduce that 
\[n((Ax+By)S_{r}(\mu)) = n(A'x+B'y) = tf_{(\omega',\theta')}(x,y) = tn(\mu)n(Ax+By)\]
for some $t \in F$.

We claim $t \in F^\times$.  Indeed, this is clear, since $(Ax+By)S_{r}(\mu)$ is invertible in $M_3(C)\otimes F(x,y)$.  Hence $tX' = -A'(t\theta') + B'(t\omega') + tC'$ is admissible. But now, since $tX'$ is admissible, we obtain
\[t^4 n(\mu)^4 Q((\omega,\theta)) = Q((t\omega',t\theta')) = \frac{1}{2} (tX', (tX') \times_{L} (tX')) = t^3 (X',Y') = t^3 n(\mu)^4 Q(f).\]
Hence $t = 1$ and $X'$ is admissible. \end{proof}

As a corollary of the proof, we observe that 
\[n((Ax+By)S_{r}(\mu)) = n(\mu)n(Ax+By)\]
for all $\mu \in T$. 

To finish the proof of Theorem \ref{B2F}, we still must characterize the image of the invariant map.
\begin{proposition} Suppose $L/F$ is cubic \'{e}tale, $\mu \in L^\times$, and $n_{L/F}(\mu) \in n_C(C^\times)$.  Then, there exists a non-degenerate pair $(A,B) \in J^2$ with invariants $(L,(1,\omega,\theta),\mu)$. \end{proposition}
\begin{proof} First, we write down $(A_1,B_1)$ with invariants $(L,(1,\omega,\theta),1)$, then multiply the associated $X_1 = X_{(A_1,B_1)}$ by $\mu$ to find the desired pair $(A,B)$.

More specifically, following Bhargava \cite{bhargavaII}, set
\[A_1 = \left(\begin{array}{ccc} & 1 & \\ & -a & \\ 1 & &-c \end{array}\right), \qquad B_1 = \left(\begin{array}{ccc} & -1 & \\ -1 & -b & \\ & & -d \end{array}\right).\]
Then one computes
\[A_1^\# = \left(\begin{array}{ccc} ac&  &a \\ & -1 & \\ a & & \end{array}\right), \qquad B_1^\# = \left(\begin{array}{ccc} bd & -d & \\ -d &  & \\ & & -1 \end{array}\right),\]
and then
\begin{align*} n(A_1 x + B_1 y) &= n(A_1)x^3 + (A_1^\#,B_1)x^2y + (A_1,B_1^\#)xy^2 + n(B_1)y^3 \\ &= ax^3 +bx^2y + cxy^2 + dy^3. \end{align*}
Thus every $L,(1,\omega,\theta)$ arises.  Furthermore, note that the top left entry of $A_1^\# \times B_1^\#$ is $1$, so the same is true about $X_{(A,B)}$, and thus $\mu(X_{(A,B)}) =1 $.  Thus we have checked that the invariant $(L,(1,\omega,\theta),1)$ appears for every \'{e}tale cubic $L$ and good basis $(1,\omega,\theta)$. 

To get the invariant $(L,\mu)$, we multiply $X_{(A_1,B_1)}$ by $\mu n(\mu) = \mu n(c)$. By Lemma \ref{B2admis}, $X = \mu n(\mu) X_{(A_1,B_1)}$ is admissible, and it has invariants $(L,[\mu n(\mu)]) = (L,[\mu])$ since $n(\mu) \in n(C)$.  Multiplying $X$ by $\mu n(\mu)$ changes the good basis, but since every good basis arises, this is irrelevant.  The proposition follows.\end{proof}

If $m \in M_n(C)$, denote by $N_{2n}(m)$ the degree $2n$ reduced norm on $M_n(C)$.  If $h \in H_n(C)$, denote by $n_{H_n(C)}(h)$ the degree $n$ norm on $H_n(C)$, the Hermitian $n\times n$ matrices over the associative composition algebra $C$.  One has $N_{2n}(m) = n_{H_n(C)}(mm^*)$.
\begin{lemma} Suppose $C$ is an associative composition algebra, and $m \in M_n(C)$.  Then the norm $n_{H_n(C)}(mm^*)$ is a norm from $C$. \end{lemma}
\begin{proof} If $C$ is commutative this is clear.  Thus we may assume $C$ is a quaternion algebra.  Then the quantity $n_{H_n(C)}(mm^*) =N_{2n}(m)$ is the reduced norm of $m$ in $M_n(C)$, while norms from $C$ are the reduced norms from $C = M_1(C)$.  That these two sets are the same is well-known.\end{proof}

The following proposition completes the proof of Theorem \ref{B2F}.
\begin{proposition}\label{nunC} Suppose $X = X_{(A,B)}$ is admissible, and $X = \mu v^* v$, with $\mu \in L^\times$ and $v \in V_3(C_L)$.  Then $n_{L/F}(\mu) = n(x)$ for some $x \in C$. \end{proposition}
\begin{proof} We have $v = v_0 m$ for $v_0 = (1, \omega,\theta)$ and some $m \in M_3(C)$.  Set $X' = v^*v$ and $Y' = \frac{1}{2} X' \times_{L} X'$.  Then 
\[Q((A,B)) = (X,Y) = (\mu X', \mu^\# Y') = n_{L/F}(\mu)(X',Y') = n_{L/F}(\mu)n(\tr_{L/F}(X')).\]
These last two quantities are equal by the proof of Lemma \ref{nX0}.  Since $X' = v^*v = m^* v_0^*v_0 m$, $\tr_{L/F}(X') = m^* \tr_{L/F}(v_0^*v_0)m$ and 
\[n(\tr_{L/F}(X')) = N_6(m)n(\tr_{L/F}(v_0^*v_0)).\]
But $n(\tr_{L/F}(v_0^*v_0)) = \mathrm{disc}(1,\omega,\theta) = Q((A,B))$.  Hence we obtain
\[Q((A,B)) = (X,Y) = n_{L/F}(\mu)n(\tr_{L/F}(X')) = n_{L/F}(\mu)N_6(m)Q((A,B)).\]
Since $N_6(m) \in n(C)$, the proposition follows. \end{proof}

The following lemma will be used in the next subsection, and has essentially already been proved above.
\begin{lemma}\label{QimpliesAdmis} Suppose $X' = -A'\theta + B'\omega + C'$ is rank one in $J \otimes_{R} T$, and set $Y' = \frac{1}{2} X' \times_{T} X'$.  If $(X',Y') = \mathrm{disc}(1,\omega,\theta) \neq 0$, then $X'$ is admissible. \end{lemma}
\begin{proof} Set $Q(\omega,\theta) = \mathrm{disc}(1,\omega,\theta)$. We have 
\[n(\tr_{L/F}(X')) = (X'_1,X'_2,X'_3) =(X',Y')= Q(\omega,\theta) \neq 0.\]
Since $X'$ is rank one, we deduce that the $X'_i$ are linearly independent, and hence $\tr_{L/F}(\mu X) = 0$ implies $\mu = 0$, from which we conclude that $A', B', C'$ are linearly independent.  

It then follows that there exists $t \in F$ so that $f'(x,y) = n(A'x+B'y) = tf_{(\omega,\theta)}(x,y)$.  We must rule out the case $t=0$.  Define $C_1 \in J$ so that $X' = -A'\theta_0 + B'\omega_0 + C_1$.  Then we have
\[n( -A'\theta_0 + B'\omega_0) = n(X'-C_1) = (X',C_1^\#) - n(C_1) = - (A',C_1^\#)\theta_0 + (B',C_1^\#)\omega_0 + 2n(C_1).\]
We have $\tr_{L/F}(X') = 3C_1$, and thus $n(C_1) \neq 0$.  Hence $n( -A'\theta_0 + B'\omega_0) \neq 0$, and so we cannot have $t = 0$.

Thus $t \in F^\times$, and $tX' = -A'(t\theta) + B'(t\omega) + t'C'$ is admissible.  Hence
\[t^4 Q((\omega,\theta)) = Q((t\omega,t\theta)) = \frac{1}{2}(tX', (tX') \times_{T} (tX')) = t^3 \frac{1}{2}(X', X' \times_{T} X') = t^3 Q((\omega,\theta)).\]
Hence $t =1$, and $X'$ is admissible.\end{proof}

Recall the element $\epsilon \in M_3(C) \otimes L$ from Definition \ref{epDef}.  If $v_0 \in C_F^3$ (row vectors), define $\Phi^{v_0}: C_F^3 \rightarrow C_L$ (this $C_F^3$ being column vectors) via the formula $\Phi^{v_0}(\ell) = v_0\epsilon \ell$.   This map $\Phi^{v_0}$ is a map of right $C_F$-modules. The following proposition will also be used below.

\begin{proposition}\label{B2freeRk1} There exists $v_0 \in C_F^3$ (row vectors), and $\ell_0 \in C_F^3$ (column vectors) so that $(v_0^*v_0,Y) \in L^\times$ and $(\ell_0 \ell_0^*,X) \in L^\times$.  For such a $v_0$, the map $\Phi^{v_0}: C_F^3 \rightarrow C_L$ is a $C_F$-module isomorphism. \end{proposition}
\begin{proof}  We already know $X = \mu u_0 u_0^*$ for some primitive column vector $u_0 \in C_L^3$ and $\mu \in L^\times$. Since $(X,Y) = Q((A,B)) \in F^\times$, taking $\tilde{v} = u_0^*$ gives $\tilde{v}Y\tilde{v}^* = (\tilde{v}^*\tilde{v},Y) \in L^\times$.  Write $\tilde{v} = v_1 + v_{\omega}\omega + v_{\theta}\theta$, with the $v_{*} \in C_F^3$, and set $v_0 = v_1 + v_\omega S_{r}(\omega) + v_{\theta}S_{r}(\theta)$.  Then $v_0Yv_0^* = \tilde{v}Y\tilde{v}^* \in L^\times$.  The existence of $\ell_0$ with $(\ell_0 \ell_0^*,X) \in L^\times$ is proved similarly.  This proves the first part of the proposition.

For the second part, suppose $v \in C_F^3$ (row vectors) and $\ell \in C_F^3$ (column vectors).  By the lifting law, we can calculate $n_{C}(\Phi^v(\ell))$.  Indeed, we have
\begin{equation}\label{eqn:QABnC}Q((A,B)) n_{C}(\Phi^v(\ell)) = Q((A,B)) \Phi^v(\ell)^* \Phi^v(\ell) = (vYv^*)(\ell^*X\ell).\end{equation}
Furthermore, note that via the properties of $\epsilon$, $\Phi^v(C_F^3) \subseteq C_L$ is closed under left multiplication by $L$ and right multiplication by $C$, and hence is a $C_L$-submodule of $C_L$.

By (\ref{eqn:QABnC}), $n_{C}(\Phi^{v_0}(\ell_0)) \in L^\times.$  Thus, for $v_0$ as above, the map $\Phi^{v_0}: C_F^3 \rightarrow C_L$ is surjective, since the image contains a unit of $C_L$ and is closed under left $C_L$-multiplication. Since $\dim_{F} C_F^3 = \dim_{F} C_L = 3 \dim_{F}C$, $\Phi^{v_0}$ is injective as well.  This completes the proof. \end{proof}

\subsection{Integral orbits}\label{subsec:integralB2}  Suppose that $R$ is an integral domain, with fraction field $F$ of characteristic $0$, and $C$ is an associative composition algebra over $R$.  Recall that this means that $C$ is a free $R$-module of finite rank, is an $R$-subring of an associative composition $F$-algebra, stable under the involution $*$, and such that $x+x^*$, $xx^*$ are in $R$ for all $x \in C$.  In this subsection we consider the orbits of a certain group $G$ on $H_3(C)^2$.

More precisely, set $G_1 = \GL_1(R) \times \GL_2(R) \times \GL_3(C)$, which acts on the right of $V := V_2(R) \otimes H_3(C)$.  Here $V_2(R) = R^2$ is the row vectors with coefficients in $R$ and the action of $\GL_2$ action is the usual right action of matrices on row vectors.  The group $\GL_3(C)$ acts on the right of $H_3(C)$ by $h \mapsto m^*hm$ for $h \in H_3(C)$ and $m \in \GL_3(C)$. Finally, $\lambda \in \GL_1(R)$ acts on $V$ by scaling by $\lambda$. We denote by $G \subseteq G_1$ the group of triples $(\lambda,g,m) \in \GL_1(R) \times GL_2(R) \times \GL_3(C)$ with $\lambda = \det(g)^{-1}$ and $\det(g)^2 = N_6(m)$, where $N_6(m) = n(mm^*)$ is the degree $6$ reduced norm on $M_3(C)$.  Then $G$ also acts on $V$ by restriction of the action of $G_1$.

We begin with a few definitions.  Suppose $T$ is a cubic ring over $R$ and $L = T \otimes_{R} F$.  We have $1, \omega, \theta$ a good basis of $T$.  Let $C_{L} = C \otimes_{R} L$.  We consider $T \otimes C$ submodules $I \subseteq L \otimes C$.
\begin{definition} If $I \subseteq L \otimes C$, we say $I$ is a fractional ideal if
\begin{itemize}
\item $I$ is closed under left multiplication by $T$ and right multiplication by $C$;
\item $I$ is free of rank $3$ as a $C$-module;
\item $I \otimes_{R}F$ is free of rank one as an $L \otimes C$-module. \end{itemize} \end{definition}

\begin{definition} Suppose $I$ is a $T \otimes C$ fractional ideal, and $b_1, b_2, b_3$ is an ordered basis of $I$ i.e., that $I = b_1C + b_2C + b_3 C$ inside $L \otimes C$.  Then $(b_1,b_2,b_3) = (1\otimes 1,\omega \otimes 1,\theta \otimes 1)g$ for some unique $g \in M_3(C_F)$.  We define $N(I;(b_1,b_2,b_3);(1,\omega,\theta)) = N_6(g) = n(g^*g)$ the norm of the ideal.  Of course, this norm depends on the choice of the good basis $(1,\omega,\theta)$ and the choice of ordered basis $(b_1,b_2,b_3)$. \end{definition}

We will parametrize the orbits of $G$ on $H_3(C)^2$ in terms of \emph{balanced} fractional ideals.  We now define when the based fractional ideal $I,(b_1,b_2,b_3)$ is balanced.
\begin{definition} Suppose $\beta \in L^\times$, and the good basis $(1,\omega,\theta)$ of $T$ is fixed.  One says that the data $((I,(b_1,b_2,b_3)),\beta)$ is \textbf{balanced} if $y^*x \in \beta C \otimes_R T$ for all $x, y$ in $I$ and $N((I,(b_1,b_2,b_3))) = n_{L/F}(\beta)$. If the data $(I,b_1,b_2,b_3,\beta)$ is balanced, we set 
\[X_{I,\beta} = \beta^{-1} b^* b \in H_3(C) \otimes T.\]
\end{definition}

Associated to a non-degenerate pair $(A,B) \in H_3(C)^2$, one can define a based balanced $T \otimes C$-module $(I,\beta)$ as follows. First, from the non-degenerate pair $(A,B)$, one defines the binary cubic form $f_{(A,B)}(x,y) = n(Ax+By)$, which is associated to a good-based cubic ring $T, (1,\omega,\theta)$.   Next, one has the rank one elements $X = X(A,B,\omega,\theta)$, $Y = Y(A,B,\omega,\theta)$ in $H_3(C) \otimes T$.  One has $X = - A\theta + B\omega + A^\# \times B^\#$ and $Y = \frac{1}{2} X \times_{T} X$.  We set $S(\omega) = -A^\#B$ and $S(\theta) = B^\#A$ in $M_3(C)$.

Finally, suppose given some auxiliary $v_0 \in C^3$ (row vectors), with $v_0 Y v_0^* = (v_0^*v_0,Y)$ in $L^\times$.  By Proposition \ref{B2freeRk1}, such a $v_0$ can always be found.  Define $\beta(v_0) = \frac{(v_0^*v_0,Y)}{Q}$, $Q = Q(A,B) = \mathrm{disc}(1,\omega,\theta)$ and $\Phi^{v_0}: C^3 \rightarrow C_L$ (this $C^3$ being column vectors) via the formula $\Phi^{v_0}(\ell) = v_0 \epsilon \ell$.  
\begin{proposition}\label{balancedProp} Write $E_1 = (1,0,0)^{t}$, $E_2 = (0,1,0)^{t}$, $E_3 = (0,0,1)^{t}$, the standard basis of column vectors.  Set $b_j = \Phi^{v_0}(E_j)$ and $I$ the $C$ submodule generated by these $b_j$, i.e., $I = \Phi^{v_0}(C^3)$.  Then the data $(I,(b_1,b_2,b_3),\beta(v_0))$ is a balanced fractional $T\otimes C$-ideal. \end{proposition}
\begin{proof} By Proposition \ref{B2freeRk1} and its proof, $I=b_1 C + b_2C +b_3C$ is a fractional $T\otimes C$-ideal.  It remains to check that the data is balanced. 

Thus, suppose $g\in M_3(C)$ satisfies $(b_1,b_2,b_3) = (1,\omega,\theta)g$.  By the lifting law Theorem \ref{B2liftC}, we have 
\[\beta X_{A,B} = g^* \left(\begin{array}{c} 1 \\ \omega \\ \theta \end{array}\right)\left(\begin{array}{ccc} 1 & \omega& \theta \end{array}\right)g.\]
Now, take $n(\tr_{L/F}( \cdot ))$ of both sides.  On the left-hand side, one obtains $n_{L/F}(\beta)Q(A,B) \in F^\times$, by Lemma \ref{nX0}.  On the right-hand side, one obtains $n(g^*V g)$, with $V$ given below the proof of Proposition \ref{equivariance}.  But 
\[n(g^*Vg) = N_6(g)n(V) = N_6(g)\mathrm{disc}(1,\omega,\theta) = N_6(g)Q(A,B).\]
Since $Q(A,B) \in F^\times$, $n_{L/F}(\beta) = N_6(g)$, so the proposition follows. \end{proof}

We now define an equivalence relation on based fractional $T \otimes C$ ideals.  Suppose $I = b_1C + b_2C +b_3 C \subseteq C_L$ is a $C_T = C \otimes_{R} T$-module, and $\beta \in L^\times$ is a unit.  We fix the good basis $(1,\omega,\theta)$ of $T$.  Now say that $(b_1,b_2,b_3, \beta)$ is equivalent to $(b_1',b_2',b_3',\beta')$ if there exists $x \in C_{L}^\times$ so that
\[(b_1',b_2',b_3',\beta') = (xb_1, xb_2,xb_3, n_{C}(x)\beta).\]

We have the following lemma.
\begin{lemma}\label{equivMod} The tuples $(b_1,b_2,b_3, \beta)$ and $(b_1',b_2',b_3',\beta')$ are equivalent if and only if $X := X_{b,\beta} = X_{b',\beta'}=:X'$. \end{lemma}
\begin{proof} It is clear that if the tuples are equivalent, then $X=X'$.  Conversely, since by assumption the $b_i$ span $C_L$ as a $C_F$-module, $b$ is primitive.  (This follows, for example, from the identity 
\[(b_1,b_2,b_3) = (1,0,0) \left(\begin{array}{ccc}1 & \omega & \theta \\ 0&1&0 \\ 0&0&1 \end{array}\right) g\]
with $g$ invertible.) Thus there exists $g \in \GL_3(C_L)$ so that $b_0 :=bg = (1,0,0)$.  Set $b_0' = b'g$.  Then one sees easily that there exists $x \in C_L^\times$ so that $b_0' = xb_0$ and $\beta' = n_C(x)\beta$.  The result for $b$ and $b'$ follows. \end{proof}

Associated to a non-degenerate pair $(A,B) \in H_3(C)^2$, one can define a based balanced $T \otimes C$-module $(I,\beta)$ as above, using some choice of $v_0$.  Then, if one changes $v_0$ to a different $v_0'$, the $I',\beta'$ one gets from $v_0'$ is equivalent to the $I,\beta$ obtained from $v_0$.  For instance, this follows from the lifting law and Lemma \ref{equivMod}.  

Conversely, associated to a based-balanced $T \otimes C$-modules $(I,(b_1,b_2,b_3),\beta)$, one obtains a pair $(A',B') \in H_3(C)^2$ by the equality $X_{b,\beta'} = -A'\theta + B'\omega + C'$.  It is clear that the association $(I,(b_1,b_2,b_3),\beta) \mapsto (A',B')$ descends to the level of equivalence classes.  

\begin{lemma}\label{ABnondeg} If $(I,(b_1,b_2,b_3),\beta)$ is a fractional $T \otimes C$-ideal, and, as always, $L$ is \'{e}tale, then the associated pair $(A',B')$ is non-degenerate.\end{lemma}
\begin{proof} To see this, consider the equality
\[X:=X_{I,\beta} = \beta^{-1} g^* \left(\begin{array}{c} 1 \\ \omega \\ \theta \end{array}\right)\left(\begin{array}{ccc} 1 & \omega& \theta \end{array}\right)g.\]
Then if $Y = \frac{1}{2} X \times_{T}X$, one has 
\[(X,Y) = n_{L/F}(\beta)^{-1}N_6(g)\mathrm{disc}(1,\omega,\theta) = \mathrm{disc}(1,\omega,\theta) \neq 0\]
since $I,\beta$ is assumed balanced.  It now follows from Lemma \ref{QimpliesAdmis} that $(A',B')$ is non-degenerate.\end{proof}

The following result says that the above associations between equivalence classes of based-balanced $T \otimes C$ fractional ideals and non-degenerate pairs $(A,B) \in H_3(C)^2$ are inverse to one another.
\begin{theorem} The associations 
\[(A,B) \mapsto (T,(1,\omega,\theta),(I,(b_1,b_2,b_3)),\beta)\]
and
\[(T,(1,\omega,\theta),(I,(b_1,b_2,b_3)),\beta) \mapsto (A',B')\]
define inverse bijections between non-degenerate pairs $(A,B)$ and based-balanced fractional $T \otimes C$-ideals.  These bijections are equivariant for the action of $G_1$, where $(\lambda, g,m) \in \GL_1(R) \times \GL_2(R) \times \GL_3(C)$ acts on the data $(T,(1,\omega,\theta),(I,(b_1,b_2,b_3),\beta)$ via
\[((\omega_0,\theta_0),(b_1,b_2,b_3),\beta;v_0)\mapsto (\lambda^3N_6(m)\det(g)(\omega_0,\theta_0)g,(b_1,b_2,b_3)m,(\lambda^4 N_6(m)\det(g)^2)^{-1}\beta;v_0m).\]
Here, $((\omega_0,\theta_0),(b_1,b_2,b_3),\beta;v_0)$ means that the data $(b_1,b_2,b_3,\beta)$ is computed via the element $v_0 \in C^3$.
\end{theorem}
\begin{proof} The lifting law says that the composition $(A,B) \mapsto (b_1,b_2, b_3,\beta) \mapsto (A',B')$ is the identity.

The direction $(I,\beta) = (b_1,b_2,b_3,\beta) \mapsto (A,B) \mapsto (b_1',b_2',b_3',\beta')$ is also the identity.  First, we check that the action maps are the same.  For this, the basis $b$ gives rise to the action map $S_I: T \rightarrow M_3(C)$.  This map by definition satisfies the identity $\lambda b = b S_{I}(\lambda)$ for all $\lambda \in T$. One sees that this is the same as the action map $S_{A,B}: T \rightarrow M_3(C)$ by observing $X(I,\beta) = X(A,B)$.  Indeed, note that if $\lambda \in T$, then
\[ X(I,\beta)S_{I}(\lambda) = \lambda X(I,\beta) = \lambda X(A,B) = X(A,B)S_{A,B}(\lambda).\]
Taking $\tr_{L/F}$ of both sides gives $S_I = S_{A,B}$ since $\tr_{L/F}(X(A,B))$ is invertible, since we have already proved in \ref{ABnondeg} that $(A,B)$ is non-degenerate.  

Since we have just checked that the actions are the same, we may compute with the $\epsilon$ defined by $S_{I}$, as opposed to $S_{A,B}$.  Now, set $v_0 = \tr_{L/F}(b) = (\tr_{L/F}(b_1),\tr_{L/F}(b_2),\tr_{L/F}(b_3)) \in C_F^3$.  Then $v_0 \epsilon = \tr_{L/F}(b) \epsilon$ and one computes (very easily) that this latter expression is just $b$.  Hence with this choice of $v_0$, $(b_1,b_2,b_3) = (b_1',b_2',b_3')$.  Furthermore, since
\[\beta^{-1} b^* b = X(I,\beta) = X(A,B) = \beta(v_0)^{-1} (b')^* b' = \beta(v_0)^{-1} b^* b,\]
we get $\beta = \beta(v_0)$.  Thus, the composition $(I,\beta) \mapsto (A,B) \mapsto (I',\beta')$ is the identity as well.

Finally, the equivariance follows from Proposition \ref{equivariance}. 
\end{proof}
Call the fractional $T \otimes C$-ideal \emph{oriented} if it comes equipped with an $\SL_3(C)$-orbit of $C$-bases $(b_1,b_2,b_3)\SL_3(C)$.  The following corollary proves Theorem \ref{twistThm2} and implies Theorem \ref{introThm2} of the introduction.
\begin{corollary}\label{cor:equivG2} If $m \in \GL_3(C)$ and $g \in \GL_2(R)$ with $N_6(m) = \det(g)^2$, so that $h = (\det(g)^{-1},g,m)$ is in $G$, then $h$ acts on the data via
\[((\omega_0,\theta_0),(b_1,b_2,b_3),\beta;v_0)\mapsto ((\omega_0,\theta_0)g,(b_1,b_2,b_3)m,\beta;v_0m).\]
Consequently, if $(R^\times)^2 = 1$ (e.g., if $R = \Z$), then the $G$ orbits on $(R^2 \otimes H_3(C))^{open}$ parametrize $(T,I,\beta)$ up to equivalence and isomorphism, where $I$ is oriented and $(I,\beta)$ is balanced. \end{corollary}
\begin{proof} The first part of the corollary is immediate.  For the second, note that if $(I,(b_1,b_2,b_3)\SL_3(C))$ is an oriented fractional $T \otimes C$-ideal, then $N(I,(b_1,b_2,b_3))$ is well-defined in $F^\times/(R^\times)^2$.  Indeed, this follows from the fact that in the equation $(b_1,b_2,b_3) = (1,\omega,\theta) g$, $g$ is well-defined up to $g \mapsto u g g'$ where $g' \in \SL_3(C)$ and $u \in \GL_3(R)$.  Hence if $(R^\times)^2 = 1$ and $I$ is oriented, the notion of a balanced pair $(I,\beta)$ is unambiguous.  Thus the second part of the corollary follows from the first. \end{proof}

\section{The second lifting law for $W_J$}\label{nccubes}
The purpose of this section is to give a second lifting law for the space $W_J$, where $J$ is of a certain form.  More specifically, suppose that $K$ is a quadratic \'{e}tale extension of $F$, and $B$ is an associative cubic norm structure over $K$.  Furthermore, assume that $B$ comes equipped with an involution of the second kind, i.e., an order-reversing involution $*: B\rightarrow B$, which induces the nontrivial $F$-linear involution on $K$.  This involution is required to be compatible with the cubic norm structure, in the sense that $n(x^*) = n(x)^*$ and $(x^*)^\# = (x^\#)^*$.  Define $J \subseteq B$ to be the $F$-subspace of elements fixed by $*$.  Then $J$ is a cubic norm structure over $F$.  In this section we will discuss a second lifting law for the space $W_J$.

Before explaining this lifting law, we list the common examples of such pairs $(J,B)$.
\begin{example}\label{ex:JBK} Here are the common examples of pairs $J, B$ as above.
\begin{enumerate}
\item Suppose $K$ is any quadratic \'{e}tale extension of $F$, $B = M_3(K)$ with its usual cubic norm structure, and $*$ conjugate transpose.  Then $J = H_3(K)$ is the Hermitian $3 \times 3$ matrices.
\item Suppose $K = F \times F$, and $A$ is a central simple $F$-algebra of degree three.  Then set $B = A \times A^{opp}$, with involution $*:(x,y^{op}) \mapsto (y,x^{op})$.
\item Suppose $A$ is an associative cubic norm structure over $F$, for which the multiplication in $A$ is commutative, and suppose that $K$ is any \'{e}tale quadratic $F$-algebra.  Set $B = A \otimes_{F} K$, with involution $(a\otimes \lambda)^* = a \otimes (\lambda^*)$ for $a \in A$ and $\lambda \in K$.
\end{enumerate}
\end{example}

Throughout this section, $J, B, K, F,$ and $*$ are as above.  In this setting, Tits defined a cubic norm structure over $F$ on the space $U = J \oplus B$, which we will review momentarily.  In this case, 
\[W_{U} = F \oplus (J \oplus B) \oplus (J \oplus B) \oplus F = W_{J} \oplus B^2\]
as $F$-vector spaces, and in fact as modules for a large subgroup $G$ of $H(W_J)$.  In this section we investigate when an element $v \in W_J$ of rank four can be lifted to a rank one element of $W_U$.  In case $(J,B)$ comes from a cubic \'{e}tale algebra $A$ over $F$, as in the third item of Example \ref{ex:JBK}, the lifting law we prove reduces to a lifting law essentially contained in \cite{ganSavin}.  

We begin this section by reviewing the second construction of Tits.  We then describe the group $G$ that acts on $W_J$ and $B^2$;  $G$ is essentially $\GU(2,B)$.  After this, we prove Theorem \ref{thm:LLWJ2}, which is the lifting law in this context.  Interestingly, to prove this lifting law, we use the \emph{first} lifting laws Theorems \ref{b1J} and \ref{b1A} for the spaces $W_J$ and $W_B$.  Finally, we end this section by giving a second construction of this lifting law, that is in a sense more explicit.  We also explain the relation between this more explicit construction and some of the results in \cite{ganSavin}.

\subsection{The second construction of Tits}\label{subsec:secondTits} In this subsection, we review the so-called second construction of Tits. More precisely, suppose that $\lambda \in K^\times$, $S \in J$, and $n(S) = \lambda \lambda^*$, so that in particular $S$ is rank three.  Out of this data, one puts on $U = J \oplus B$ the norm, adjoint, and pairing as follows: For $X \in J$ and $\alpha \in B$,
\begin{enumerate}
\item $n((X,\alpha))= n(X) -(X,\alpha S \alpha^*) + \tr_{K/F}(\lambda n(\alpha))$;
\item $(X,\alpha)^\# = (X^\#-\alpha S \alpha^*,-X\alpha + \lambda^{-1} (\alpha^*)^\# S^\#)$;
\item $((X,\alpha),(Y,\beta)) = (X,Y) + \tr_{J}(\alpha S \beta^* + \beta S \alpha^*)$.\end{enumerate}
It is a fact that with these definitions, $U$ is a cubic norm structure.  This is the so-called second construction of Tits.  We call the above cubic norm structure $U = U(S,\lambda)$. 

\subsection{Group actions}\label{subsec:GBgrp} In this subsection, we define an action of a large subgroup of $\GU(2,B)$ on $W_J$.  Here
\[\GU(2,B) = \{(g,\nu(g)) \in \GL_2(B)\times \GL_1: g^*\mm{}{1}{-1}{}g = \nu(g) \mm{}{1}{-1}{}\}.\]
Set $G \subseteq \GL(2,B)$ to be
\[ G = \{g \in \GU(2,B): \det(g) = \nu(g)^3\},\]
where $\det(g)$ in this definition is the degree $6$ reduced norm on $M_2(B)$ as a central simple $K$-algebra. This group acts naturally on $W_J$ as follows.  First, consider on $W_{B} = K \oplus B \oplus B \oplus K$ the involution $*: (a,b,c,d) \mapsto (a^*,b^*,c^*,d^*)$.  Clearly, $W_J$ is the fixed space of $*$ on $W_B$.  Since $\GL_2(B)$ acts on $W_B$, we have an action of $\GU(2,B)$ and thus $G$ on $W_B$.  We will check that this action of $G$ preserves $W_J$.  We use the left action of $\GL_2(B)$ and $\GU(2,B)$ on $W_{B}$.

One way to check this is simply by checking it on generators for $G$.  However, we will give a different proof, which is perhaps more enlightening, by checking that $\GU(2,B)$ preserves an extra $K$-valued Hermitian form on $W_B$.

On the column vectors $W_2(B)$, one has the $B$-valued Hermitian symplectic form 
\begin{equation}\label{eqn:BHerm}\left\langle \left(\begin{array}{c}  x \\ y \end{array}\right), \left(\begin{array}{c}  x' \\ y' \end{array}\right)\right\rangle_B = (x^*, y^*) \left(\begin{array}{cc} & 1 \\ -1 & \end{array}\right) \left(\begin{array}{c}  x' \\ y' \end{array}\right) = x^*y' - y^*x'.\end{equation}
This $B$-valued pairing induces on $W_2(B)^{\otimes 3}$ a $K$-valued Hermitian form.  Namely, for $\eta_i, \eta_j' \in W_2(B)$, define
\begin{equation}\label{Kherm}\langle \eta_1\otimes \eta_2 \otimes \eta_3, \eta_1'\otimes \eta_2' \otimes \eta_3' \rangle = (\langle \eta_1, \eta_1'\rangle_B, \langle \eta_2,\eta_2'\rangle_B, \langle \eta_3, \eta_3' \rangle_B),\end{equation}
where on the right-hand side we are applying the $K$-valued symmetric trilinear form on $B$ that is the polarization of the norm form.  

Note that the elements of the form $\eta_1 a\otimes \eta_2 a\otimes \eta_3 a - n(a) \eta_1 \otimes \eta_2 \otimes \eta_3$ are in the radical of the form (\ref{Kherm}).  That is, 
\[ \langle \eta_1 a\otimes \eta_2 a\otimes \eta_3 a - n(a) \eta_1 \otimes \eta_2 \otimes \eta_3, \eta_1'\otimes \eta_2' \otimes \eta_3' \rangle = 0\]
for all $\eta_1',\eta_2', \eta_3'$.  Thus, this form descends to $V_{B} \simeq W_{B}$.  Furthermore, it is clear that this form is Hermitian, and that $\GU(2,B)$ preserves it up to $\nu^3$: 
\[\langle g\eta_1\otimes g\eta_2 \otimes g\eta_3, g\eta_1'\otimes g\eta_2' \otimes g\eta_3' \rangle = \nu(g)^3 \langle \eta_1\otimes \eta_2 \otimes \eta_3, \eta_1'\otimes \eta_2' \otimes \eta_3' \rangle.\]

We calculate what this form is in coordinates of $W_B$.  For $v, v' \in W_{B}$, denote $\langle v,v'\rangle_K$ the Hermitian form (\ref{Kherm}), and denote $\langle v, v'\rangle$ (with no $K$-subscript) the usual $K$-valued symplectic form on $W_B$.
\begin{lemma}\label{lem:KHermform} If $v, v' \in W_B$, then $\langle v,v'\rangle_{K} = 6\langle v^*,v' \rangle$.  In particular, since $G$ preserves both this form and the symplectic form up to $\det(g) = \nu(g)^3$, $G$ preserves $W_J$. \end{lemma}
\begin{proof} Recall that an element $(a,b,c,d)$ of $W_{B}$ is the image of
\[a e\otimes e\otimes e + \sum_{cyc}{e\otimes e \otimes bf} + \sum_{cyc}{f \otimes f \otimes ce} + d f\otimes f\otimes f.\]
The first part of the lemma is now a simple computation.  For example, 
\begin{align*}\langle (0,b,0,0), (0,0,c',0)\rangle_{K} &= \langle \sum_{cyc}{e\otimes e \otimes bf}, \sum_{cyc}{f \otimes f \otimes ce} \rangle_{K} \\ &= 3(1,1,-b^*c) \\ &=-6(b^*,c) \\&= 6 \langle (0,b,0,0)^*, (0,0,c,0) \rangle.\end{align*}

For the second part of the lemma, suppose $g \in \GU(2,B)$.  Then for all $v,v' \in W_B$ one has
\[6 \nu(g)^3 \langle v^*,v'\rangle = \nu(g)^3 \langle v,v'\rangle_{K} = \langle gv, gv'\rangle_{K} = 6\langle (gv)^*,gv'\rangle = 6 \det(g) \langle g^{-1}(gv)^*,v'\rangle.\]
Hence if $g \in G$ so that $\nu(g)^3 = \det_6(g)$, then $(gv)^* = g(v^*)$.  It follows that $G$ preserves $W_J$, as claimed.\end{proof}

From now on, we let $G$ act on $W_J$ as the $\nu^{-1}$ twist of the cubic polynomial action derived from the action of $\GL_2(B)$ on $W_{B}$.  With this action, we obtain a map $G \rightarrow H(W_J)$, preserving similitudes.  Indeed, this follows immediately from Lemma \ref{lem:KHermform} and the fact that the cubic polynomial action of $\GL_2(B)$ on $W_B$ preserves the symplectic and quartic form up to $\det_6(g)$ and $\det_6(g)^2$, respectively.

Since $G$ acts on $W_J$ and on $B^2$, it acts on $W_{U} = W_J \oplus B^2$.  The explicit identification of the right-hand side with the left-hand side is
\[v + \eta = (a,b,c,d) + \left(\begin{array}{c} u\\ v \end{array}\right) \mapsto (a,(b,-u),(c,v),d).\]
Here $(b,-u)$ and $(c,v)$ are considered elements of $U = J \oplus B$.  With this choice of identification, the action of $G$ on $W_J \oplus B^2 \simeq W_{U}$ preserves the symplectic and quartic form on $W_U$ up to similitude.
\begin{proposition}\label{prop:Gaction} Suppose $U = U(S,\lambda)$ for any $S, \lambda$ as above.  Then the map $G \rightarrow \GL(W_{U})$ defined by the action $g(v + \eta) = gv + g\eta$ preserves the symplectic and quartic form on $W_U$ up to $\nu$ and $\nu^2$, respectively.  That is, this action defines an embedding $G \rightarrow H(W_U)$, preserving similitudes.\end{proposition}
\begin{proof} It suffices to check this on generators of $G$, which may be done easily.\end{proof}

\subsection{The lifting law} With the above preparations, we can now state and prove the lifting law. Suppose given $v = (a,b,c,d) \in W_J$ of rank four.  We ask when we can lift $v$ to a rank one element of $W_{U}$.  This question is answered by the following theorem.
\begin{theorem}\label{thm:LLWJ2} Suppose $v \in W_J$ is rank four, and set $J_2 = \mm{}{1}{-1}{}$.
\begin{enumerate}
\item If there exists $(S,\lambda)$ as above, and $\eta \in B^2$, so that $v + \eta$ is rank one in $W_{U} \simeq W_J \oplus B^2$, then $q(v) = \omega^2$ for $\omega \in K^\times$ with $\omega^* = -\omega$.
\item Conversely, suppose $q(v) = \omega^2$ for $\omega \in K^\times$ with $\omega^* = -\omega$.  Then there exists $S, \lambda, \eta$ with $v + \eta$ rank one in $W_{U(S,\lambda)}$.  More precisely, by the lifting law Theorem \ref{b1A} for $W_{B}$, there exists $\lambda \in K^\times$ and $\eta \in B^2$ so that $\lambda \eta^{!} = \frac{-\omega v + v^\flat}{2}$.  Then $n\left(\frac{\langle \eta, \eta \rangle_B}{\omega}\right) = (\lambda \lambda^*)^{-1}$ is invertible.  Set $S = \left(\frac{\langle \eta, \eta \rangle_B}{\omega}\right)^{-1}.$  Then $v + \eta$ is rank one in $W_{U(S,\lambda)}$.  Furthermore, one has the identity $\eta S \eta^* = S(v) - \frac{\omega}{2}J_2$.
\item Finally, suppose $S,\lambda$ and $\eta$ are such that $v + \eta$ is rank one in $W_{U(S,\lambda)}$.  Then $\lambda \eta^{!} = \frac{\pm \omega v + v^\flat}{2}$.  Choose the element $\omega$ so that $\lambda \eta^{!} = \frac{-\omega v + v^\flat}{2}$.  Then $\frac{\langle \eta, \eta \rangle_B}{\omega} = S^{-1}$ and $\eta S \eta^* = S(v) - \frac{\omega}{2} J_2$.
\end{enumerate}
\end{theorem}
\begin{proof} We prove the statements in turn.  Consider the first claim.  Thus suppose $(S,\lambda)$ are as above, and $v + \eta$ is rank one in $W_J \oplus B^2$.  By equivariance, i.e., Proposition \ref{prop:Gaction}, we may assume $v=(1,0,c,d)$.  Then $v +\eta = (1,(0,-u),(c,v),d)$, and thus
\[(c,v) = (0,-u)^\# = (-uSu^*,\lambda^{-1}(u^*)^\#S^\#)\]
and $d = n((0,-u)) = -\tr_{K/F}(\lambda n(u))$.  Consequently,
\begin{align*} q(v) &= d^2 + 4n(c) = \tr_{K/F}(\lambda n(u))^2 - 4n(S)n(u)n(u^*) = \tr_{K/F}(\lambda n(u))^2 - 4n_{K/F}(\lambda n(u)) \\ &= \left(\lambda n(u) - (\lambda n(u))^*\right)^2.\end{align*}
This proves the first part.

Now consider the second claim.  Thus assume $q(v) = \omega^2$ with $\omega^* = -\omega$ in $K$.  Recall the notation $X(\pm \omega, v) = \frac{\pm \omega v + v^\flat}{2}$.  By the lifting law for $W_{B}$, there exist $\lambda \in K^\times$ and $\eta \in B^2$ so that $\lambda \eta^{!} = X(-\omega,v)$.  Thus
\[\omega q(v) = \langle X(\omega,v), X(-\omega,v) \rangle = \langle \lambda^*( \eta^*)^{!}, \lambda \eta^{!} \rangle = \lambda \lambda^* n(\langle \eta, \eta \rangle_B).\]
It follows that $n(\langle \eta, \eta \rangle_B/\omega) = (\lambda \lambda^*)^{-1}$, as claimed.  Thus, $\langle \eta, \eta \rangle_B/\omega = S^{-1}$, with $S \in J$ and $n(S) = \lambda \lambda^*$. 

We claim that with these definitions of $S$ and $\lambda$, $v+\eta$ is rank one in $W_{U(S,\lambda)}$.  To check this, we may use equivariance of these conditions under the action of $G \subseteq \GU(2,B) \subseteq \GL(2,B)$, the group of $(g,\nu)$ such that $\det_6(g) = \nu(g)^3$.  Indeed, if $g \in G$, set $v' = gv$, $\eta' = g \eta$, $\omega' = \nu(g) \omega$. Then $\lambda (\eta')^{!} = \frac{-\omega' v' + (v')^\flat}{2}$, $\langle \eta', \eta' \rangle_B/\omega' = S^{-1}$, and $v+\eta$ is rank one if and only if $v' + \eta'$ is rank one.  Thus we may assume $v = (1,0,c,d)$, so that $v+\eta = (1,(0,-u),(c,v),d)$.  We must check $-uSu^* = c$, $\lambda^{-1} (u^*)^\# S^\# = v$, and $\tr_{K/F}(\lambda n(u)) = -d$.

For $v = (1,0,c,d)$, one has $v^\flat = (-d,2c^\#, dc,d^2+2n(c))$.  Thus
\[(\lambda n(u), \lambda v u^\#, \lambda u v^\#, \lambda n(v)) = \lambda\eta^{!} = \frac{-\omega v + v^\flat}{2} = (-(d+\omega)/2,c^\#,(d-\omega)c/2, (-\omega d  + d^2 + 2n(c))/2).\]
Immediately, we obtain $\tr_{K/F}(\lambda n(u)) = -d$.

Now, since $\langle \eta, \eta, \rangle_B/\omega = S^{-1}$, 
\[-u \frac{(\langle \eta, \eta \rangle_B)^\#}{q(v)}u^* = -n(S)^{-1} u S u^*.\]
It is an identity in associative cubic norm structures that
\[ u(\langle \eta, \eta \rangle_B)^\# u^* = u(u^*v - v^* u)^\# u^* = n(u^*)uv^\# - (vu^\#) \times (vu^\#)^* + n(u) (uv^\#)^*.\]
Thus
\begin{align*} -uSu^* &= \frac{n(S)}{q(v)}\left(-u(\langle \eta, \eta\rangle_B)^\# u^*\right) \\ &= \frac{\lambda \lambda^*}{q(v)}\left( (vu^\#) \times (vu^\#)^* - n(u^*)uv^\# -n(u)(uv^\#)^*\right) \\ & = \frac{1}{q(v)}\left( c^\# \times c^\# + \left(\frac{d-\omega}{2}\right)^2 c + \left(\frac{d+\omega}{2}\right)^2 c\right) \\ &=\frac{1}{q(v)}\left(\frac{d^2 + 4n(c) + \omega^2}{2} \right) c\\ &= c,\end{align*}
as desired.

Next, we must check that $\lambda^{-1} (u^*)^\# S^\# = v$.  But we have
\[\lambda^{-1} (u^*)^\# S^\# = \lambda^* (u^*)^\# S^{-1} = \frac{(\lambda u^\#)^* (u^* v - v^* u)}{\omega} = \frac{(\lambda n(u))^* v - (\lambda v u^\#)^* u}{\omega}.\]
But $(\lambda vu^\#)^* = (c^\#)^* = c^\# = \lambda vu^\#$, and thus $(\lambda v u^\#)^* u =\lambda n(u) v$.  Thus,
\[\lambda^{-1} (u^*)^\# S^\# = \left(\frac{(\lambda n(u))^* - (\lambda n(u))}{\omega}\right) v = v.\]
Thus, we have checked that if $\lambda \eta^{!} = X(-\omega,v)$ and $\langle \eta, \eta \rangle_B/ \omega = S^{-1}$, then $v + \eta$ is rank one.

To finish the proof of the second claim, we must evaluate $\eta S \eta^* = \left(\begin{array}{cc} uSu^* & uSv^* \\ vSu^* & vSv^* \end{array}\right)$.  We have already checked $uSu^* = -c$.  Since we have already checked that $v + \eta$ is rank one, we obtain
\[(0,-du) = (c,v)^\# = (c^\# - vSv^*, -cv+ \lambda^{-1}(v^*)^\# S^\#).\]
Thus $vSv^* =c^\#$.  Finally, as above,
\begin{align*} q(v) uSv^* &= \lambda \lambda^* u (u^*v -v^*u)^\# v^* =\lambda \lambda^* \left( (uv^\#) (vu^\#)^* - 1 \times ((uv^\#)^*(vu^\#)) + n(u)n(v)^*\right) \\ &= \left(\frac{d-\omega}{2}\right)c c^\# - 1 \times \left(\left(\frac{d+\omega}{2}\right) c c^\#\right) - \left(\frac{d+\omega}{2}\right)\left(\frac{\omega d + d^2 + 2n(c)}{2}\right) \\ &= -\frac{d}{4}(d^2 + 4n(c) + \omega^2) - \frac{\omega}{2}(d^2 + 4n(c)).\end{align*}
Hence $uSv^* = -d/2 - \omega/2$.  Taking the conjugate, one obtains $vSu^* = -d/2 + \omega/2$.  Since for $v =(1,0,c,d)$ one has $S(v) = \left(\begin{array}{cc} -c & -d/2 \\ -d/2 & c^\#\end{array}\right)$, this proves $\eta S \eta^*= S(v) -\frac{\omega}{2} J_2$, completing the proof of the second claim.

We now consider the third claim.  Thus suppose $S \in J$, $\lambda \in K^\times$, $n(S) = \lambda \lambda^*$.  Furthermore suppose $v \in W_J$ with $q(v) =\omega^2$ and $v + \eta$ is rank one in $W_{U(S,\lambda)}$.  Consider $\tr_{K/F}(\lambda \eta^{!}) := (\lambda \eta^{!})^* + \lambda \eta^{!}$ in $W_J$.  We first claim that $\tr_{K/F}(\lambda \eta^{!})= v^\flat$.  

Note that $W_J \otimes K = W_J \oplus \omega W_J = W_{B}$, and the action of $G$ commutes with this decomposition.  Thus the statement $\tr_{K/F}(\lambda \eta^{!}) = v^\flat$ is invariant under the $G$ action, so we may as usual assume $v = (1,0,c,d)$.  Since $v + \eta$ is rank one, $(c,v) = (-uSu^*,\lambda^{-1} (u^*)^\# S^\#)$ and $d = -\tr_{K/F}(\lambda n(u))$.  One obtains from these that $\lambda vu^\# =c^\#$, $\lambda uv^\# =-(\lambda n(u))^* c$ and $\lambda n(v) = (\lambda^* n(u)^*)^2.$  Thus
\[\tr_{K/F}(\lambda \eta^{!}) = (-d,2c^\#,-dc, \tr_{K/F}((\lambda n(u))^2)).\]
Now for $x \in K$, $\tr(x^2) = (x+x^*)^2 - 2xx^*$, and thus
\[\tr_{K/F}((\lambda n(u))^2) = (-d)^2 - 2\lambda \lambda^* n(u)n(u^*) = d^2 + 2n(-uSu^*) = d^2 + 2n(c).\]
Hence $\tr_{K/F}(\lambda \eta^{!}) = v^\flat$, as desired.

We obtain that $\lambda \eta^{!} = \frac{-\omega x + v^\flat}{2}$ is rank one, for some $x \in W_J$.  That $x = \pm v$ follows from the essential uniqueness of rank one lifts, Lemma \ref{uniqueLifts}.  Namely, one applies Lemma \ref{uniqueLifts} to the element $2\omega \lambda \eta^{!}$, and uses that $(x^\flat)^\flat = - q(x)x$ for $x \in W_J$ and $q(x^\flat) = q(x)^3.$  Now, choose $\omega$ so that $\lambda \eta^{!} = \frac{-\omega v + v^\flat}{2}$.  We must evaluate $\frac{\langle \eta, \eta \rangle_B}{\omega}$.  Again, by equivariance, we may assume $v = (1,0,c,d)$.  But then
\[ u^*v - v^*u = \lambda^{-1}n(u^*)S^\# - (\lambda^*)^{-1}S^\# n(u) = (\lambda^* n(u^*) - \lambda n(u)) S^{-1}.\]
Since $\lambda \eta^{!} = \frac{-\omega v + v^\flat}{2}$, $(\lambda^* n(u^*) - \lambda n(u))  = \omega$, giving that $\frac{\langle \eta, \eta \rangle_B}{\omega} = S^{-1}$.

The final part of the third claim follows as the final part of the second claim.  This completes the proof of the theorem.
\end{proof}

\begin{remark} Suppose $A$ is a commutative associative cubic norm structure, as in the third item of Example \ref{ex:JBK}.  If $v \in W_{A}$ is rank four, then one can define $K = F[x]/(x^2 - q(v))$ and $B = A \otimes K$ as in this example.  Theorem \ref{thm:LLWJ2} then applies.  Thus in the case that $A$ is commutative, every rank four element of $W_{A}$ can be lifted to a rank one element of $W_{U}$, for an appropriate $U$.\end{remark}

\subsection{A more explicit construction of the lift} In Theorem \ref{thm:LLWJ2}, given $v \in W_J$ of rank four with $q(v)$ of a certain form, we produced a cubic norm structure $U = J \oplus B$ and a rank one lift $\tilde{v} = v + \eta \in W_{U} \simeq W_J \oplus B^2$.  To give this construction one must choose $\eta \in B^2$ and $\lambda \in K^\times$ so that $\lambda \eta^{!} = \frac{-\omega v + v^\flat}{2}$. The existence of such $\lambda, \eta$ are guaranteed by the lifting law Theorem \ref{b1A}.  In this subsection we explain how one can construct the cubic norm structure $U$ and the rank one lift $\tilde{v} \in W_{U}$ without making any choices.  However, the price that is paid to do this is to first construct a space $\widetilde{U} = J + B^2$ and a rank one $B$-module $I(v,\omega) \subseteq B^2 \subseteq \widetilde{U}$.  The cubic norm structure $U$ is then given as $\widetilde{U}/I(v,\omega)$.  The rank one lift in in $W_{U}$ can then be produced more canonically, as the image of an element from $F \oplus \widetilde{U} \oplus \widetilde{U} \oplus F$ in $W_U = F \oplus \left(\widetilde{U}/I(v,\omega)\right) \oplus \left(\widetilde{U}/I(v,\omega)\right) \oplus F$.

After we describe $\widetilde{U}$, $U = \widetilde{U}/I(v,\omega)$, and the rank one lift, we explain the relation of some of these constructions to elements of \cite{ganSavin}.  Throughout this subsection, $v = (a,b,c,d) \in W_J$ is rank four.

\subsubsection{Formulas defining the cubic norm structure on $\widetilde{U}/I(v,\omega)$} We now define $I(v,\omega) \subseteq B^2$ and the explicit formulas for the norm, adjoint, and pairing on the quotient $\widetilde{U}/I(v,\omega)$.

Set
\[h(v,\omega) := -\frac{\omega}{2} J_2 + S(v)\]
where recall $J_2 = \mm{}{1}{-1}{}$ and
\[S(v) = \left(\begin{array}{cc} b^\# - ac & ad -cb - \tr(ad -cb)/2 \\ ad -bc - \tr(ad-bc)/2 & c^\# -db \end{array}\right).\]
Then $h(v,\omega) \in M_2(B)$ satisfies $h(v,\omega)^* = h(v,\omega)$, where $*$ on $M_2(B)$ is the composite of matrix transpose and $*$ on $B$.  We define $Ad(\cdot; v,\omega): \widetilde{U} \rightarrow \widetilde{U}$ as follows.  For $x \in J$ and a row vector $\ell \in B^2$, set
\[Ad((x,\ell); v,\omega) := (x^\# - \ell h(v,\omega) \ell^*, - x\ell + \delta(\ell;v)^*J_2)\]
where here if $\ell = (u,v)$,
\[\delta(\ell; v) := \left(\begin{array}{c} au^\# + u \times (vb) + cv^\# \\ bu^\# + (uc) \times v + dv^\#\end{array}\right).\]

The norm on $\widetilde{U}$ is defined as
\[n((x,\ell); v, \omega) := n(x) - (x,\ell h(v,\omega) \ell^*) + \tr_{K/F}\left(\left\langle \frac{-\omega v + v^\flat}{2}, J_2^{-1} \ell^{!}\right\rangle\right).\]
Here, $\ell$ is a row vector, so if $\ell = (u,v)$, then $\ell^{!} = (n(u),u^{\#}v, v^\#u, n(v)).$  Note that $\ell^{!}$ and $-\omega v + v^\flat$ live in $W_B \supseteq W_J$, not $W_J$.

The submodule $I(v,\omega)$ is defined to be the set of $\ell \in B^2$ for which $\ell h(v,\omega) = 0$.  Equivalently, since the trace pairing on $B$ is non-degenerate, $I(v,\omega)$ is the set of $\ell$ for which $\tr_{B/K}(\ell h(v,\omega) (\ell')^*) = 0$ for all $\ell' \in B^2$.

A pairing $\widetilde{U} \otimes \widetilde{U} \rightarrow F$ is defined as follows.  Suppose $(x_1,\ell_1)$, $(x_2,\ell_2)$ are in $\widetilde{U}$, with $x_i \in J$ and $\ell_i \in B^2$.  Define
\[((x_1,\ell_1), (x_2,\ell_2);v,\omega) := (x_1,x_2) + \tr_{B/K}(\ell_1 h(v,\omega) \ell_2^* + \ell_2 h(v,\omega)\ell_1^*).\]
It is clear that this lands in $F$, and descends to $U(v,\omega) \otimes U(v,\omega)$.

Define $Ad(\cdot, \cdot; v,\omega): \widetilde{U} \otimes \widetilde{U}\rightarrow \widetilde{U}$ via $Ad(y,z;v,\omega) = Ad(y+z;v,\omega) - Ad(y;v,\omega) - Ad(z;v,\omega)$.  When $v$ and $\omega$ are fixed, we abbreviate $Ad(y,z;v,\omega)$ as $y \times z$.

The above formulas do not define a cubic norm structure on $\widetilde{U}$.  However, they do define a cubic norm structure on $U = \widetilde{U}/I(v,\omega)$, as stated in Theorem \ref{quarticUCL}.  Furthermore, these formulas have the advantage that they depend explicitly, and in fact equivariantly, on the element $v \in W_J$.  We first prove that $U$ is a cubic norm structure, by producing an isomorphism with the second Tits construction $U(S,\lambda)= J \oplus B$.  We then discuss the equivariance properties of the formulas and the rank one lift.

\begin{theorem}\label{quarticUCL} Suppose $v$ in $W_J$ is rank $4$, with $q(v) = \omega^2$, and $\omega \in K^\times$ with $\omega^* = -\omega$.  The submodule $I(v,\omega)$ satisfies $\widetilde{U} \times I(v,\omega)\subseteq I(v,\omega)$, and the norm map $n(\cdot; v,\omega)$ is well-defined on $U(v,\omega) = \widetilde{U}/I(v,\omega)$.  The induced adjoint $Ad(\cdot;v,\omega)$, norm, and pairing on $U(v,\omega) := \widetilde{U}/I(v,\omega)$ make it a cubic norm structure.  In fact, $U \simeq U(S,\lambda)$ for appropriate $S,\lambda$ in the second construction of Tits.\end{theorem}
\begin{proof} We explain the isomorphism $\widetilde{U}/I(\omega,v) \simeq U(S,\lambda)$, from which the other statements follow.  

Note that $h(v,\omega)J_2 = \frac{\omega + R(v)}{2}$, and this latter element of $M_2(B)$ occurs in the lifting law Theorem \ref{b1A}.  Thus for a column vector $\eta_{0}$ in $B^2$, by Theorem \ref{b1A} we have
\[ (h(v,\omega)J_2\eta_0)^{!} = \langle X(\omega,v), \eta_{0}^{!} \rangle X(-\omega,v).\]
Here recall that $X(\pm \omega, v) = \frac{\pm \omega v + v^\flat}{2}$.  Choose $\eta_0$ so that $\langle X(\omega,v), \eta_{0}^{!}\rangle \in K^\times$, and set $\eta = h(v,\omega)J_2\eta_0$.  Then with $\lambda = (\langle X(\omega,v), \eta_{0}^{!}\rangle)^{-1}$, one has $\lambda \eta^{!} = \frac{-\omega v+v^\flat}{2}$.  Define $S = \left(\frac{\langle \eta, \eta\rangle_{B}}{\omega}\right)^{-1}$ as in Theorem \ref{thm:LLWJ2}.

We now define a map $\tilde{U} \rightarrow U(S,\lambda)$ as 
\[(x,\ell) \mapsto (x, \ell \eta) = (x,\ell h(v,\omega) J\eta_{0}).\]
Note that $\ell \eta \in B$ since $\ell$ is a row vector and $\eta$ is a column vector.  Furthermore, it is clear that by the definition of $I(v,\omega)$, this map induces a map $\tilde{U}/I(v,\omega) \rightarrow U(S,\lambda)$.

We claim that $B^2/I(v,\omega) \rightarrow B$ via $\ell \mapsto \ell \eta$ is a left $B$-module isomorphism.  To see this, note that is injective, because if $\ell \eta = 0$, then $\ell \eta S \eta^* =0$.  But from Theorem \ref{thm:LLWJ2}, $\eta S \eta^* = h(v,\omega)$, so $\ell h(v,\omega) = 0$, so $\ell \in I(v,\omega)$.  To see that $B^2/I(v,\omega) \rightarrow B$ is surjective, note that since it is a $B$-module map, it suffices to check that the image contains a unit of $B$.  But if $\ell = \frac{1}{\omega} \eta^* J_2$, then $\ell \eta = \frac{1}{\omega}\langle \eta, \eta \rangle_{B} = S^{-1}$ is invertible.  Thus $\tilde{U}/I(v,\omega) \simeq U(S,\lambda)$ is an $F$-linear isomorphism.

Using the fact that $h(v,\omega) = \eta S \eta^*$ and $\frac{-\omega v + v^\flat}{2} = \lambda \eta^{!}$, it is easy to check that the norm and pairing defined above on $\tilde{U}$ factor through the map $\tilde{U}\rightarrow U(S,\lambda)$ and the norm and pairing on the $U(S,\lambda)$.  One can also show that the adjoint map $\#:\tilde{U} \rightarrow \tilde{U}$ commutes with the map $\tilde{U} \rightarrow U(S,\lambda)$.  The statements of the theorem now all follow.  
\end{proof}

As mentioned above, the formulas for the norm, adjoint, and pairing on $\tilde{U}$ have the advantage that they are explicitly defined in terms of $v \in W_J$, that they are suitably equivariant in $v$, and that they enable one to write down a rank one lift $\tilde{v} \in W_{U}$ somewhat more canonically than we could in Theorem \ref{thm:LLWJ2}.  We now address these equivariance and lifting claims.

\subsubsection{Equivariance of formulas on $\widetilde{U}$} For the equivariance, we have the following theorem.
\begin{theorem}\label{thm:equivU} Suppose $x, x_1, x_2 \in J$ and $\ell, \ell_1, \ell_2$ in $B^2$ are row vectors.  Furthermore, suppose $g \in G$.  Then
\begin{enumerate}
\item The norm is equivariant for the action of $G$: $n((x,\ell); gv,\nu(g) \omega) = n((x,\ell g); v,\omega)$.
\item The adjoint is equivariant for the action of $G$: $Ad((x,\ell); gv, \nu(g) \omega) = Ad((x,\ell g); v,\omega) g^{-1}$. The $g^{-1}$ on the right-hand side of this equality acts by multiplication on the $B^2$ component in $\widetilde{U}$ and acts trivially on the $J$ component.
\item The pairing is equivariant for the action of $G$: 
\[((x_1, \ell_1), (x_2, \ell_2); gv, \nu(g) \omega)= ((x_1, \ell_1 g), (x_2, \ell_2 g); v,\omega).\]
\item The $B$ submodule $I(v,\omega)$ is equivariant for the action of $G$: $I(gv,\nu(g) \omega) = I(v,\omega) g^{-1}$.
\end{enumerate}\end{theorem}
\begin{proof} Most of the statements of the theorem are straightforward checks.  We explain the proof of the most difficult step, which is the equivariance of the map $\delta(\ell; v)$, which is needed to prove that the adjoint map is suitably equivariant.  In fact, the map $\delta(\ell;v)$ satisfies an equivariance property for all of $\GL_2(B)$: $g \delta(\ell g; v) = \delta(\ell; g \cdot v)$ for $g \in \GL_2(B)$, where here $g \cdot v$ is the cubic polynomial action of $\GL_2(B)$ on $W_{B}$.  To check this equivariance statement, it suffices to check it on generators for $\GL_2(B)$, and the only nontrivial part is the equivariance for the action of the elements $n(X)$, $X \in B$.

Suppose $\ell = (u,v)$, and recall $v = (a,b,c,d) \in W_{B}$.  To check the equivariance for $n(X)$, one must verify the identity
\[ \left(\begin{array}{cc} 1 & \\ X & 1 \end{array}\right) \left(\begin{array}{c} a (u+ vX)^\# + (u + vX) \times (vb) + c v^\# \\ b(u+vX)^\# + ((u+vX)c) \times v + dv^\#\end{array}\right) = \left(\begin{array}{c} a' u^\# + u \times (v b') + c' v^\# \\ b' u^\# + (uc') \times v + d' v^\#\end{array}\right),\]
where $a' = a$, $b' = b + aX$, $c' = c + b\times X + a X^\#$, and $d' = d + (c,X) + (b, X^\#) + a n(X)$.  The top row of this equality is immediate, using the identity $(vX) \times (vb) = (X \times b)v^\#$.  For the second row, one must prove that \begin{itemize}
\item $b( u \times (vX)) + X(u \times (vb)) = \left( u(b \times X)\right) \times v$
\item $X (u \times (vX)) = (uX^\#) \times v$
\item $d' = d + (Xc) \times 1 + bX^\# + Xc'$.
\end{itemize}
This last identity follows from $(Xc) \times 1 = (c,X) -Xc$ and $X(b \times X) + bX^\# = (b,X^\#)$.  The second identity follows from pairing each side against an arbitrary $w \in B:$
\[(w, X(u \times (vX))) = (w X, u \times (vX)) = ((w X) \times (vX), u) = (X^\# (w \times v), u) = (w, (u X^\#) \times v).\]  
The first identity similarly follows by pairing against an arbitrary $w$ in $B$, while now using the following claim.
\begin{claim} If $b,X, v, w$ are in $B$, then $(b \times X)(v \times w) = (vb) \times (w X) + (vX) \times (w b)$. \end{claim}
\begin{proof} One obtains the claim by linearizing the identity $\left(v(b+X)\right) \times \left(w(b+X)\right) = (b+X)^\# (v \times w)$. \end{proof}
Applying the claim, one obtains
\begin{align*}(w, (u(b \times X)) \times v) &= ((b\times X)(w \times v), u) = ((vb) \times (w X) + (vX) \times (w b), u) \\ &= (w, b( u \times (vX)) + X(u \times (vb))),\end{align*}
finishing the proof.
\end{proof}

\subsubsection{The rank one lift} We now explain the lifting law from the point of view of $\widetilde{U}$.  Since $\widetilde{U} = J \oplus B^2$, one has a decomposition $F \oplus \widetilde{U} \oplus \widetilde{U} \oplus F = W_{J} \oplus M_2(B)$.  Explicitly, if $\ell_1, \ell_2 \in B^2$ are row vectors, $m = \left(\begin{array}{c} \ell_1 \\ \ell_2 \end{array}\right) \in M_2(B)$, and $v' = (a',b',c',d') \in W_J$, then we identify $v' + m \in W_J \oplus M_2(B)$ with the element $(a',(b',-\ell_1),(c',\ell_2),d')$ in $F \oplus \widetilde{U} \oplus \widetilde{U} \oplus F$.

We let $G$ act on $M_2(B)$ via left translation: $m \mapsto gm$ in $M_2(B)$.  Taking the quotient of $\widetilde{U}$ by $I(v,\omega)$, we obtain a map
\begin{equation}\label{eqn:LLmap} W_J \oplus M_2(B) \simeq F \oplus \widetilde{U} \oplus \widetilde{U} \oplus F \rightarrow F \oplus \left(\widetilde{U}/I(v,\omega)\right) \oplus \left(\widetilde{U}/I(v,\omega)\right) \oplus F \simeq W_{U}.\end{equation}
It is clear that this map is equivariant for the action of $G$ on both sides.  We now have the following result.
\begin{theorem}\label{quarticUCL2} Suppose $v$ in $W_J$ is rank $4$, with $q(v) = \omega^2$, and $\omega \in K^\times$ with $\omega^* = -\omega$.  Let $\widetilde{U}, I(v,\omega)$ and $U = \widetilde{U}/I(v,\omega)$ be as above.  Then the image of the element $\tilde{v} = v + 1_2$ in $W_U$ under the map (\ref{eqn:LLmap}) is rank one.  \end{theorem}
This theorem gives the somewhat more canonical rank one lift $\tilde{v} \in W_{U}$ of $v \in W_J$ promised above.
\begin{proof} The cubic norm structure $U = \widetilde{U}/I(v,\omega)$ was proved in Theorem \ref{quarticUCL} to be isomorphic to $U(S,\lambda)$, in the notation of that theorem.  The map constructed in Theorem \ref{quarticUCL} sends $x + \ell \in J \oplus B^2$ to $(x,\ell \eta)$, where $\eta \in B^2$ is a column vector with $\lambda \eta^{!} = \frac{-\omega v + v^\flat}{2}$.  Thus if $\eta = \left(\begin{array}{c} u\\ v\end{array}\right)$, the image of $v +1_2$ under the map (\ref{eqn:LLmap}), composed with the isomorphism $U \simeq U(S,\lambda)$, is $(a,(b,-u),(c,v),d) \in W_{U(S,\lambda)}$.  But this latter element was proved to be rank one in Theorem \ref{thm:LLWJ2}.  This completes the proof of the theorem.\end{proof}

\begin{remark} Denote $\omega_0$ the image of the element $(1,0) \in B^2$ in $B^2/I(v,\omega) \subseteq U$, and similarly set $\theta_0$ the image of the element $(0,1) \in B^2$ in $B^2/I(v,\omega) \subseteq U$.  Then Theorem \ref{quarticUCL2} says that $(a,b-\omega_0,c+\theta_0,d)$ is rank one in $W_U$.  This theorem generalizes the lifting law discussed in subsection \ref{subsec:LLbcfs}.\end{remark}

\subsubsection{Relation with work of Gan-Savin \cite{ganSavin}} We close this section by discussing the relation of the lifting laws in this section with elements of the paper \cite{ganSavin}. 

More specifically, consider the case when the pair $J$ and $B$ are as in the third item of Example \ref{ex:JBK}, so that $B = A \otimes_{F} K$, with $A$ a commutative associative cubic norm structure.  Then $A^2 \subseteq B^2$, and in fact the map $A^2 \rightarrow B^2/I(v,\omega)$ is an isomorphism of left $A$-modules.  Indeed, both sides have the same dimension as $F$-vector spaces, and it is injective because $A^2 \cap I(v,\omega) = 0$.  That $A^2 \cap I(v,\omega) = 0$ is immediate: If $ 0=\ell h(v,\omega) = \ell S(v) -\frac{\omega}{2} \ell J_2 $, then clearly $\ell = 0$, by looking at the coefficient of $\omega$.

It follows from this fact that in this setting, one can write the formulas for the cubic norm structure $U = J \oplus B^2/I(v,\omega) \simeq A \oplus A^2$ entirely in terms of $A$.  One obtains the following formulas.  Here $v = (a,b,c,d) \in W_{A}$ is rank four, $x, x_1, x_2 \in A$ and $\ell, \ell_1, \ell_2 \in A^2$ are row vectors:
\begin{enumerate}
\item The norm on $U$ is given by $n((x,\ell);v) = n(x) - (x,\ell S(v) \ell^{t}) + \langle v^\flat, J_2^{-1} \ell^{!}\rangle$.
\item The adjoint on $U$ is given by $Ad((x,\ell);v) = (x^\# - \ell S(v) \ell^{t}, -x\ell + \delta(\ell;v)^{t} J_2)$.
\item The pairing on $U$ is given by $((x_1,\ell_1),(x_2,\ell_2); v) = (x_1,x_2) + \tr_{A/F}(\ell_1 S(v) \ell_2^{t} + \ell_2 S(v) \ell_1^{t})$.
\end{enumerate}
Here if $\ell = (u,v)$,
\[\delta(\ell;v) = \left(\begin{array}{c} au^\# + u \times (vb) + cv^\# \\ bu^\# + (uc) \times v + dv^\# \end{array}\right).\]
Note that the $\omega$ drops out in all of these formulas.  These explicit formulas for the cubic norm structure $U$ associated to $v \in W_{A}$ of rank four are (in different notation) the ones given by Gan-Savin given in \cite[Section 10]{ganSavin}.

As in Theorem \ref{thm:equivU}, the above formulas for the cubic norm structure on $U$ are equivariant for the action of a certain subgroup $G \subseteq \GL_2(A)$.  Namely, denote $\det_2: \GL_2(A) \rightarrow \GL_1(A)$ the usual (degree two) determinant map, which we denote with a subscript $2$ to distinguish it from the degree $6$ map $\GL_2(A) \rightarrow \GL_1(F)$.  Then if $G = \{g \in \GL_2(A): \det_2(g) \in \GL_1(F)\}$, the above formulas defining the cubic norm structure on $U$ are equivariant for the action of $G$.  Here, the action of $G$ on $W_{A}$ is the $\det_2^{-1}$ twist of the cubic polynomial action.

Finally, we restate the lifting law in this context, which is implicit in \cite{ganSavin}.  Define a map $W_{A} \oplus M_2(A) \rightarrow W_{U}$ by 
\[v + m = (a',b',c',d') + \left(\begin{array}{c} \ell_1 \\ \ell_2\end{array}\right) \mapsto (a',(b',-\ell_1),(c',\ell_2),d') \in W_{U}.\]
Then $G \subseteq \GL_2(A)$ acts on $M_2(A)$ by left translation, $m \mapsto gm$, and by the action on $W_{A}$ just described.  This defines an action of $G$ on $W_{U}$, and one has the following result.  As above, denote by $\omega_0 = (1,0) \in A^2 \subseteq U$ and $\theta_0 = (0,1) \in A^2 \subseteq U$.
\begin{theorem}[Gan-Savin]\label{ganSavinlift} Suppose $v\in W_A$ is rank $4$.  With the formulas defined above, $U$ is a cubic norm structure.  The action of $G$ on $W_{U}$ preserves the symplectic and quartic form on $W_{U}$ up to similitude, thus giving a map $G \rightarrow H(W_{U})$, preserving similitudes.  The lifted element
\[v + 1_2 = (a,b,c,d) + (0,-\omega_0, \theta_0,0) =: (a,-\omega,\theta,d) \in W_{U}\]
is rank one.\end{theorem}

\section{Lower rank lifting laws}\label{lowerRank}
In this section we give a few examples of lower rank lifting laws.  That is, we show how to lift elements of non-maximal rank in certain prehomogeneous spaces, i.e., elements not in the open orbit, to rank one elements of bigger prehomogeneous vector spaces. More precisely, we use the Cayley-Dickson construction to lift elements of $H_3(C)$ of rank two, for $C$ an associative composition algebra, to rank one elements of $H_3(D)$, with $D$ a composition algebra produced from $C$ via the Cayley-Dickson construction.  We again use the Cayley-Dickson construction to lift rank two elements of $W_{H_3(C)}$ to rank one elements of $W_{H_3(D)}$.  Finally, we lift rank $3$ elements of $W_J$ to rank one elements of $W_{J \oplus B}$ by using the second Tits construction.

\subsection{The Cayley-Dickson construction and rank two elements}\label{lowRank} Throughout this subsection, $F$ is the ground field, and $C$ is an associative composition $F$-algebra.  In this subsection, we use the Cayley-Dickson construction to lift rank two elements of $H_3(C)$ to rank one elements of $H_3(D)$, and rank two elements of $W_{H_3(C)}$ to rank one elements of $W_{H_3(D)}$.  Here $D = C \oplus C$ is a composition algebra obtained from $C$ by the Cayley-Dickson construction.

Before stating the results, we recall the Cayley-Dickson construction.  Suppose $C$ is an associative composition algebra over $F$, and $\gamma \in \GL_1(F) = F^\times$.  Define $D= C(\gamma)$ to be $C^2$ with addition defined component-wise and multiplication given by
\[(x_1,y_1) \cdot (x_2, y_2) = (x_1 x_2 +\gamma y_2^* y_1, y_2 x_1 + y_1 x_2^*).\]
The conjugation $*$ on $C(\gamma)$ is $(x,y)^* = (x^*,-y)$ and the norm is $n((x,y)) = n(x) - \gamma n(y)$.  With these definitions, $D$ is a composition $F$-algebra.

\subsubsection{Rank two elements of $H_3(C)$} Denote by $V_3$ the vector space $F^3$, considered as row-vectors, i.e., the defining (right) representation of $\GL_3$. For $\gamma \in F^\times$, define $U(\gamma) = H_3(C) \oplus V_3(C)$, with norm, adjoint, and pairing given as follows: If $X \in H_3(C)$ and $v \in V_3(C)$ a row vector, then
\begin{enumerate}
\item $n((X,v)) = n(X) + \gamma v Xv^*$;
\item $(X,v)^\# = (X^\# + \gamma v^* v, -vX)$;
\item $\langle (X,v) , (Y,w) \rangle = (X,Y) - \gamma (v,w)$ where $(v,w) = v w^* + w v^* = (v_1,w_1) + (v_2, w_2) + (v_3, w_3)$.\end{enumerate}
Via these definitions, $U(\gamma)$ becomes a cubic norm structure.  In fact, $U(\gamma) \simeq H_3(C(\gamma))$, where $C(\gamma)$ is the composition algebra formed from  $C$ via the Cayley-Dickson construction.

To see that $U(\gamma) \simeq H_3(C(\gamma))$, define an $F$-linear map $\phi: U(\gamma) \rightarrow H_3(C(\gamma))$ via
\[\phi((X,v)) = \left(\begin{array}{ccc} c_1 & a_3 & a_2^* \\ a_3^* & c_2 & a_1 \\ a_2 & a_1^* & c_3 \end{array}\right)\]
where $c_i = X_{ii}$ and $a_i = (a_i(X), v_i) = (X_{i+1,i+2},v_i)$.  Here the indices are taken modulo $3$.
\begin{lemma} The map $\phi$ is an isomorphism of cubic norm structures. In particular, $U(\gamma)$ is a cubic norm structure.\end{lemma}
\begin{proof} The proof is a straightforward computation.  We remark that it also is easy to verify directly that $U(\gamma)$ is a cubic norm structure. \end{proof}

Here is the first lower-rank lifting law.
\begin{theorem} Suppose $X \in H_3(C)$, and there exists $\gamma \in F^\times$ and $v \in V_3(C)$ so that $(X,v)$ is rank one in $U(\gamma)$.  Then, $n(X) = 0$, so that $X$ has rank at most $2$.  If $X$ has rank exactly two, then $(X,v)$ is rank one in $U(\gamma)$ if and only if $X^\# = -\gamma v^* v$, and given $X \in H_3(C)$ of rank two, the set of such $\gamma \in F^\times$ and  $v \in V_3(C)$ so that $X = -\gamma v^* v$ is non-empty.  Finally, the class of $\gamma \in F^\times/n(C^\times)$ is uniquely determined.\end{theorem}
\begin{proof} Suppose one has $\gamma \in F^\times$ and $v \in V_3(C)$ so that $(X,v)$ is rank one in $U(\gamma)$.  Then $X^\# = -\gamma v^* v$, and $Xv = 0$.  Since $X^\#$ is then rank one, we conclude $n(X) = 0$.  

Now suppose $X$ has rank exactly two.  Then $X^\#$ has rank one, and thus we know from Lemma \ref{lem:rk1Field} that there exists a $v \in V_3(C)$ and a gamma in $F^\times$ so that $-X^\# =\gamma v^*v$, and that furthermore this $\gamma$ is uniquely determined in $F^\times/n(C^\times)$.  

We shall require the following lemma.
\begin{lemma}\label{wPrim} Suppose $Y \in H_3(C)$, $Y \neq 0$, and $Y = \gamma w^* w$ for some $\gamma \in F^\times$ and $w \in V_3(C)$.  Then $w$ is primitive. \end{lemma}
\begin{proof} By $\GL_3(C)$-equivariance, we may assume 
\[Y = \left(\begin{array}{ccc} 1 &0 &0\\ 0& 0& 0\\ 0&0&0\end{array}\right).\]
Write $w = (w_1, w_2, w_3)$.  Then from the equation $Y = \gamma w^* w$, we deduce $n(w_1) \in F^\times$, and hence $w$ is primitive. \end{proof}

Now suppose $X^\# = -\gamma v^* v$.  By the lemma, $v$ is primitive.  But we get
\[0 = n(X) = X^\# X = -\gamma v^* (vX).\]
Since $v$ is primitive, $v^* (vX) = 0$ implies $vX = 0$.  Hence $vX = 0$, and so $(X,v)^\# = 0$.  This completes the proof. \end{proof}

\subsubsection{Rank two elements of $W_{H_3(C)}$} In this subsection we consider the lifting of rank two elements of $W_{H_3(C)}$ to rank one elements of $W_{H_3(C(\gamma))} = W_{U(\gamma)}$, with notation as in the previous subsubsection.  Before giving the lifting law, we make a definition.  If $x=(a,b,c,d) \in W_{H_3(C)}$, then we define
\[S(x) = \left(\begin{array}{cc} b^\#-ac & ad-cb - \tr(ad-cb)/2 \\ ad -bc -\tr(ad-bc)/2 & c^\#-bd\end{array}\right).\]
Here the products $bc$ and $cb$ are taken in $M_3(C) \supseteq H_3(C)$.  Thus, $S(x)$ is an element of $M_6(C)$.  This is an identical definition to that given in (\ref{eqn:Sdef}), but we restate it since $H_3(C)$ is not an associative algebra.

Denote by $W_6 = V_3 \oplus V_3$ the defining right representation of $\GSp_6$ on row vectors, and $W_6(C) = W_6 \otimes C$.  Put on $W_6(C)$ the $C$-valued symplectic-Hermitian form $\langle u, u'\rangle_C = u J_6 (u')^*$, where $J_6 = \mm{}{1_3}{-1_3}{}$.  Thus if $v,w,v',w' \in V_3(C)$, $u = (v,w)$ and $u'=(v',w')$, then $\langle u, u'\rangle_C = v(w')^*-w(v')^* \in C$.

For $\gamma \in F^\times$, we make an identification $W_{H_3(C)} \oplus W_6(C) \simeq W_{U(\gamma)}$ via
\[(a,b,c,d) + (v,w) \mapsto (a,(b,-v),(c,w),d).\]
Here $v, w \in V_3(C)$.  We have the following lifting law.
\begin{theorem}\label{thm:rk2WLift} Suppose $x = (a,b,c,d) \in W_{H_3(C)}$.  If there exists $\gamma \in F^\times$ and $u = (v,w) \in W_6(C)$ so that $x + u$ is rank one in $W_{U(\gamma)}$, then $x^\flat = 0$, so that $x$ has rank at most $2$.  For $\gamma \in F^\times$, define the set $\mathrm{Lift}(x,\gamma)$ as
\[\mathrm{Lift}(x,\gamma) := \{ u \in W_6(C): \langle u, u \rangle_C = 0, - S(x) = \gamma u^* u\}.\]
For $x \in W_{H_3(C)}$ of rank exactly two, one has $x + u$ is rank one in $W_{U(\gamma)}$ if and only if $u \in \mathrm{Lift}(x,\gamma)$.  Furthermore, given $x$ of rank two, there exists $\gamma \in F^\times$ so that $\mathrm{Lift}(x,\gamma)$ is non-empty, and the class of such $\gamma$ is uniquely determined in $F^\times/n(C^\times)$. \end{theorem}

To prove this result, we will use the following proposition.  Define $\GU_6(C) =\{g \in \GL_6(C): gJ_6 g^* = \nu(g)J_6\}$, i.e., the group preserving the symplectic-Hermitian form $\langle \cdot,\cdot \rangle_C$ up to similitude.  Consider the group $H(W_{H_3(C)}) \times \GU_6(C)$.  This group acts on $W_{U(\gamma)} = W_{H_3(C)} \oplus W_6(C)$ by letting $H(W_{H_3(C)})$ act on $W_{H_3(C)}$ and $\GU_6(C)$ on $W_6(C)$; we use right actions for both of these groups.  Define the group $G(\gamma,C)$ to be the subgroup of $(h,g) \in H(W_{H_3(C)}) \times \GU_6(C)$ with $\nu(h) = \nu(g)$ and such that $(h,g)$ preserves the symplectic and quartic form on $W_{U(\gamma)}$ up to this common similitude.  Thus, $G(\gamma,C)$ maps to $H(W_{U(\gamma)})$ by definition.
\begin{proposition}\label{prop:GCequiv} Denote by $\overline{F}$ an algebraic closure of $F$.  The projections $G(\gamma,C) \rightarrow H(W_{H_3(C)})$ and $G(\gamma,C) \rightarrow \GU_6(C)$ are surjections on $\overline{F}$-points, whose kernels are $1 \times \mu_2 \subseteq H(W_{H_3(C)}) \times \GU_6(C)$ and $\mu_2\times 1 \subseteq H(W_{H_3(C)}) \times \GU_6(C)$, respectively.  Furthermore, if $(h,g) \in G(\gamma,C)$, then $S(x h) = g^*S(x)g$ for all $x \in W_{H_3(C)}$.\end{proposition}
\begin{proof} This proposition was essentially proved in \cite[Appendix A]{pollack}.  More specifically, when $C$ is a quaternion algebra, all but the last statement was proved in \emph{loc. cit.}  For a general associative composition algebra, the surjectivity and kernel statements are proved identically.  We take this opportunity, however, to correct a typo in \cite{pollack}: In equation (A.4) of \emph{loc. cit.}, the term $\tr(b \times B, c \times C)$ should be multiplied by $2$.  I.e., (A.4) should read
\[\tr(B,C)(ad-(b,c))+2\tr(b\times B,c \times C) + 2\tr(B^\#,c^\#-bd) + 2\tr(C^\#,b^\#-ac).\]
(This has no affect on the arguments of \cite[Appendix A]{pollack}.)

The claim $S(xh) = g^*S(x)g$ is essentially the content of Proposition \ref{Requiv} or Lemma \ref{lem:Requiv}, but we state it again because we are now working in the context of the special cubic norm structure $H_3(C)$ instead of an associative cubic norm structure $A$, as was the case in Proposition \ref{Requiv} and Lemma \ref{lem:Requiv}. For the claim that $S(x h) = g^*S(x)g$, it suffices to prove this over the algebraic closure.  To do so, one first checks this claim directly for the elements called $n_{\Theta}(X), J_{\Theta}, m_{\Theta}(\lambda), z_{\Theta}(\lambda)$ and $(M(m,\delta),\mm{m}{}{}{\,^*m^{-1}})$ on page 1428 of \emph{loc. cit.}.  Assuming this for the moment, then for a general $(h,g) \in G(\gamma,C)$, the equivariance $S(x h) = g^*S(x)g$ follows from the same fact for these specific elements, plus the statement about $\mu_2$-kernels.

Checking that $S(xh) = g^*S(x)g$ for the elements $J_{\Theta}, m_{\Theta}(\lambda), z_{\Theta}(\lambda)$ and $(M(m,\delta),\mm{m}{}{}{\,^*m^{-1}})$ is immediate from their definitions.  The only nontrivial check that must be performed is for the element $n_{\Theta}(X)$.  To do so, one must verify the following two identities.  Define $(a',b',c',d')$ via
\[(a',b',c',d') = (a,b+aX,c+b\times X + a X^\#,d+(c,X)+(b,X^\#)+an(X)).\]
Then we must check
\begin{equation}\label{eqn:n(X)21} a'd'-b'c' -\tr(a'd'-b'c')/2 = ad-bc -\tr(ad-bc)/2 + X(b^\#-ac)\end{equation}
and 
\begin{equation}\label{eqn:n(X)22} (c')^\#-d'b' = c^\#-db + X(ad-cb -\tr(ad-cb)/2) + (ad-bc-\tr(ad-bc)/2)X + X(b^\#-ac)X.\end{equation}
The proof of (\ref{eqn:n(X)21}) reduces to the identity $(b,X^\#) = bX^\# + X(X \times b)$, valid in special cubic norm structures.  The proof of (\ref{eqn:n(X)22}) follows from $c \times (b \times X) = (c,X)b + (c,b)X - (bcX+Xcb)$, again valid in special cubic norm structures, and other similar identities. 
\end{proof} 

\begin{proof}[Proof of Theorem \ref{thm:rk2WLift}]  By proposition \ref{prop:GCequiv}, we may assume $x = (1,0,c,d)$.  With $x$ in this form, we find $x + u = (1, (0,-v), (c,w),d)$, and this is rank one if and only if
\[ (c,w) = (0,-v)^\# = (\gamma v^*v, 0)\]
and $d = n((0,-v)) = 0$.  Hence $c^\# = 0$, and $d = 0$, and thus
\begin{equation}\label{xflat}x^\flat = (-d,2c^\#,dc,d^2 + 2n(c)) = 0.\end{equation}
Thus if $x$ has a rank one lift for some $\gamma$, then $x$ has rank at most $2$.

Now suppose $x$ has rank exactly two.  Then by (\ref{xflat}), we get $d = 0$, $c^\# = 0$, and thus $S(x) = \mm{-c}{-d/2}{-d/2}{c^\#} = \mm{-c}{0}{0}{0}$.  Furthermore, for $u = (v,w)$, we see that $x + u$ is rank one in $W_{U(\gamma)}$ if and only if $c = \gamma v^* v$ and $w = 0$.  In particular $u \in \mathrm{Lift}(x,\gamma)$.  Since $c^\# = 0$, there exists $\gamma \in F^\times$ and $v \in V_3(C)$ so that $c = \gamma v^*v$, and hence for this $\gamma$, $\mathrm{Lift}(x,\gamma)$ is non-empty.  Conversely, suppose $u =(v,w) \in \mathrm{Lift}(x,\gamma)$.  Then since $S(x) = -\gamma u^*u$, $c = \gamma v^* v$ and $v^* w = 0$.  Since $x$ has rank exactly two, $c \neq 0$.  Thus by Lemma \ref{wPrim}, $v$ is primitive and hence since $v^* w = 0$, $w = 0$.  Hence $u \in \mathrm{Lift}(x,\gamma)$ implies $x + u$ is rank one in $W_{U(\gamma)}$.

Finally, we must check that for $x$ of rank two, $\gamma$ is uniquely determined in $F^\times/n(C^\times)$.  Thus suppose $x+u_1$ is rank one in $W_{U(\gamma_1)}$, and $x+ u_2$ is rank one in $W_{U(\gamma_2)}$.  Then $\gamma_1 u_1^* u_1 = -S(x) = \gamma_2 u_2^* u_2$.  By multiplying these quantities by $u$ on the left and $u^*$ on the right, one obtains $\gamma_1 n(u_1 u^*) = \gamma_2 n(u_2 u^*)$ for all $u \in W_6(C)$.  Since $u$ can be chosen so that $n(u_1 u^*) \neq 0$ (as follows from equivariance and the calculation of the previous paragraph), we deduce that $\gamma_1$ and $\gamma_2$ represent the same class in $F^\times/n(C^\times)$.  This completes the proof.\end{proof}

\subsection{Rank three elements of $W_{J}$ and the Tits construction} In this subsection we assume $J, B$ and $K$ are as in section \ref{nccubes}.  That is, suppose $K$ is a quadratic \'{e}tale extension of $F$, $B$ is a cubic associative algebra over $K$ with an involution of the second kind $*$ compatible with the cubic norm, and $J = B^{*=1}$.  See the very beginning of section \ref{nccubes} for details.  Given $h \in J$ of rank three, a special case of the second construction of Tits can be used to make $U(h) = J \oplus B$ into a cubic norm structure.  In this subsection, we discuss how this construction can be used to lift rank $3$ elements of $W_J$ to rank one elements of $W_{U(h)}$.

First, let us recall this construction of Tits, in a somewhat special case.  See subsection \ref{subsec:secondTits} for the general construction. One defines a norm, adjoint, and pairing on $U(h) = J \oplus B$ as follows: For $X,Y \in J$ and $\alpha, \beta \in B$,
\begin{itemize}
\item $n((X,\alpha)) = n(X) - (X,\alpha h^\#\alpha^*) + n(h)\tr_{K/F}(n(\alpha))$
\item $((X,\alpha))^\# = (X^\# - \alpha h^\# \alpha^*, -X\alpha + (\alpha^*)^\# h)$
\item $((X,\alpha),(Y,\beta)) = (X,Y) + \tr_{J}(\alpha h^\# \beta^* + \beta h^\# \alpha^*).$
\end{itemize}
It is a fact that these formulas make $U(h)$ into a cubic norm structure\footnote{One can even consider this construction itself as a lifting law.  We leave the formulation of such a lifting law to the interested reader.}.  In the notation of subsection \ref{subsec:secondTits}, this is the special case where $(S,\lambda) = (h^\#,n(h))$.

Now, recall the group $G \subseteq \GU_2(B)$ from subsection \ref{subsec:GBgrp}.  Namely, $G$ is the subgroup of $\GL_2(B)$ for which $g^* \mm{}{1}{-1}{}g = \nu(g) \mm{}{1}{-1}{}$ and $\det_6(g) = \nu(g)^3$.  The group $G$ acts on the left of $W_J$ by the $\nu^{-1}$-twist of the restriction of the cubic polynomial action of $\GL_2(B)$ on $W_B$; see subsection \ref{subsec:GBgrp}.  Denote by $W_2(B) = B^2$ the $2 \times 1$ column vectors with coefficients in $B$.  Then we identify $W_J \oplus W_2(B)$ with $W_{U(h)}$ via $x  + \eta \mapsto (a,(b,-u),(c,v),d)$, where $x = (a,b,c,d)$ and $\eta = \left(\begin{array}{c} u \\ v \end{array}\right)$. Via this identification, $G$ acts on the left of $W_{U(h)}$, and in fact one obtains a map $G \rightarrow H(W_{U(h)})$.  That is, $G$ preserves the symplectic and quartic form on $W_{U(h)}$, up to similitude.  This fact was stated in Proposition \ref{prop:Gaction}.

Recall that we have a $B$-valued symplectic-Hermitian form on $W_2(B)$ defined by $\langle \eta, \eta'\rangle_B = \eta^* J_2 \eta'$, where $J_2 = \mm{}{1}{-1}{}$.  A vector $\eta \in W_2(B)$ is said to be \emph{isotropic} if $\langle \eta, \eta \rangle_B = 0$. Here is the lifting law.
\begin{theorem} Suppose $x = (a,b,c,d) \in W_J$.  If there exists a rank $3$ $h$ in $J$ and an isotropic $\eta = \left(\begin{array}{c} u \\ v \end{array}\right)$ in $W_2(B)$ with $x + \eta$ rank one in $W_{U(h)}$, then $q(x) = 0$, so that $x$ has rank at most $3$.  Given $x \in W_J$, and $h \in J$ of rank $3$, define
\[\mathrm{Lift}(x,h) =\{\eta \in W_2(B): \langle \eta, \eta \rangle_{B} = 0, S(x) =  \eta h^\# \eta^*, \frac{x^\flat}{2} = n(h)\eta^{!}\}.\]
Suppose $x$ has rank $3$.  Then if $\eta \in W_2(B)$ is isotropic, $x + \eta$ is rank one in $W_{U(h)}$ if and only if $\eta \in \mathrm{Lift}(x,h)$.  Furthermore, for $x$ of rank $3$, there exists $h \in J$ with $n(h) \neq 0$ so that $\mathrm{Lift}(x,h)$ is non-empty. \end{theorem}
\begin{proof} Suppose $x + \eta$ is rank one in $W_{U(h)}$ for some isotropic $\eta$.  Applying Lemma \ref{lem:Requiv}, one sees that the conditions defining $\mathrm{Lift}(x,h)$ are equivariant for the action of $G$.  Thus by this equivariance, we may assume $x = (1,0,c,d)$, so $x + \eta = (1,(0,-u),(c,v),d)$.  Since this is rank one, $(c,v) = (0,-u)^\# = (-uh^\#u^*,(u^*)^\# h)$ and $d = n((0,-u)) = -n(h)\tr_{K/F}(n(u))$.  Since $\eta$ is istropic, $u^* v = n(u)^* h$ is Hermitian, and thus $n(u) = n(u)^*$, so $d = -2n(h)n(u)$.  Hence 
\[q(x) = d^2 + 4n(c) = (-2n(h)n(u))^2 + 4n(-uh^\#u^*) = 0\]
since $n(u) = n(u)^*$.  Thus $x$ has rank at most $3$.

Suppose now that $x$ has rank exactly $3$.  To check that $x + \eta$ is rank one in $W_{U(h)}$ for an isotropic $\eta$ if and only if $\eta \in \mathrm{Lift}(x,h)$, by equivariance for the action of $G$, it suffices to consider the case $x = (1,0,c,d)$.  Then recall that $x^\flat = (-d,2c^\#,dc,d^2+2n(c))$ and $S(x) = \mm{-c}{-d/2}{-d/2}{c^\#}$.  Then as above, if $\eta$ is isotropic, $x + \eta = (1,(0,-u),(c,v),d)$ is rank one if and only if $n(u) \in F$, $(c,v) = (-uh^\#u^*,(u^*)^\# h)$, and $d = - 2n(u)n(h)$.

We check that these conditions are equivalent to $\eta$ being in $\mathrm{Lift}(x,h)$.  First suppose $\eta$ isotropic and $x + \eta$ is rank one.  Then $-d/2 = n(h)n(u)$, 
\[c^\# = (-uh^\#u^*)^\# = n(h)(u^*)^\#h u^\# = n(h)vu^\#,\]
\[dc/2 = -n(h)n(u)(-uh^\#(u^*)^\#) = n(h)u (n(u) h^\# (u^*)^\#) = n(h) u v^\#,\] and
\[d^2/2 + n(c) = 2n(h)^2n(u)^2 - n(u)^2n(h)^2 = n(h)n(u^2)n(h) = n(h) n(v).\]
Thus $x^\flat/2 = n(h) \eta^{!}$.  Furthermore
\begin{align*} \eta h^\# \eta^* &=  \left(\begin{array}{c} u\\v \end{array}\right)h^\# \left(\begin{array}{cc} u^* & v^* \end{array}\right) = \left(\begin{array}{cc} u h^\# u^* & uh^\# v^* \\ v h^\# u^* & vh^\# v^* \end{array}\right) \\ &= \left(\begin{array}{cc} uh^\#u^* & n(h)n(u) \\ n(h)n(u) & n(h)(u^*)^\# h u^\# \end{array}\right) = \left(\begin{array}{cc} -c & -d/2 \\ -d/2 & c^\# \end{array}\right)\\ &= S(x).\end{align*}

Conversely, suppose $x$ is rank $3$, $n(h) \neq 0$, and $\eta \in \mathrm{Lift}(x,h)$.  Then immediately $c = - uh^\# u^*$ and $d = - 2n(h)n(u)$, and hence $n(u) \in F$, and thus $d = n(h)\tr_{K/F}(-n(u))$.  We wish to show $v = (u^*)^\# h$.  If $d \neq 0$, this is easy.  For the general case, however, we argue as follows.  One has 
\[(v - (u^*)^\#h) h^\# u^* = vh^\#u^* - n(h)n(u) = (-d/2) - (-d/2) = 0\]
and
\[(v - (u^*)^\#h) h^\# v^* = vh^\# v^* - n(h)(vu^\#)^* = c^\# - c^\# = 0.\]
Hence $(v - (u^*)^\#h)h^\# \eta^* = 0$.

Now, since $x$ has rank exactly $3$, $n(h) \eta^{!} = x^\flat/2 \neq 0$.  Hence $\eta$ is primitive, and so $(v - (u^*)^\#h)h^\# \eta^* = 0$ implies $(v - (u^*)^\#h)h^\# = 0$, and thus $v = (u^*)^\#h$, since $n(h) \neq 0$. 

Finally, it follows from equivariance and Lemma \ref{lem:liftJexist} below that there exists $h \in J$ of rank three for which $\mathrm{Lift}(x,h)$ is non-empty. This completes the proof.  \end{proof}

\begin{lemma}\label{lem:liftJexist} Given $x = (1,0,c,d)$ in $W_J$ of rank three, there exists $h \in J$ with $n(h) \neq 0$ and $u,v \in B$ so that $(1,(0,-u),(c,v),d) \in W_{U(h)}$ is rank one. Furthermore, the element $\eta = (u,v)$ of $W_2(B)$ can be chosen to be isotropic.\end{lemma}
\begin{proof} Since $x$ is rank three, $q(x) = d^2 + 4n(c) =0$.  We break the proof up into two cases, according as to whether $d=0$ or not.  

Suppose first that $d \neq 0$.  It follows that $n(c) \neq 0$.  In this case, set $h = -2 d^{-1} c^\#$, $u = 1$ and $v = h$.  Then $(0,-1)^\# = (-h^\#,h) = (-(2/d)^2n(c) c, v) = (c,v)$ and $n((0,-1)) = -2n(h) = 2(2/d)^3 n(c)^2 = d$.  Thus, $x +\eta = (1,(0,-1),(c,v),d)$ is rank one in $W_{U(h)}$.  It is clear that $\eta = (1,h)$ is isotropic, and thus this completes the proof of the lemma when $d \neq 0$.

Now suppose that $d=0$.  Then $n(c) = 0$, but $c^\# \neq 0$ since $x^\flat = (-d,2c^\#,dc,d^2+2n(c)) \neq 0$  by assumption.  Hence $c$ is rank two, and $c^\#$ is rank one.  We claim that we may reduce to the case that $\tr(c^\#) \neq 0$.  Suppose for now that $c$ is such that $\tr(c^\#) \neq 0$.  Set $S = c - c^\#$.  Note that $n(S) = -\tr((c^\#)^2) = -\tr(c^\#)^2$, since $c^\#$ is rank one.  Furthermore, if $u = 1 - \frac{c^\#}{\tr(c^\#)}$, then
\[uSu = \left(1 - \frac{c^\#}{\tr(c^\#)}\right)(c-c^\#)\left(1 - \frac{c^\#}{\tr(c^\#)}\right) = c \left(1 - \frac{c^\#}{\tr(c^\#)}\right) = c,\]
where we have used the identity $(c^\#)^2 = \tr(c^\#)c^\#$, since $c^\#$ is rank one.  Now, set $h = (\tr(c^\#))^{-1} S^\#$ and $v = (u^*)^\# h$.  Since $n(u) \in F$ (in fact, $n(u) = 0$), $u^*v = n(u)h$ is in $J$, and thus $\eta = (u,v)$ is isotropic.  Furthermore, $-h^\# = S$ and thus $-uh^\# u^* = -uh^\# u = c$.  Since $n(u) = 0$, it now follows that $x + \eta = (1,(0,-u),(c,v),0)$ is rank one in $W_{U(h)}$.  This completes the proof of the lemma in case $n(c) = 0$ but $\tr(c^\#) \neq 0$.

Finally, we explain the reduction to the case $\tr(c^\#) \neq 0$.  This follows by equivariance for the action of the group $G$.  Indeed, if $y \in J$ with $n(y) \neq 0$, then the element $m =\mm{y^{-1}}{}{}{y}$ of $\GL_2(B)$ is in $G$, and $m x = (n(y)^{-1},0,y^{-1}cy^\#,0)$.   Furthermore, $\tr((y^{-1}cy^\#)^\#) = \tr(yc^\# y) = (y^2,c^\#)$.  Since $c^\# \neq 0$, it follows by a Zariski density argument that there exists $y \in J$ with $n(y) \neq 0$ and $(y^2,c^\#) \neq 0$.  Scaling $mx$ by $n(y)$ gives an element $(1,0,c',0)$ in the $G$ orbit of $x$ with $\tr((c')^\#) \neq 0$.  This completes the proof of the lemma.\end{proof}

\bibliography{liftingLaws_Bib} 
\bibliographystyle{amsalpha}
\end{document}